\documentclass[12pt]{amsart}
\usepackage{amssymb}
\usepackage{enumitem}
\usepackage{graphicx}
\usepackage[cmtip,all]{xy}
\usepackage{bbm}
\usepackage{comment}
\usepackage{mathrsfs}
\usepackage{mathtools}
\usepackage{xcolor}
\usepackage[dvipdfmx]{hyperref}
\usepackage{autobreak}
\hypersetup{
 colorlinks,
 linkcolor={blue!50!black},
 citecolor={blue!50!black},
 urlcolor={blue!80!black}
}

\everymath{\displaystyle}

\calclayout

\def\DefineSymbol#1#2{\newcommand{#1}{{\mathrm {#2}}}}
\def\DefineCategory#1#2{\newcommand{#1}{{\mathrm {#2}}}}

\theoremstyle{plain}
	\newtheorem{theorem}{Theorem}[section]
	\newtheorem{lemma}[theorem]{Lemma}
	\newtheorem{proposition}[theorem]{Proposition}
	\newtheorem{corollary}[theorem]{Corollary}

\theoremstyle{definition}
	\newtheorem{definition}[theorem]{Definition}
	\newtheorem{lemma and definition}[theorem]{Lemma and Definition}
	\newtheorem{example}[theorem]{Example}

	\newtheorem{notation}[theorem]{Notation}
	\newtheorem{setting}[theorem]{Setting}
\theoremstyle{remark}
	\newtheorem{remark}[theorem]{Remark}
	\numberwithin{equation}{section}

\DefineSymbol{\pr}{pr}
\DefineSymbol{\id}{id}
\DefineSymbol{\const}{const}
\DefineSymbol{\op}{op}
\DefineSymbol{\diag}{diag}
\DefineSymbol{\proet}{pro\acute{e}t}
\DefineSymbol{\cond}{cond}
\DefineSymbol{\cont}{cont}
\DefineSymbol{\conti}{conti}
\DefineSymbol{\Cond}{Cond}
\DefineSymbol{\dCond}{dCond}
\DefineSymbol{\disc}{disc}
\DefineSymbol{\Tot}{Tot}
\DefineSymbol{\triv}{triv}
\DefineSymbol{\Bun}{Bun}
\DefineSymbol{\adic}{adic}
\DefineSymbol{\rig}{rig}
\DefineSymbol{\nc}{nc}
\DefineSymbol{\nuc}{nuc}
\DefineSymbol{\cpt}{cpt}
\DefineSymbol{\an}{an}
\DefineSymbol{\alg}{alg}
\DefineSymbol{\la}{la}
\DefineSymbol{\rla}{rla}
\DefineSymbol{\dep}{dep}
\DefineSymbol{\red}{red}
\DefineSymbol{\ex}{ex}

\DeclareMathOperator{\Hom}{Hom}

\DeclareMathOperator{\Map}{Map}
\DeclareMathOperator{\Fun}{Fun}

\DeclareMathOperator{\Ker}{Ker}
\DeclareMathOperator{\Coker}{Coker}

\DeclareMathOperator{\Aut}{Aut}
\DeclareMathOperator{\Gal}{Gal}

\DeclareMathOperator{\Eq}{Eq}

\DeclareMathOperator{\fib}{fib}

\DeclareMathOperator{\Tr}{Tr}

\DeclareMathOperator{\PSh}{PSh}
\DeclareMathOperator{\Shv}{Shv}

\DeclareMathOperator{\AnSpec}{AnSpec}
\DeclareMathOperator{\Spa}{Spa}

\DeclareMathOperator{\Kos}{Kos}

\DeclareMathOperator{\VB}{VB}
\DeclareMathOperator{\SD}{SD}
\DeclareMathOperator{\PD}{PD}
\DeclareMathOperator{\intHom}{\underline{Hom}}

\newcommand{\heart}{\heartsuit}

\newcommand{\hotimes}{\mathbin{\hat{\otimes}}}
\newcommand{\mapsfrom}{\mathrel{\reflectbox{\ensuremath{\mapsto}}}}

\DefineCategory{\Set}{Set}
\DefineCategory{\Ab}{Ab}
\DefineCategory{\Ring}{Ring}
\DefineCategory{\Mod}{Mod}
\DefineCategory{\LMod}{LMod}
\DefineCategory{\Rep}{Rep}
\DefineCategory{\coMod}{coMod}
\DefineCategory{\Alg}{Alg}
\DefineCategory{\Ch}{Ch}
\DefineCategory{\Mon}{Mon}
\DefineCategory{\CMon}{CMon}
\DefineCategory{\PCoh}{PCoh}
\DefineCategory{\Perf}{Perf}
\DefineCategory{\Vect}{Vect}
\DefineCategory{\FP}{FP}
\DefineCategory{\cAff}{cAff}
\DefineCategory{\Aff}{Aff}
\DefineCategory{\AffRing}{AffRing}
\DefineCategory{\AnRing}{AnRing}
\DefineCategory{\AnAdic}{AnAdic}
\DefineCategory{\AffAnAdic}{AffAnAdic}
\DefineCategory{\Ani}{Ani}
\DefineCategory{\CAlg}{CAlg}
\DefineCategory{\Corr}{Corr}
\DefineCategory{\Comm}{Comm}
\renewcommand{\Pr}{\mathrm{Pr}}

\newcommand{\Cat}{\mathcal{C}at}

\renewcommand{\Bbb}{\mathbb{B}}

\newcommand{\Ebb}{\mathbb{E}}

\newcommand{\Gbb}{\mathbb{G}}
\newcommand{\Hbb}{\mathbb{H}}

\newcommand{\Qbb}{\mathbb{Q}}

\newcommand{\Zbb}{\mathbb{Z}}

\newcommand{\Acal}{\mathcal{A}}
\newcommand{\Bcal}{\mathcal{B}}
\newcommand{\Ccal}{\mathcal{C}}
\newcommand{\Dcal}{\mathcal{D}}
\newcommand{\Ecal}{\mathcal{E}}
\newcommand{\Fcal}{\mathcal{F}}

\newcommand{\Kcal}{\mathcal{K}}
\newcommand{\Lcal}{\mathcal{L}}
\newcommand{\Mcal}{\mathcal{M}}
\newcommand{\Ncal}{\mathcal{N}}
\newcommand{\Ocal}{\mathcal{O}}

\newcommand{\Rcal}{\mathcal{R}}

\newcommand{\Ucal}{\mathcal{U}}

\newcommand{\Xfrak}{\mathfrak{X}}

\newcommand{\gfrak}{\mathfrak{g}}

\renewcommand{\tilde}{\widetilde}
\renewcommand{\hat}{\widehat}
\renewcommand{\bar}{\overline}

\begin{document}

\title[Finiteness and duality of cohomology of $(\varphi,\Gamma)$-modules]{Finiteness and duality of cohomology of $(\varphi,\Gamma)$-modules and the 6-functor formalism of locally analytic representations}%[論文の上に表示されるタイトル]{正式なタイトル}
\author{Yutaro Mikami}
\date{\today}
\address{Graduate School of Mathematical Sciences, University of Tokyo, 3-8-1 Komaba, Meguro-ku, Tokyo 153-8914, Japan}
\email{y-mikmi@g.ecc.u-tokyo.ac.jp}
%14G22 リジッド幾何学 13D09 可換環と導来圏 11F80 Galois representations 11S25 Galois cohomology
\begin{abstract}
Finiteness and duality of the cohomology of families of $(\varphi,\Gamma)$-modules were proved by Kedlaya-Pottharst-Xiao.
In this paper, we study solid locally analytic representations introduced by Rodrigues Jacinto-Rodr\'{\i}guez Camargo in terms of analytic stacks and 6-functor formalisms, which are developed by Clausen-Scholze, Heyer-Mann, respectively.
By using this, we provide a generalization of the result of Kedlaya-Pottharst-Xiao, giving a new proof for cases already proved there.
\end{abstract}
\maketitle

\setcounter{tocdepth}{1}

\tableofcontents

%%%%%%%%%%%%%%%%%%%%%%%%%%%%%%%%%%%%%%%%%% introduction %%%%%%%%%%%%%%%%%%%%%%%%%%%%%%%%%%%%%%%%%%
\section*{Introduction}
\subsection{Background}
Let $K$ be a finite extension of $\Qbb_p$, and $G_K$ be the absolute Galois group of $K$.
The theory of \textit{$(\varphi,\Gamma_K)$-modules} is one of the most powerful tools for the study of $p$-adic representations of $G_K$.
In \cite{Fontaine90}, Fontaine introduced \textit{\'{e}tale $(\varphi,\Gamma_K)$-modules}, and he proved that the category of \'{e}tale $(\varphi,\Gamma_K)$-modules is equivalent to the category of $p$-adic representations of $G_K$.
Berger introduced \textit{$(\varphi,\Gamma_K)$-modules over the Robba ring $\Rcal_K$} in \cite{Ber02} building on works of Cherbonnier-Colmez and Kedlaya \cite{CC98, Ked04}.
The category of $(\varphi,\Gamma_K)$-modules over the Robba ring contains the category of \'{e}tale $(\varphi,\Gamma_K)$-modules (or equivalently the category of $p$-adic representations of $G_K$) as a full subcategory.
It plays an important role in Colmez's proof of the \textit{(locally analytic) $p$-adic local Langlands correspondence for $GL_2(\Qbb_p)$} (\cite{Col10,Col16}), which is the correspondence between unitary Banach representations (resp. locally analytic representations) of $GL_2(\Qbb_p)$ and \'{e}tale $(\varphi,\Gamma_{\Qbb_p})$-modules of rank $2$ (resp. $(\varphi,\Gamma_{\Qbb_p})$-modules over the Robba ring $\Rcal_{\Qbb_p}$ of rank $2$).

A generalization of Galois cohomology of $p$-adic Galois representations to $(\varphi,\Gamma_K)$-modules was studied by Herr and Liu in \cite{Herr98, Liu08}, and Liu proved finiteness, the Tate local duality, and the Euler characteristic formula for cohomology of $(\varphi,\Gamma_K)$-modules.
In \cite{KPX14}, Kedlaya-Pottharst-Xiao extended this result to families of $(\varphi,\Gamma_K)$-modules, that is, $(\varphi,\Gamma_K)$-modules over ``$\Rcal_{K,A}=\Rcal_{K}\otimes_{\Qbb_p} A$'' for an affinoid $\Qbb_p$-algebra $A$.
For example, these results are used in the study of eigenvarieties.

Recently, Emerton-Gee-Hellmann proposed a \textit{categorical $p$-adic local Langlands correspondence} in \cite{EGH23}.
Roughly speaking, they conjecture that there exists a nice functor 
$$\Rep^{\la}_{\square}(GL_n(K))\to \Dcal(\Xfrak_{K,n}),$$ 
where $\Rep^{\la}_{\square}(GL_n(K))$ is the stable $\infty$-category of locally analytic representations of $GL_n(K)$, and $\Dcal(\Xfrak_{K,n})$ is the stable $\infty$-category of solid quasi-coherent sheaves on \textit{the analytic Emerton-Gee stack} $\Xfrak_{K,n}$, which is a rigid analytic moduli stack of $(\varphi,\Gamma_K)$-modules over $\Rcal_K$ of rank $n$.
More precisely, we need to restrict both categories to subcategories satisfying (conjectural) suitable conditions, but we will ignore this here.
It is natural to ask whether $\Xfrak_{K,n}$ is a rigid analytic Artin stack.
As explained in \cite{EGH23}, for every point $x\in \Xfrak_{K,n}(L)$ where $L$ is a finite extension of $\Qbb_p$, there is a smooth morphism from a rigid analytic space to $\Xfrak_{K,n}$ whose image contains $x$.
However, unlike the case of algebraic varieties, rigid analytic varieties contain many closed points other than $x\in \Xfrak_{K,n}(L)$.
Therefore, considering a neighborhood of $x\in \Xfrak_{K,n}(L)$ alone is insufficient to prove that $\Xfrak_{K,n}$ is a rigid analytic Artin stack, and we need to consider a neighborhood of a closed point $x$ such that $k(x)$ is not finite over $\Qbb_p$.

Moreover, there is another problem: The conjectural functor 
$$\Rep^{\la}_{\square}(GL_n(K))\to \Dcal(\Xfrak_{K,n})$$ 
is not expected to be fully faithful (\cite[Remark 6.2.8 (c)]{EGH23}).
As suggested in \cite[Theorem 4.4.4]{RJRC23}, this problem is expected to be overcome by considering a modification of $\Xfrak_{K,n}$.
To give a moduli interpretation of this modification, it will be necessary to consider families of $(\varphi,\Gamma_K)$-modules over algebras which are not affinoid $\Qbb_p$-algebras (\cite{Mikami24}).

\subsection{Statements of the main result}
In this paper, we use the theory of solid $\Dcal$-stacks introduced by Clausen-Scholze and Rodr\'{\i}guez Camargo (\cite{RC24}).
There is a 6-functor formalism on solid quasi-coherent sheaves on solid $\Dcal$-stacks, which is a powerful tool for proving finiteness and duality.
Moreover, this theory allows us to treat various coefficient rings, such as Banach $\Qbb_p$-algebras and algebraic-affinoid $\Qbb_{p,\square}$-algebras considered in \cite{Mikami24}.
These are the reasons why we use the theory of solid $\Dcal$-stacks.

We prove finiteness and duality of the cohomology of families of $(\varphi,\Gamma_K)$-modules over general algebras which are not necessarily affinoid $\Qbb_p$-algebras in this paper.
First, we explain how to formulate this problem.
We set $K_{\infty}=K(\zeta_{p^{\infty}})$, and let $X_{K_{\infty}}=\left(\Spa W(\Ocal_{\hat{K}_{\infty}^{\flat}})\setminus\{p[p^{\flat}]=0\}\right)/\varphi^{\Zbb}$ denote the Fargues-Fontaine curve.
The Fargues-Fontaine curve $X_{K_{\infty}}$ admits a natural action of $\Gamma_K=\Gal(K_{\infty}/K)$, and therefore, we can define $X_K^{\la}=X_{K_{\infty}}^{\la}$ by taking $\Gamma_K$-locally analytic vectors.
For a precise definition, see Definition \ref{defn:locally analytic Fargues-Fontaine curve}.
In \cite{Mikami24}, for a (condensed) $\Qbb_{p,\square}$-algebra $A$ that is either a Banach $\Qbb_p$-algebra or an algebraic-affinoid $\Qbb_{p,\square}$-algebra, the author defined $(\varphi,\Gamma_K)$-modules over $B_{K,\infty,A}$ and their cohomology as a generalization of families of $(\varphi,\Gamma_K)$-modules over affinoid $\Qbb_p$-algebras and their cohomology, where $B_{K,\infty,A}$ is a family of rings of analytic functions on affinoid open subspaces of $X_{K,A_{\square}}^{\la}=X_{K}^{\la}\times \AnSpec A_{\square}$.
We note that a $(\varphi,\Gamma_K)$-module $\Mcal$ over $B_{K,\infty,A}$ can be regarded as a solid quasi-coherent sheaf on the solid $\Dcal$-stack $X_{K,A_{\square}}^{\la}/\Gamma_K^{\la}$.
Then we have the following proposition:

\begin{proposition}[{Proposition \ref{prop:comparison cohomology}}]
	For a $(\varphi,\Gamma_K)$-module $\Mcal$ over $B_{K,\infty,A}$, the $(\varphi,\Gamma_K)$-cohomology $\Gamma_{\varphi,\Gamma_K}(\Mcal)\in \Dcal(A_{\square})$ defined in \cite[Definition 3.1]{Mikami24} is equivalent to $(g_{A_{\square}})_*\Mcal \in \Dcal(A_{\square})$, where $g_{A_{\square}}\colon X_{K,A_{\square}}^{\la}/\Gamma_K^{\la} \to \AnSpec A_{\square}$ is the projection.
\end{proposition}

Therefore, finiteness and duality of the cohomology of families of $(\varphi,\Gamma_K)$-modules can be formulated as follows.
\begin{theorem}[{Corollary \ref{cor:X D-smooth}, Corollary \ref{cor:X weakly proper}, Theorem \ref{thm:Local Tate duality}}]\label{intro:main theorem}
	The morphism 
	$$g\colon X_{K}^{\la}/\Gamma_K^{\la} \to \AnSpec \Qbb_{p,\square}$$ 
	is weakly $\Dcal$-proper and $\Dcal$-smooth with the dualizing complex $\Ocal_{X_{K}^{\la}}\cdot \chi[2]$, where $\Ocal_{X_{K}^{\la}}\cdot \chi$ is the twist of $\Ocal_{X_{K}^{\la}}$ by the $p$-adic cyclotomic character $\chi\colon \Gamma_K\to\Zbb_p^{\times}$.
\end{theorem}
The weakly $\Dcal$-properness leads to an equivalence $(g_{A_{\square}})_!\simeq (g_{A_{\square}})_*$, and the $\Dcal$-smoothness leads to an equivalence $g_{A_{\square}}^!A\simeq \Ocal_{X_{K,A_{\square}}^{\la}}\cdot \chi[2]$.
Therefore, we get the Tate local duality from Theorem \ref{intro:main theorem}:
\begin{align*}
	\intHom_{A}(\Gamma_{\varphi,\Gamma_K}(\Mcal),A)&\simeq \intHom_{A}((g_{A_{\square}})_!\Mcal,A)\\
	&\simeq (g_{A_{\square}})_* \intHom(\Mcal,g_{A_{\square}}^!A)\\
	&\simeq \Gamma_{\varphi,\Gamma_K}\left(\intHom(\Mcal,\Ocal_{X_{K,A_{\square}}^{\la}}\cdot \chi[2])\right).
\end{align*}

Next, we explain how to prove this theorem.
The most difficult part of the proof is to show that $g$ is $\Dcal$-smooth, so we only explain how to prove this part.
We set $\widetilde{B}_{K_{\infty}}^{I}=\Ocal_{Y_{K_{\infty}}}(Y_{K_{\infty}}^{I})$ for a closed interval $I=[p^{-k},p^{-l}]$, where $Y_{K_{\infty}}^{I}$ is the rational open subspace of $Y_{K_{\infty}}=\Spa W(\Ocal_{\hat{K}_{\infty}^{\flat}})\setminus\{p[p^{\flat}]=0\}$ defined by $|p^{p^{-l}}|\leq |[p^{\flat}]|\leq |p^{p^{-k}}|$.
Let $B_{K,\infty}^I\subset \widetilde{B}_{K_{\infty}}^{I}$ denote the subalgebra of $\Gamma_K$-locally analytic vectors.
Then $\AnSpec (B_{K,\infty}^I,B_{K,\infty}^{I+})_{\square}$ becomes ``an open subspace'' of $X_{K}^{\la}$ and such open subspaces form an open cover of $X_{K}^{\la}$.
Therefore, it is enough to show that $\AnSpec (B_{K,\infty}^I,B_{K,\infty}^{I+})_{\square}/\Gamma_K^{\la}\to \AnSpec\Qbb_{p,\square}$ is $\Dcal$-smooth.
Key ingredients of the proof are the following:
\begin{itemize}
	\item There are closed subalgebras $B_{K,0}^I \subset B_{K,1}^I \subset \cdots \subset\widetilde{B}_{K_{\infty}}^{I}$ such that $B_{K,\infty}^I=\varinjlim_{n} B_{K,n}^I$ and they satisfy the Tate-Sen conditions.
	\item The morphism $\Spa(B_{K,n}^I,B_{K,n}^{I+})\to \Spa(\Qbb_p)$ is smooth, and the morphism $\Spa(B_{K,n}^I,B_{K,n}^{I+})\to \Spa(B_{K,m}^I,B_{K,m}^{I+})$ is finite \'{e}tale for $n\geq m$ in Huber's sense.
	\item For every $n\geq 0$, there exists $h(n)\geq 0$ such that $B_{K,n}^I$ is a (static) $\Gamma_{K,h(n)}=\chi^{-1}(1+p^{h(n)}\Zbb_p)$-analytic representation.
	Conversely, for every $m\geq 0$, there exists $h^{\prime}(m)\geq 0$ such that the subalgebra of $\widetilde{B}_{K_{\infty}}^{I}$ consisting of $\Gamma_{K,m}$-analytic vectors is contained in $B_{K,h^{\prime}(m)}^I$.
\end{itemize}

We want to use \cite[Proposition D.2.9]{HM24} (Proposition \ref{prop:suave criterion}) to prove that $\AnSpec (B_{K,\infty}^I,B_{K,\infty}^{I+})_{\square}/\Gamma_K^{\la}\to \AnSpec\Qbb_{p,\square}$ is $\Dcal$-smooth.
To apply this proposition, a nice geometric interpretation of ``$\Gamma_{K,h(n)}$-analytic representations'' is necessary.
For simplicity, we assume $\chi(\Gamma_K)=1+p\Zbb_p$ and $p>2$.
We set $G=\Gamma_K$ and $G_h=\Gamma_{K,h}$.
There is a rigid analytic group $\Gbb_1=\Spa \Qbb_p\langle T/p\rangle$ such that $\Gbb_1(\Qbb_p)=1+p\Zbb_p$, which can be constructed from $\log \colon 1+p\Zbb_p \simeq p\Zbb_p$.
We set $\Gbb_h=\Gbb_1\langle T/p^h\rangle$, which is a rigid analytic subgroup.
%and $\mathring{\Gbb_h}=\cup_{h^{\prime}>h}\Gbb_{h^{\prime}}$ for $h\in \Qbb_{\geq 0}$
Let $\Gbb^h$ denote the rigid analytic group $\bigsqcup_{g\in G/\Gbb_h(\Qbb_p)}g \Gbb_h$, and let $C^h(G,\Qbb_p)$ denote the ring $\Ocal(\Gbb^h)$ of analytic functions on $\Gbb^h$.
We introduce $\Gbb^{h\dagger}$, which is the overconvergent version of $\Gbb^h$, and define $\Gbb^{h\dagger}$-analytic representations.
\begin{definition}
	Let $h>1$ be a rational number.
	\begin{enumerate}
	\item Let $C^{h\dagger}(G,\Qbb_p)\coloneqq \varinjlim_{h>h^{\prime}\geq 1}C^{h^{\prime}}(G,\Qbb_p)$ denote the space of overconvergent analytic functions on $\Gbb^h$.
	\item We define a dagger group $\Gbb^{h\dagger}=\AnSpec C^{h\dagger}(G,\Qbb_p)_{\square}$.
	\item For $V\in \Dcal(\Qbb_{p,\square})$, we define the space of $\Gbb^{h\dagger}$-analytic functions with values in $V$ as 
	\begin{align*}
		C^{h\dagger}(G,V)\coloneqq C^{h\dagger}(G,\Qbb_p)_{\square} \otimes_{\Qbb_{p,\square}} V.
	\end{align*}
	If $V$ has a continuous $G$-action, then $C^{h\dagger}(G,V)$ also has a $G$-action given by $(gf)(h)=g\cdot f(g^{-1}h)$ for $f\in C^{h\dagger}(G,V)$ and $g,h \in G$.
	\item A representation $V\in \Rep_{\square}(G)=\Mod_{\Qbb_{p,\square}[G]}(\Dcal(\Qbb_{p,\square}))$ of $G$ is said to be $\Gbb^{h\dagger}$-analytic if the natural morphism 
	$$V^{h\dagger-an}= \intHom_{\Qbb_{p,\square}[G]}(\Qbb_p,C^{h\dagger}(G,V))\to V$$
	is an equivalence.
	Let $\Rep_{\square}^{h\dagger}(G)\subset \Rep_{\square}(G)$ denote the full subcategory of $\Gbb^{h\dagger}$-analytic representations.
	\end{enumerate}
\end{definition}

\begin{remark}
	Such representations are also studied in \cite{LSS23} by using classical $p$-adic analysis.
\end{remark}

In \cite{RJRC23}, Rodrigues Jacinto-Rodr\'{\i}guez Camargo defined a group object $\Gbb^{\la}$ such that $\Dcal(*/\Gbb^{\la})$ is equivalent to the $\infty$-category $\Rep_{\square}^{\la}(G)$ of solid locally analytic representations, which already appeared as $\Gamma_K^{\la}$ in the statement of Theorem \ref{intro:main theorem}.

\begin{proposition}[{Proposition \ref{prop:geometric interpretation}, Proposition \ref{prop:6-ff classifying stack}}]
There is an equivalence of symmetric monoidal $\infty$-categories
$$\Rep_{\square}^{h\dagger}(G) \simeq \Dcal(*/\Gbb^{h\dagger}).$$
Under this equivalence, the pullback along the morphism $*/\Gbb^{\la} \to */\Gbb^{h\dagger}$ is equivalent to the inclusion $\Rep_{\square}^{h\dagger}(G)\subset \Rep_{\square}^{\la}(G)$.
Moreover, the morphisms $*/\Gbb^{\la} \to */\Gbb^{h\dagger}$ and $*/\Gbb^{h\dagger}\to \AnSpec\Qbb_{p,\square}$ are weakly $\Dcal$-proper, and the morphism $*/\Gbb^{h\dagger}\to \AnSpec\Qbb_{p,\square}$ is $\Dcal$-smooth for $h$ sufficiently large.
\end{proposition}

By using this proposition, we can apply \cite[Proposition D.2.9]{HM24}, and then, we find that $\AnSpec (B_{K,\infty}^I,B_{K,\infty}^{I+})_{\square}/\Gamma_K^{\la}\to \AnSpec\Qbb_{p,\square}$ is $\Dcal$-smooth.

\begin{remark}
	Finiteness and duality of the cohomology of families of $(\varphi,\Gamma_K)$-modules are also studied in \cite[Corollary 4.5.2, Remark 4.5.3]{ALM24}.
	This paper uses the 6-functor formalism on solid $\Dcal$-stacks, whereas Ansch\"{u}tz-Le Bras-Mann use the 6-functor formalism $\Dcal_{[0,\infty)}$ on small $v$-stacks.
	Moreover, this paper considers the locally analytic action of $\Gamma_K$, whereas they consider the continuous action.
\end{remark}
\begin{remark}
	It is natural to ask whether the morphism
	$$X_{K_{\infty}}/\Gamma_K \to \AnSpec \Qbb_{p,\square}$$
	is also $\Dcal$-smooth.
	The author expects that it should be true, but we will not pursue this problem in this paper.
\end{remark}

\subsection{Outline of the paper}
This paper is organized as follows.
In Section 1, we recall the basic notions of solid $\Dcal$-stacks defined in \cite{RC24}, and the basic properties of 6-functor formalisms proved in \cite{HM24}.
In Section 2, we define $\Gbb^{h\dagger}$ and $\Gbb^{h\dagger}$-analytic representations, and then we prove basic properties of them.
In the former part of Section 3, we prove  for morphisms satisfying some conditions related with the Tate-Sen axioms.
In the latter part of Section 3, we prove finiteness and duality of the cohomology of $(\varphi,\Gamma_K)$-modules.
In Appendix A, we study the Poincar\'{e} duality for proper morphisms.

\subsection{Convention}
\begin{itemize}
	\item
	All rings, including condensed ones, are assumed unital and commutative.
	\item 
	In contrast to \cite{Mikami24}, we use the symbol $-\otimes-$ to refer to a \textit{derived tensor product} and use the symbol $H^0(-\otimes-)$ or $\pi_0(-\otimes -)$ to refer to a \textit{non-derived tensor product}.
	Similarly, we adopt analogous notation for Hom and limit.
	\item
	We use the terms \textit{f-adic ring} and \textit{affinoid pair} rather than \textit{Huber ring} and \textit{Huber pair}.
	% \item For an f-adic ring $A$, we denote the ring of power-bounded elements of $A$ by $A^{\circ}$.
	\item For a complete non-archimedean field $K$, we use the term \textit{affinoid $K$-algebra} to refer to a \textit{topological $K$-algebra topologically of finite type over $K$}.
	\item
	We denote the augmented simplex category by $\Delta_+$, which is the full subcategory of the category of totally ordered sets consisting of the totally ordered sets $[n]=\{0,\ldots,n\}$ for all $n \geq -1$, where $[-1]=\emptyset$.
	We also denote the simplex category by $\Delta$, which is the full subcategory of $\Delta_+$ consisting of $[n]$ for $n\geq 0$.

	For $n\geq 1$ and $0 \leq i \leq n$, let $d_n^i \colon [n-1]\to [n]$ denote the $i$-th coface map defined by
	$$d_n^i(j)=\begin{cases} j & (j<i)\\
		j+1 & (j\geq i).
	\end{cases}
	$$
	For $n\geq 0$ and $0 \leq i \leq n$, let $s_n^i \colon [n+1]\to [n]$ denote the $i$-th codegeneracy map defined by
	$$s_n^i(j)=\begin{cases} j & (j\leq i)\\
		j-1 & (j> i).
	\end{cases}
	$$
	%\item
	%For a topological ring $A$, we denote the discrete ring whose underlying ring is the underlying ring of $A$ by $A_{\disc}$. 
	\item Throughout this paper, all radii $r$ and $s$ are assumed to be rational numbers.
	\item In this paper, a 2-category refers to a bicategory rather than a strict 2-category. 
	\item An adjunction
	$$F \colon \Ccal \rightleftarrows \Dcal\colon G$$
	means that $F$ is a left adjoint functor of $G$.
\end{itemize}

\subsection{Convention and notation about condensed mathematics}
In this paper, we use condensed mathematics.
We summarize the notations and conventions related to condensed mathematics.

\begin{itemize}
	\item Throughout this paper, we fix an uncountable solid cutoff cardinal $\kappa$ as in \cite[Definition 2.9.11]{Mann22} and work with $\kappa$-condensed objects.
	Our results do not depend on the choice of $\kappa$.
	If the reader prefers to work with light condensed objects, one may simply replace ``condensed'' with ``light condensed'' throughout the paper without affecting any of the arguments.
    \item We often identify a compactly generated topological set, ring, group, etc. $X$ whose points are closed (i.e., $X$ is $T1$) with a condensed set, ring, group, etc. $\underline{X}$ associated to $X$. 
    It is justified by \cite[Proposition 1.7]{CM}.
    If there is no room for confusion, we simply write $X$ for $\underline{X}$.
    \item 
    In contrast to \cite{Mann22}, we use the term \textit{ring} to refer to an \textit{ordinary ring} (not a condensed animated ring). 
    Sometimes we use the term \textit{discrete ring} (resp. \textit{discrete animated ring}) to refer to an ordinary ring (resp. animated ring) in order to emphasize that it is not a condensed one. We also use the term \textit{static ring} (resp. \textit{static analytic ring})  to refer to an ordinary ring (resp. analytic ring) in order to emphasize that it is not an animated one.
    \item We use the terms ``analytic animated ring" and ``uncompleted analytic animated ring" according to \cite{Mann22}.
    \item For an uncompleted analytic animated ring $\Acal$, we denote the underlying condensed animated ring of $\Acal$ by $\underline{\Acal}$.
    \item For an uncompleted analytic animated ring $\Acal$, an object $M\in \Dcal(\underline{\Acal})$ is said to be \textit{$\Acal$-complete} if it lies in $\Dcal(\Acal)$.    
    \item Let $\Zbb_{\square}$ and $\Zbb_{p,\square}$ denote the analytic rings defined in \cite[Example 7.3]{CM}. 
    Moreover, for a usual ring $A$, let $(A,A)_{\square}$ denote the analytic ring defined in \cite[Definition 2.9.1]{Mann22}.
    We note that the analytic ring structure of $\Zbb_{p,\square}$ is induced from $\Zbb_{\square}$.
    \item For a condensed animated ring $A$ and for a morphism of usual rings $B\to \pi_0A(\ast)$, let $(A,B)_{\square}$ denote the condensed animated ring $A$ with the induced analytic ring structure from $(B,B)_{\square}$.
    For details, see \cite[Definition 2.1.5]{RC24}.
    When $B=\Zbb$, then we simply write $A_{\square}$ for $(A,\Zbb)_{\square}$.
    \item For animated $\Qbb_{p,\square}$-algebras $R$ and $A$, let $R_A$ denote the condensed animated $A$-algebra $R \otimes_{\Qbb_{p,\square}}A$.
\end{itemize}

\subsection*{Acknowledgements}
The author is grateful to Yoichi Mieda for his support during the studies of the author.
In addition, the author is grateful to Lucas Mann for his answers in questions about the 6-functor formalism.
This paper originated from discussions with Juan Esteban Rodr\'{\i}guez Camargo at the ``Workshop on Shimura varieties, representation theory and related topics'' held in October 2024 in Tokyo, and the author would like to thank him.
This work was supported by JSPS KAKENHI Grant Number JP23KJ0693.

%%%%%%%%%%%%%%%%%%%%%%%%%%%%%%%%%%%%%%%%%%%%% solid D-stacks %%%%%%%%%%%%%%%%%%%%%%%%%%%%%%%%%%

\section{Solid $\Dcal$-stacks and 6-functor formalisms}
In this section, we briefly recall some results of \cite{RC24} and \cite{HM24}.
For details, see loc. cit.
\subsection{Solid $\Dcal$-stacks}
In this subsection, we define solid $\Dcal$-stacks, and construct a 6-functor formalism on them according to \cite[3.2]{RC24}.

Let $\AnRing_{\Zbb_{\square}}$ denote the $\infty$-category of analytic animated rings over $\Zbb_{\square}$.

\begin{definition}[{\cite[Definition 2.6.1, Definition 2.6.6]{RC24}}]
	Let $\Acal\in \AnRing_{\Zbb_{\square}}$ be an analytic animated ring over $\Zbb_{\square}$.
	\begin{enumerate}
		\item \textit{The subring of $+$-bounded elements} is the (discrete) animated ring given by
		$$\Acal^+=\Map_{\AnRing_{\Zbb_{\square}}}((\Zbb[T],\Zbb[T])_{\square},\Acal).$$
		\item The analytic animated ring $\Acal$ over $\Zbb_{\square}$ is called a \textit{solid affinoid animated ring} if the natural morphism of analytic animated rings
		$$(\underline{\Acal},\pi_0(\Acal^+))_{\square}\to \Acal$$
		is an equivalence.
		Let $\AffRing_{\Zbb_{\square}}\subset \AnRing_{\Zbb_{\square}}$ denote the full subcategory of solid affinoid animated rings.
		Moreover, let $\Aff_{\Zbb_{\square}}$ denote the opposite category of $\AffRing_{\Zbb_{\square}}$, and we call it the $\infty$-category of \textit{solid affinoid spaces}.
		We also let $\AnSpec \Acal\in \Aff_{\Zbb_{\square}}$ denote the object corresponding to $\Acal \in \AffRing_{\Zbb_{\square}}$.
	\end{enumerate}
\end{definition}
\begin{remark}
	For a solid affinoid animated ring $\Acal$, let $\AffRing_{\Acal}$ (resp. $\Aff_{\Acal}$) denote the $\infty$-category of solid affinoid animated rings over $\Acal$ (resp. solid affinoid spaces over $\AnSpec \Acal$).
\end{remark}

\begin{definition}[{\cite[Definition 2.2.2, Proposition 2.2.4]{RC24}}]\label{def:open locale}
	A morphism 
	$$f\colon \AnSpec \Bcal \to \AnSpec \Acal$$ 
	in $\Aff_{\Zbb_{\square}}$ is called an \textit{open immersion in the associated locale} if it satisfies that the functor
	$$f^*=\Bcal\otimes_{\Acal}- \colon \Dcal(\Acal)\to \Dcal(\Bcal)$$
	admits a fully faithful left adjoint functor $f_!\colon  \Dcal(\Bcal) \to \Dcal(\Acal)$ such that the natural morphism 
	$$f_!(f^*M\otimes_{\Bcal}N)\to M\otimes_{\Acal} f_!N$$
	becomes an equivalence for every $M\in \Dcal(\Acal)$ and $N\in \Dcal(\Bcal)$ (in other words, it satisfies the projection formula).
	If there is no room for confusion, we simply call it an \textit{open immersion}.
\end{definition}

\begin{example}
Let $\Spa(B,B^+)\to \Spa(A,A^+)$ be an open immersion of (sheafy) analytic adic spaces.	
Then the morphism $\AnSpec(B,B^+)_{\square}\to \AnSpec(A,A^+)_{\square}$ is an open immersion in the associated locale by \cite[Proposition 2.3.2]{RC24}.
\end{example}

\begin{definition}\label{def:suitable dcp}
	\begin{enumerate}
		\item Let $I$ be the family of open immersions in $\Aff_{\Zbb_{\square}}$.
		\item Let $P$ be the family of morphisms in $\Aff_{\Zbb_{\square}}$ consisting of those morphisms $f\colon \AnSpec \Bcal \to \AnSpec \Acal$ such that the analytic ring structure of $\Bcal$ is induced from $\Acal$.
		\item Let $E$ be the family of morphisms in $\Aff_{\Zbb_{\square}}$ consisting of those morphisms of the form $f \circ i$ with $i\in I$ and $f\in P$.
	\end{enumerate}
\end{definition}

\begin{lemma}[{\cite[Lemma 3.2.5]{RC24}}]
	The pair $(P,I)$ is a suitable decomposition of the geometric set up $(\Aff_{\Zbb_{\square}},E)$ in the sense of \cite[Definition 3.3.2]{HM24}.
	Moreover, 
	the functor
	$$\Aff_{\Zbb_{\square}} \to \CAlg(\Pr^{L,\ex});\; \Acal \mapsto \Dcal(\Acal)$$
	extends to a 6-functor formalism
	$$\Dcal \colon \Corr(\Aff_{\Zbb_{\square}},E)\to \Pr^{L,\ex},$$
	where $\Pr^{L,\ex}$ is the $\infty$-category of presentable stable $\infty$-categories and $\CAlg(\Pr^{L,\ex})$ is the $\infty$-category of presentably symmetric monoidal stable $\infty$-categories.
\end{lemma}

By \cite[Proposition 3.4.2]{HM24}, we can extend it to a 6-functor formalism
$$\Dcal\colon \Corr(\PSh(\Aff_{\Zbb_{\square}}),E^{\prime})\to \Pr^{L,\ex},$$
where $E^{\prime}$ is the family of morphisms in $\PSh(\Aff_{\Zbb_{\square}})$ consisting of those morphisms $f^{\prime}\colon Y^{\prime}\to X^{\prime}$ such that for every morphism $g \colon X\to X^{\prime}$ from a solid affinoid space $X \in \Aff_{\Zbb_{\square}}$, the pullback of $f^{\prime}$ along $g$ lies in $E$.

\begin{notation}
Let $(\Ccal,E)$ be a geometric setup.
\begin{itemize}
	\item We let $\Ccal_E$ denote the wide subcategory of $\Ccal$ with morphisms $E$.
	\item Let $\Dcal \colon \Corr(\Ccal,E) \to \Cat_{\infty}$ be a 6-functor formalism, where $\Cat_{\infty}$ is the $\infty$-category of $\infty$-categories.
	Then $$(f\colon Y\to X) \mapsto (f^*\colon \Dcal(X)\to \Dcal(Y))$$ defines a functor
	$$\Dcal^*\colon \Ccal^{\op} \to \Cat_{\infty}.$$
	Similarly, $$E\ni (f\colon Y\to X) \mapsto (f^!\colon \Dcal(X)\to \Dcal(Y))$$ defines a functor
	$$\Dcal^!\colon (\Ccal_E)^{\op} \to \Cat_{\infty}.$$
\end{itemize}
\end{notation}

\begin{definition}[{\cite[Definition 3.1.5]{RC24}, \cite[Definition 3.4.6]{HM24}}]\label{defn:solid space *-cover !-cover}
	Let $X\in \Aff_{\Zbb_{\square}}$ be a solid affinoid space.
	\begin{enumerate}
		\item Let $\Ucal\subset (\Aff_{\Zbb_{\square}})_{/X}$ be a sieve.
		Then it is called a \textit{universal $\Dcal^*$-cover}\footnote{More precisely, to avoid set-theoretic size issues, we should restrict the $\infty$-category $\Aff_{\Zbb_{\square}}$ to an essentially small subcategory, for instance the full subcategory consisting of $\kappa'$-compact objects for a sufficiently large regular cardinal $\kappa'$, where we note that $\Aff_{\Zbb_{\square}}$ is compactly generated $\infty$-category by \cite[Proposition 2.6.8]{RC24}.
		Since the particular choice of such an essentially small subcategory does not affect any of our arguments, we will ignore these size issues for the remainder of the paper.} if 
		$$\Dcal^*\colon \Aff_{\Zbb_{\square}}^{\op}\to \Cat_{\infty}$$ descends universally along $\Ucal$.
		\item Let $\Ucal\subset ((\Aff_{\Zbb_{\square}})_E)_{/X}$ be a sieve.
		Then it is called a \textit{small universal $\Dcal^!$-cover} if it is generated by a small family of morphisms $\{X_i\to X\}_{i}$ and 
		$$\Dcal^!\colon (\PSh(\Aff_{\Zbb_{\square}})_{E^{\prime}})^{\op}\to \Cat_{\infty}$$ descends universally along $\{X_i\to X\}_{i}$.
		\item Let $\Ucal\subset ((\Aff_{\Zbb_{\square}})_E)_{/X}$ be a sieve.
		Then it is called a \textit{$\Dcal$-cover} if it is a small universal $\Dcal^!$-cover and if it generates a sieve of $(\Aff_{\Zbb_{\square}})_{/X}$ that is both a canonical cover and a universal $\Dcal^*$-cover.
		The topology on $\Aff_{\Zbb_{\square}}$ defined by $\Dcal$-covers is called the \textit{$\Dcal$-topology}.
	\end{enumerate}
\end{definition}
\begin{remark}
	By construction, $\Dcal^*\colon \PSh(\Aff_{\Zbb_{\square}})^{\op}\to \Cat_{\infty}$ preserves all small limits.
	Therefore, for a universal $\Dcal^*$-cover $\Ucal\subset (\Aff_{\Zbb_{\square}})_{/X}$, $$\Dcal^*\colon \PSh(\Aff_{\Zbb_{\square}})^{\op}\to \Cat_{\infty}$$
	descends universally along $\Ucal$. 
\end{remark}

\begin{definition}[{\cite[Definition 3.2.7]{RC24}}]\label{def:solid D-stack 6ff}
	We call the $\infty$-category of sheaves of anima on $\Aff_{\Zbb_{\square}}$ with the $\Dcal$-topology $\Shv_{\Dcal}(\Aff_{\Zbb_{\square}})$ the $\infty$-category of \textit{solid $\Dcal$-stacks}\footnote{If the reader prefers to use analytic stacks defined in \cite[Definition 4.2.1]{ABLBRCS25}, one may simply replace ``solid $\Dcal$-stacks'' with ``analytic stacks'' throughout the paper without affecting any of the arguments.}.
	By \cite[Theorem 3.4.11]{HM24}, we can extend the 6-functor formalism 
	$$\Dcal \colon \Corr(\Aff_{\Zbb_{\square}},E)\to \Pr^{L,\ex}$$
	to
	$$\Dcal\colon \Corr(\Shv_{\Dcal}(\Aff_{\Zbb_{\square}}),\widetilde{E})\to \Pr^{L,\ex},$$
	where $\widetilde{E}$ satisfies the following:
	\begin{itemize}
		\item The functor $\Dcal^* \colon \Shv_{\Dcal}(\Aff_{\Zbb_{\square}})^{\op}\to \Pr^{L,\ex}$ preserves all small limits.
		\item The family $\widetilde{E}$ is $*$-local on the target: Let $f\colon Y\to X$ be a morphism in $\Shv_{\Dcal}(\Aff_{\Zbb_{\square}})$ whose pullback to every object in $\Aff_{\Zbb_{\square}}$ lies in $\widetilde{E}$. Then $f$ lies also in $\widetilde{E}$.
		\item The family $\widetilde{E}$ is $!$-local: Let $f\colon Y\to X$ be a morphism in $\Shv_{\Dcal}(\Aff_{\Zbb_{\square}})$ which lies, universally $\Dcal^!$-locally on the source or the target, in $\widetilde{E}$. Then $f$ lies also in $\widetilde{E}$.
		\item The family $\widetilde{E}$ is tame: Every morphism $f\colon Y\to X$ in $\widetilde{E}$ with $X\in \Aff_{\Zbb_{\square}}$ lies, universally $\Dcal^!$-locally on the source, in $E$.
	\end{itemize}
	For a solid $\Dcal$-stack $X$, we denote the monoidal unit of $\Dcal(X)$ by $\Ocal_X$.
\end{definition}

\begin{remark}
	By the uniqueness of \cite[Proposition 3.4.2]{HM24}, the 6-functor formalisms 
	$$\Dcal\colon \Corr(\PSh(\Aff_{\Zbb_{\square}}),E^{\prime})\to \Pr^{L,\ex}$$
	and
	$$\Dcal\colon \Corr(\Shv_{\Dcal}(\Aff_{\Zbb_{\square}}),\widetilde{E})\to \Pr^{L,\ex}$$
	coincide on the geometric setup $(\Shv_{\Dcal}(\Aff_{\Zbb_{\square}}),E^{\prime}\cap \Shv_{\Dcal}(\Aff_{\Zbb_{\square}}))$, which justifies the use of the same notation $\Dcal$ to denote the two 6-functor formalisms.
\end{remark}

\begin{remark}
	In the same way, we can define the $\infty$-category $\Shv_{\Dcal}(\Aff_{\Qbb_{p,\square}})$, which is equivalent to the $\infty$-category $\Shv_{\Dcal}(\Aff_{\Zbb_{\square}})_{/\AnSpec\Qbb_{p,\square}}$.
	Therefore, we will refer to objects of $\Shv_{\Dcal}(\Aff_{\Qbb_{p,\square}})$ as solid $\Dcal$-stacks over $\AnSpec\Qbb_{p,\square}$.
\end{remark}

\begin{definition}\label{defn:solid *-cover !-cover}
	Let $X\in \Shv_{\Dcal}(\Aff_{\Zbb_{\square}})$ be a solid $\Dcal$-stack.
	\begin{enumerate}
		\item Let $\Ucal\subset \Shv_{\Dcal}(\Aff_{\Zbb_{\square}})_{/X}$ be a sieve.
		Then it is called a \textit{universal $\Dcal^*$-cover} if 
		$$\Dcal^*\colon \Shv_{\Dcal}(\Aff_{\Zbb_{\square}})^{\op}\to \Cat_{\infty}$$ descends universally along $\Ucal$.
		\item Let $\Ucal\subset (\Shv_{\Dcal}(\Aff_{\Zbb_{\square}})_{\widetilde{E}})_{/X}$ be a sieve.
		Then it is called a \textit{small universal $\Dcal^!$-cover} if it is generated by a small family of morphisms $\{X_i\to X\}_{i}$ and 
		$$\Dcal^!\colon (\Shv_{\Dcal}(\Aff_{\Zbb_{\square}})_{\widetilde{E}})^{\op}\to \Cat_{\infty}$$ descends universally along $\{X_i\to X\}_{i}$.
	\end{enumerate}
\end{definition}
\begin{remark}
	Let $X\in \Aff_{\Zbb_{\square}}$ be a solid affinoid space.
	For a universal $*$-cover $\Ucal\subset (\Aff_{\Zbb_{\square}})_{/X}$ of $X$ (resp. a small universal $!$-cover $\Ucal\subset ((\Aff_{\Zbb_{\square}})_E)_{/X}$ of $X$), it generates a universal $*$-cover (resp. small universal $!$-cover) of $X$ in $\Shv_{\Dcal}(\Aff_{\Zbb_{\square}})$.
\end{remark}
\begin{remark}\label{rem:canonical cover *-cover}
	Since $\Dcal^* \colon \Shv_{\Dcal}(\Aff_{\Zbb_{\square}})^{\op}\to \Pr^{L,\ex}$ preserves all small limits, canonical covers in $\Shv_{\Dcal}(\Aff_{\Zbb_{\square}})$ are universal $\Dcal^*$-covers.
\end{remark}

%%%%%%%%%%%%%%%%%%%%%%%%%%%%%%%%%%%%%%%%%%%%% suave and prim %%%%%%%%%%%%%%%%%%%%%%%%%%%%%%%%%%%%%%%%%%%

\subsection{Suave and prim objects and morphisms}
First, we recall the definitions of suave and prim objects and morphisms.
Let $(\Ccal,E)$ be a geometric setup such that $\Ccal$ admits pullbacks, and $\Dcal \colon \Corr(\Ccal,E)\to \Cat_{\infty}$ be a 6-functor formalism.

\begin{definition}\label{defn:general *-cover !-cover}
	Let $X\in \Ccal$ be an object.
	\begin{enumerate}
		\item Let $\Ucal\subset \Ccal_{/X}$ be a sieve.
		Then it is called a \textit{universal $\Dcal^*$-cover} if 
		$$\Dcal^*\colon \Ccal^{\op}\to \Cat_{\infty}$$ descends universally along $\Ucal$.
		\item Let $\Ucal\subset (\Ccal_E)_{/X}$ be a sieve.
		Then it is called a \textit{small universal $\Dcal^!$-cover} if it is generated by a small family of morphisms $\{X_i\to X\}_{i}$ and 
		$$\Dcal^!\colon (\Ccal_E)^{\op}\to \Cat_{\infty}$$ descends universally along $\{X_i\to X\}_{i}$.
	\end{enumerate}
\end{definition}

\begin{remark}
	When the 6-functor formalism on $(\Ccal,E)$ is the 6-functor formalism on $(\Shv_{\Dcal}(\Aff_{\Zbb_{\square}}),\widetilde{E})$, then these notions defined in Definition \ref{defn:general *-cover !-cover} coincide with those defined in Definition \ref{defn:solid *-cover !-cover}. 
	However, when the 6-functor formalism on $(\Ccal,E)$ is the 6-functor formalism on $(\Aff_{\Zbb_{\square}},E)$, then the notion of small universal $\Dcal^!$-covers defined in Definition \ref{defn:general *-cover !-cover} does not coincide with that defined in Definition \ref{defn:solid space *-cover !-cover}. 
	In the rest of this paper, we will use the 6-functor formalism on $(\Shv_{\Dcal}(\Aff_{\Zbb_{\square}}),\widetilde{E})$ instead of that on $(\Aff_{\Zbb_{\square}},E)$, so it causes no issues.
\end{remark}

\begin{definition}
	For an object $S\in \Ccal$, we define a 2-category $\Kcal_{\Dcal,S}$ as follows:
	\begin{itemize}
		\item The objects of $\Kcal_{\Dcal,S}$ are the morphisms $X\to S$ in $E$.
		\item For objects $X,Y\in \Kcal_{\Dcal,S}$, the category of morphisms $\Fun_{\Kcal_{\Dcal,S}}(Y,X)$ is the homotopy category of $\Dcal(X\times_S Y)$.
		\item For objects $X,Y,Z\in \Kcal_{\Dcal,S}$, the composition 
		$$\Fun_{\Kcal_{\Dcal,S}}(Y,X) \times \Fun_{\Kcal_{\Dcal,S}}(Z,Y) \to \Fun_{\Kcal_{\Dcal,S}}(Z,X)$$ is given by
		$$\Dcal(X\times_S Y)\times \Dcal(Y\times_S Z)\to \Dcal(X\times_S Z);\; (M,N)\mapsto \pi_{13!}(\pi_{12}^*M\otimes \pi_{23}^*N),$$ 
		where $\pi_{ij}$ denote the projections from $X\times_S Y \times_S Z$.
	\end{itemize}
\end{definition}

\begin{remark}
	A straightforward verification that the above definition indeed defines a 2-category requires a lot of computations.
	However, in \cite{HM24,Zav23}, this is proved without such tedious computations using the theory of $(\infty,2)$-categories. (It is surprising that $(\infty,2)$-category theory can be used to prove that  $\Kcal_{\Dcal,S}$ is a 2-category!) In this paper, we do not use the fact that $\Kcal_{\Dcal,S}$ admits a structure of an $(\infty,2)$-category, so we omit the details.
\end{remark}

\begin{definition}[{\cite[Definition 4.4.1]{HM24}}]
For a morphism $f\colon X\to S$ in $E$ and an object $P\in \Dcal(X)$, $P$ is said to be \textit{$f$-suave} (resp. \textit{$f$-prim}) if $P$ is a left adjoint (resp. right adjoint) as a morphism $X\to S$ in $\Kcal_{\Dcal,S}$.
We let $\SD_f(P)$ (resp. $\PD_f(P)$) denote the associated right adjoint morphism (resp. left adjoint morphism) $S\to X$ in $\Kcal_{\Dcal,S}$.
\end{definition}

\begin{remark}
	In \cite{RC24}, $f$-suave objects (resp. $f$-prim objects) are called $f$-smooth objects (resp. $f$-proper objects).
\end{remark}

\begin{remark}\label{rem:suave and primand dualizable}
	An object $P\in \Dcal(X)$ is $\id$-suave (resp. $\id$-prim) if and only if it is dualizable in $\Dcal(X)$.
\end{remark}

\begin{definition}[{\cite[Definition 4.5.1]{HM24}}]
	A morphism $f\colon X\to S$ in $E$ is said to be \textit{$\Dcal$-suave} (resp. \textit{$\Dcal$-prim}) if the monoidal unit $\mathbf{1}_X\in \Dcal(X)$ is $f$-suave (resp. $f$-prim). 
	In this case, we write $\omega_f\coloneqq \SD_f(\mathbf{1}_X) \in \Dcal(X)$ (resp. $\delta_f\coloneqq \PD_f(\mathbf{1}_X) \in \Dcal(X)$), and we call it the \textit{dualizing complex} (resp. the \textit{codualizing complex}) of $f$.
	Moreover, the morphism $f\colon X\to S$ is said to be \textit{$\Dcal$-smooth} if it is $\Dcal$-suave and $\omega_f\in \Dcal(X)$ is invertible.
	If $\Dcal$ is clear from the context, we will drop it from the notation.
\end{definition}
\begin{remark}
	In \cite{RC24}, $\Dcal$-smooth morphisms are called cohomologically smooth morphisms.
\end{remark}

We list properties of suaveness and primness which will be needed later without proofs.

Let $f\colon X\to S$ be a morphism in $E$, and let $P\in \Dcal(X)$ be an object.
For every morphism $g\colon S^{\prime}\to S$, let $g^{\prime}\colon X^{\prime}=X\times_S S^{\prime} \to X$ and $f^{\prime}\colon X^{\prime}\to S^{\prime}$ denote the pullbacks of $g$ and $f$.
We let $\pi_i \colon X\times_S X\to X$ denote the $i$-th projection ($i=1,2$), and let $\Delta \colon X \to X\times_S X$ denote the diagonal morphism.

\begin{lemma}[{\cite[Lemma 4.4.5]{HM24}}]\label{lem:suave characterization}
	The following are equivalent:
	\begin{enumerate}
		\item The object $P$ is $f$-suave. 
		\item The natural morphism
		$$\pi_1^*\intHom(P,f^!\mathbf{1}_S)\otimes \pi_2^*P \to \intHom(\pi_1^*P,\pi_2^!P)$$
		is an equivalence.
		\item The natural morphism
		$$\pi_1^*\intHom(P,f^!\mathbf{1}_S)\otimes \pi_2^*P \to \intHom(\pi_1^*P,\pi_2^!P)$$
		becomes an equivalence after applying $\Hom(\mathbf{1}_X, \Delta^!(-))$. 
	\end{enumerate}
	In this case, we have 
	$$\SD_f(P) \simeq \intHom(P,f^!\mathbf{1}_S).$$
	Moreover, for every morphism $g\colon S^{\prime}\to S$ in $E$ and every object $M\in \Dcal(S^{\prime})$, the natural morphism
	$$g^{\prime*}\intHom(P,f^!\mathbf{1}_S)\otimes f^{\prime*}M \to \intHom(g^{\prime*}P,f^{\prime!}M)$$
	is an equivalence in $\Dcal(X\times_S S^{\prime})$.
\end{lemma}

\begin{corollary}[{\cite[Lemma 4.5.4, Corollary 4.5.11]{HM24}}]\label{cor:suave map characterization}
	The morphism $f\colon X\to S$ is $\Dcal$-suave if and only if the natural morphism
	$$\pi_1^*f^!\mathbf{1}_S\to \pi_2^!\mathbf{1}_X$$
	is an equivalence.
	In this case, we have an equivalence
	$$\omega_f\simeq f^!\mathbf{1}_S.$$
	Moreover, the natural morphisms
	\begin{align*}
	\omega_f\otimes f^*\to f^!,\quad f^* \to \intHom(\omega_f, f^!)
	\end{align*}
	are equivalences of functors from $\Dcal(S)$ to $\Dcal(X)$. 
	In particular, $f^!\colon \Dcal(S)\to \Dcal(X)$ preserves all small colimits, and therefore, $f_! \colon \Dcal(X) \to \Dcal(S)$ preserves compact objects.
\end{corollary}

\begin{lemma}[{\cite[Lemma 4.4.6]{HM24}}]\label{lem:prim characterization}
	The following are equivalent:
	\begin{enumerate}
		\item The object $P$ is $f$-prim. 
		\item The natural morphism
		$$f_!(\pi_{2*}\intHom(\pi_1^*P,\Delta_!\mathbf{1}_X)\otimes P)\to f_*\intHom(P,P)$$
		is an equivalence.
		\item The natural morphism
		$$f_!(\pi_{2*}\intHom(\pi_1^*P,\Delta_!\mathbf{1}_X)\otimes P)\to f_*\intHom(P,P)$$
		becomes an equivalence after applying $\Hom(\mathbf{1}_S, -)$. 
	\end{enumerate}
	In this case, we have 
	$$\PD_f(P) \simeq \pi_{2*}\intHom(\pi_1^*P,\Delta_!\mathbf{1}_X).$$
	Moreover, for every morphism $g\colon S^{\prime}\to S$ in $E$ and every object $M\in \Dcal(X\times_S S^{\prime})$, the natural morphism
	$$f^{\prime}_!(g^{\prime*}\pi_{2*}\intHom(\pi_1^*P,\Delta_!\mathbf{1}_X)\otimes M)\to f^{\prime}_*\intHom(g^{\prime*}P,M)$$
	is an equivalence in $\Dcal(S^{\prime})$.
\end{lemma}

\begin{corollary}[{\cite[Lemma 4.5.5, Corollary 4.5.11]{HM24}}]\label{cor:prim map characterization}
	The morphism $f\colon X\to S$ is $\Dcal$-prim if and only if the natural morphism
	$$f_!\pi_{2*}\Delta_!\mathbf{1}_X\to f_*\mathbf{1}_X$$
	is an equivalence.
	In this case, we have an equivalence
	$$\delta_f\simeq \pi_{2*}\Delta_!\mathbf{1}_X.$$
	Moreover, the natural morphisms
	$$f_!(\delta_f\otimes -)\to f_*, \quad f_!\to f_*\intHom(\delta_f,-)$$
	are equivalences of functors from $\Dcal(X)$ to $\Dcal(S)$. 
\end{corollary}

\begin{lemma}[{\cite[Lemma 4.4.8]{HM24}}]\label{lem:prim and suave obj D-local}
	\begin{enumerate}
		\item If $P$ is $f$-suave (resp. $f$-prim), then for every morphism $g\colon S^{\prime}\to S$, $P^{\prime}=g^{\prime*}P$ is $f^{\prime}$-suave (resp. $f^{\prime}$-prim).
		In this case, there is a natural equivalence $\SD_{f^{\prime}}(P^{\prime})\simeq g^{\prime*}\SD_f(P)$ (resp. $\PD_{f^{\prime}}(P^{\prime})\simeq g^{\prime*}\PD_f(P)$).
		\item Let $\Ucal\subset \Ccal_{/S}$ be a universal $\Dcal^*$-cover.
		If for every morphism $g\colon S^{\prime}\to S$ in $\Ucal$, the object $P^{\prime}=g^{\prime*}P$ is $f^{\prime}$-suave (resp. $f^{\prime}$-prim), then $P$ is $f$-suave (resp. $f$-prim).
	\end{enumerate}
\end{lemma}
\begin{proof}
	The equivalence $\SD_{f^{\prime}}(P^{\prime})\simeq g^{\prime*}\SD_f(P)$ is not explicitly written in \cite[Lemma 4.4.8]{HM24}, however it easily follows from the proof of this lemma.
\end{proof}

\begin{lemma}[{\cite[Lemma 4.4.9]{HM24}}]\label{lem:suave and prim composition}
Let $f\colon X\to S$ and $g\colon Y\to X$ be morphisms in $E$.
\begin{enumerate}
	\item Let $P\in \Dcal(X)$ be a $f$-suave object, and $Q\in \Dcal(Y)$ be a $g$-suave object.
	Then $g^*P\otimes Q$ is $(f\circ g)$-suave, and we have a natural equivalence
	$$\SD_{f\circ g}(g^*P\otimes Q)\simeq g^*\SD_f(P)\otimes \SD_g(Q).$$
	\item Let $P\in \Dcal(Y)$ be a $(f\circ g)$-suave object, and $Q\in \Dcal(Y)$ be a $g$-prim object.
	Then $g_!(Q\otimes P)$ is $f$-suave, and we have a natural equivalence
	$$\SD_f(g_!(Q\otimes P)) \simeq g_!(\PD_{g}(Q)\otimes \SD_{f\circ g}(P)).$$
\end{enumerate}
The same is true with ``suave'' and ``prim'' swapped (and $\PD$ replaced with $\SD$).
\end{lemma}

\begin{corollary}[{\cite[Lemma 4.5.7, Lemma 4.5.9]{HM24}}]\label{cor:prim and suave map D-local}
	\begin{enumerate}
		\item Let $f\colon Y\to X$ and $g\colon X\to S$ be $\Dcal$-suave (resp. $\Dcal$-smooth, $\Dcal$-prim) morphisms.
		Then $g\circ f\colon Y\to S$ is also $\Dcal$-suave (resp. $\Dcal$-smooth, $\Dcal$-prim) and its dualizing complex (resp. codualizing complex) is given by $g^*\omega_f\otimes \omega_g$ (resp. $g^*\delta_f\otimes \delta_g$).
		\item If the morphism $f$ is $\Dcal$-suave (resp. $\Dcal$-smooth, $\Dcal$-prim), then for every morphism $g\colon S^{\prime}\to S$, the pullback $f^{\prime}\colon X^{\prime}=X\times_S S^{\prime}\to S^{\prime}$ of $f$ is also $\Dcal$-suave (resp. $\Dcal$-smooth, $\Dcal$-prim) and its dualizing complex (resp. codualizing complex) is given by $g^{\prime *}\omega_f$ (resp. $g^{\prime *}\delta_f$).
		\item Let $\Ucal\subset \Ccal_{/S}$ be a universal $\Dcal^*$-cover.
		If for every morphism $g\colon S^{\prime}\to S$ in $\Ucal$, the pullback $f^{\prime}\colon X^{\prime}=X\times_S S^{\prime}\to S^{\prime}$ of $f$ is $\Dcal$-suave (resp. $\Dcal$-smooth, $\Dcal$-prim), then $f$ is also $\Dcal$-suave (resp. $\Dcal$-smooth, $\Dcal$-prim).
	\end{enumerate}
\end{corollary}
% \begin{proof}
% 	Only (3) for $\Dcal$-smooth morphisms is not written explicitly in \cite{HM24}.
% 	However it easily follows from (2) and the fact that the invertibility of the dualizing complex $\omega_f$ can be checked universally $\Dcal^*$-locally.
% \end{proof}

\begin{lemma}[{\cite[Corollary 4.5.18]{HM24}}]\label{lem:suave diagonal}
Let $f\colon X\to S$ be a morphism in $E$. 
Suppose that the diagonal morphism $\Delta_f\colon X\to X\times_S X$ of $f$ is $\Dcal$-suave.
Then, every $f$-suave object $P\in \Dcal(X)$ is dualizable and $\SD_f(P)\cong P^{\vee}\otimes \omega_{\Delta_f}^{-1}$, where $P^{\vee}$ is the dual of $P$ in $\Dcal(X)$.

The same is true with ``prim'' in place of ``suave'', and with $\SD$ (resp.\ $\omega$) replaced by $\PD$ (resp.\ $\delta$).
\end{lemma}

\begin{lemma}[{\cite[Lemma 4.4.10, Lemma 4.5.8]{HM24}}]\label{lem:suave suave-local}
	Let $f\colon X\to S$ be a morphism in $E$, and let $\{f_i\colon X_i\to X\}_i$ be a universal $\Dcal^*$-cover in $\Ccal_E$.
	Let $\Ucal$ denote a sieve in $(\Ccal_E)_{/X}$ generated by $\{f_i\colon X_i\to X\}_i$.
	\begin{enumerate}
		\item We assume that for every $i$, both $f_i\colon X_i\to X$ and $f\circ f_i\colon X_i\to S$ are $\Dcal$-suave (resp. $\Dcal$-smooth), and that $\Dcal(X)$ admits $\Ucal$-indexed colimits.
		Then, the morphism $f\colon X\to S$ is also $\Dcal$-suave (resp. $\Dcal$-smooth).
		\item 
		We assume that for every $i$, both $f_i\colon X_i\to X$ and $f\circ f_i\colon X_i\to S$ are $\Dcal$-prim, and that $\Dcal(X)$ has $\Ucal^{\op}$-indexed limits and $f_!$ preserves them.
		Then, the morphism $f\colon X\to S$ is also $\Dcal$-prim.
	\end{enumerate}
\end{lemma}
\begin{proof}
	The claim for $\Dcal$-smooth morphisms is not explicitly written in \cite{HM24}, so we explain the proof of it.
	For every $i$, we have $f_i^*\omega_f \simeq \intHom(\omega_{f_i},f_i^!\omega_f)$ by Corollary \ref{cor:suave map characterization}.
	Since $f_i$ and $f\circ f_i$ are $\Dcal$-smooth, $\omega_{f_i}$ and $\omega_{f\circ f_i}\simeq f_i^!\omega_f$ are invertible.
	Therefore, $f_i^*\omega_f$ is invertible for every $i$.
	Since $\{f_i\colon X_i\to X\}_i$ is a universal $\Dcal^*$-cover, we find that $\omega_f$ is invertible.
\end{proof}
\begin{remark}\label{rem:mono cech cover finite}
	In (2), it is hard to check the condition that $f_!$ preserves $\Ucal^{\op}$-indexed limits in general.
	However, if $\{f_i\colon X_i\to X\}_i$ is a finite family of monomorphisms (i.e., the diagonals become equivalences), then $\Ucal^{\op}$ admits a finite cofinal subcategory by \cite[Lemma A.4.7]{HM24}.
	In this case, we only need to check that $f_!$ preserves finite limits, which is trivial in practice.
	For example, this condition is automatically satisfied in the 6-functor formalism for solid quasi-coherent sheaves on solid $\Dcal$-stacks.
\end{remark}

% \begin{definition}[{\cite[Definition 3.1.25]{RC24}}]
% 	Let $f\colon Y\to X$ be a morphism in $E$.
% 	Then the morphism $f$ is called \textit{descendable $\Dcal$-cover} if it satisfies the following conditions:
% 	\begin{enumerate}
% 		\item The morphism $f$ is $\Dcal$-prim.
% 		\item The morphism $f$ is a canonical cover.
% 		\item The $\Ebb_{\infty}$-algebra $f_{*}\mathbf{1}_Y$ in $\Dcal(X)$ is descendable in the sense of \cite[Definition 3.18]{Mat16}.
% 	\end{enumerate}
% \end{definition}

\begin{lemma}[{\cite[Lemma 4.7.4, Corollary 4.7.5]{HM24}}]\label{lem:prim descendable local}
We assume that the 6-functor formalism $\Dcal \colon \Corr(\Ccal,E)\to \Cat_{\infty}$ factors through the subcategory of $\Cat_{\infty}$ consisting of stable $\infty$-categories with exact functors (such a 6-functor formalism is called a \textit{stable 6-functor formalism} in \cite[Definition 3.1.1]{HM24}).
Let $f\colon Y\to X$ be a morphism in $E$.
We assume that $f$ is $\Dcal$-prim and the $\Ebb_{\infty}$-algebra $f_{*}\mathbf{1}_Y$ in $\Dcal(X)$ is descendable in the sense of \cite[Definition 3.18]{Mat16}.
\begin{enumerate}
	\item The morphism $f$ is a universal $\Dcal^*$-cover and a universal $\Dcal^!$-cover.
	\item Let $g\colon X\to S$ be a morphism in $E$. 
	If $g\circ f$ is $\Dcal$-prim, then $g$ is also $\Dcal$-prim.
\end{enumerate}
\end{lemma}

The following proposition plays an important role in the proof of the main theorem.
\begin{proposition}\label{prop:suave criterion}
	Let $f\colon X\to S$ be a morphism in $E$, and $X\overset{g_i}{\longrightarrow}X_i\overset{f_i}{\longrightarrow} S$ be factorizations of $f$ such that $f_i$ and $g_i$ lie in $E$.
	We assume the following:
	\begin{enumerate}
		\item The morphisms $f_i\colon X_i\to S$ are $\Dcal$-suave, and the morphisms $g_i\colon X\to X_i$ are $\Dcal$-prim.
		\item The family of functors $\{(g_i\times \id_X)_!\colon \Dcal(X\times_S X)\to \Dcal(X_i\times_S X)\}_i$ is conservative.
		\item For every $i$, $g_{i!}\delta_{g_i}$ is $f_i$-suave.
	\end{enumerate}
Then the morphism $f\colon X\to S$ is $\Dcal$-suave.
\end{proposition}
It essentially follows from the following proposition.
\begin{proposition}[{\cite[Proposition D.2.9]{HM24}}]\label{prop:suave criterion general}
Let $\Ecal$ be a 2-category, let $f\colon X\to S$ be a morphism in $\Ecal$, and let $\{l_i \colon X_i \rightleftarrows X \colon r_i\}_i$ be a family of adjunctions in $\Ecal$.
We assume the following:
\begin{enumerate}
	\item The family of functors $\{r_{i*}\colon \Fun_{\Ecal}(X,X)\to\Fun_{\Ecal}(X,X_i)\}_i$ is conservative.
	\item For $Z\in \{X,S\}$, the functor $f_*\colon \Fun_{\Ecal}(Z,X) \to \Fun_{\Ecal}(Z,S)$ is  functor.
	\item For every $i$, the composition $f\circ l_i$ is a left adjoint morphism.
\end{enumerate}	
Then $f\colon X\to S$ is a left adjoint morphism.
\end{proposition}

\begin{lemma}\label{lem:preserving adjoint}
Let $f \colon X\to S$ and $g\colon Y\to X$ be morphisms in $E$.
Let $\Gamma \colon Y\to X\times_S Y$ denote the graph morphism of $g$.
We assume that $g$ is $\Dcal$-prim.
Then $\Gamma_!\mathbf{1}_Y\in \Dcal(X\times_S Y)$ is a right adjoint as a morphism $Y\to X$ in $\Kcal_{\Dcal,S}$, and its left adjoint is a morphism $X\to Y$ in $\Kcal_{\Dcal,S}$ corresponding to $\Gamma_!\delta_g\in \Dcal(X\times_S Y)$.
\end{lemma}
\begin{proof}
	It follows from \cite[Theorem 4.2.4(iii)]{HM24}.
\end{proof}
\begin{proof}[Proof of Proposition \ref{prop:suave criterion}]
Let $\Gamma_i \colon X \to X_i\times_S X$ be the graph morphism of $g_i$.
By Lemma \ref{lem:preserving adjoint}, $\Gamma_{i!}\mathbf{1}_X\in \Dcal(X_i\times_S X)$ is a right adjoint as a morphism $X\to X_i$ in $\Kcal_{\Dcal,S}$, and its left adjoint is a morphism $X_i\to X$ in $\Kcal_{\Dcal,S}$ corresponding to $\Gamma_{i!}\delta_{g_i}\in \Dcal(X_i\times_S X)$.
We apply Proposition \ref{prop:suave criterion general} for $\Ecal=\Kcal_{\Dcal,S}$, $f=\mathbf{1}_X\colon X\to S$, $l_i=\Gamma_{i!}\delta_{g_i} \colon X_i\to X$, and $r_i=\Gamma_{i!}\mathbf{1}_X \colon X\to X_i$.
We check the three conditions in Proposition \ref{prop:suave criterion general}.

First, we compute $r_{i*}\colon \Dcal(X\times_S X)\to \Dcal(X_i\times_S X)$.
Let $p_{kl}$ denote the projections from $X_i\times_S X\times_S X$ ($1\leq k<l\leq 3$), and $p_k$ denote the projections from $X\times_S X$ ($k=1,2$). 
Then $r_{i*}$ is given by $p_{13!}(p_{23}^*(-)\otimes p_{12}^*\Gamma_{i!}\mathbf{1}_X)$.
We have equivalences 
\begin{align*}
	&p_{13!}(p_{23}^*(-)\otimes p_{12}^*\Gamma_{i!}\mathbf{1}_X)\\
	\simeq& p_{13!}(p_{23}^*(-)\otimes (\Gamma_{i}\times \id)_!p_1^*\mathbf{1}_X)\\
	\simeq& p_{13!}(p_{23}^*(-)\otimes (\Gamma_{i}\times \id)_!\mathbf{1}_{X\times_S X})\\
	\simeq& p_{13!}(\Gamma_{i}\times \id)_!(\Gamma_{i}\times \id)^*p_{23}^*(-)\\
	\simeq& (g_i\times \id)_!(-),
\end{align*}
where the first equivalence follows from the proper base change, and the third equivalence follows from the projection formula.
From the above, we get an equivalence $r_{i*} \simeq (g_i\times \id)_!$, and therefore, we find that the family of functors $\{r_{i*}\colon \Dcal(X\times_S X)\to \Dcal(X_i\times_S X)\}_i$ is conservative.

Next, we check the condition (2).
If $Z=S$, then the functor $$f_*\colon \Fun_{\Ecal}(S,X) \to \Fun_{\Ecal}(S,S)$$ is given by $f_!\colon \Dcal(X)\to \Dcal(S)$, and it is a left adjoint functor.
If $Z=X$, then by the same computation as for $r_{i*}$, we find that the functor $$f_*\colon \Fun_{\Ecal}(X,X) \to \Fun_{\Ecal}(X,S)$$ is given by $p_{2!}\colon \Dcal(X\times_S X)\to \Dcal(X)$, and it is also a left adjoint functor.

Finally, we compute $f\circ l_i$.
Let $q\colon X\times_S X_i \to X_i$ denote the projection.
Then, by definition, $f\circ l_i \colon X_i\to S$ in $\Kcal_{\Dcal,S}$ corresponds to $q_!\Gamma_{i!}\delta_{g_i}\in \Dcal(X_i)$.
It is equivalent to $g_{i!}\delta_{g_i}$, which is $f_i$-suave.
Therefore, $f\circ l_i \colon X_i\to S$ is a left adjoint morphism in $\Kcal_{\Dcal,S}$.
\end{proof}

Next, we consider \'{e}taleness and properness of morphisms of solid $\Dcal$-stacks.
\begin{definition}[{\cite[Definition 4.6.1]{HM24}}]\label{def:etale}
	Let $f\colon Y\to X$ be a truncated morphism of solid $\Dcal$-stacks in $\widetilde{E}$ (Definition \ref{def:solid D-stack 6ff}).
	Let $\Delta_f \colon Y\to Y\times_X Y$ denote the diagonal morphism of $f$.
	The morphism $f$ is said to be \textit{$\Dcal$-\'{e}tale} if it is $\Dcal$-suave and if the diagonal morphism $\Delta_f$ is $\Dcal$-\'{e}tale or an equivalence.
\end{definition}
\begin{remark}
	By the definition of truncated morphisms, if $f$ is $n$-truncated then $\Delta_f$ is $(n-1)$-truncated.
	Therefore, the above definition is well-defined.
\end{remark}

\begin{remark}\label{rem:D-etale D-smooth}
	Let $f\colon Y\to X$ be a $-1$-truncated $\Dcal$-\'{e}tale morphism. 
	The diagonal morphism $\Delta_f \colon Y \to Y\times_X Y$ is an equivalence, so we have natural equivalences $\Delta_{f}^*\simeq\Delta_{f}^!$ and $\omega_{\Delta_f}\simeq \Ocal_Y$.
	By Lemma \ref{lem:suave diagonal} and Corollary \ref{cor:suave map characterization}, we get natural equivalences $\omega_f\simeq \Ocal_Y$ and $f^*\simeq f^!$.

	For a general $\Dcal$-\'{e}tale morphism $f\colon Y\to X$, we get a natural equivalence $\omega_f\simeq \Ocal_Y$ and $f^*\simeq f^!$ by induction.
	In particular, $\Dcal$-\'{e}tale morphisms are $\Dcal$-smooth.
\end{remark}

\begin{lemma}[{\cite[Lemma 3.2.9]{RC24}}]\label{lem:open etale}
	Let $f\colon \AnSpec\Bcal\to \AnSpec \Acal$ be an open immersion in the associated locale in $\Aff_{\Zbb_{\square}}$.
	Then it is $\Dcal$-\'{e}tale.
\end{lemma}

Morphisms that we want to be \'{e}tale are truncated, so the above definition works well.
However, there are a lot of non-truncated morphisms that we want to be proper.
For example, a closed immersion $\AnSpec (\Zbb[T]/T)_{\square}\to \AnSpec(\Zbb[T],\Zbb[T])_{\square}$ is not truncated. In fact, almost all Zariski closed immersions are not truncated.
Therefore, a definition similar to Definition \ref{def:etale} is not appropriate for our purpose.
Instead, we introduce the following notion.

\begin{definition}
	Let $f\colon Y\to X$ be a morphism of solid $\Dcal$-stacks in $\widetilde{E}$.
	The morphism $f$ is said to be \textit{weakly $\Dcal$-proper} if it is $\Dcal$-prim and if the codualizing complex $\delta_f$ is (non-canonically) equivalent to the monoidal unit $\Ocal_Y\in \Dcal(Y)$.
\end{definition}

\begin{remark}
	We note that our definition is slightly different from the definition in \cite[Definition 3.1.19]{RC24}.
	In loc. cit., it is imposed that the diagonal morphism $\Delta_f \colon Y\to Y\times_X Y$ is also $\Dcal$-prim and its codualizing complex is (non-canonically) equivalent to $\Ocal_Y\in \Dcal(Y)$.
\end{remark}
\begin{remark}
	This definition is rather ad hoc.
	For example, while we have an equivalence $f_*\simeq f_!$ for a weakly $\Dcal$-proper morphism $f\colon Y\to X$, this equivalence depends on the choice of $\delta_f\simeq \Ocal_Y$.
	However, for a morphism $f$ in $P$, we can take a natural identification $\delta_f\simeq \Ocal_Y$ compatible with the natural identification $f_!\simeq f_*$ which comes from the construction of the 6-functor formalism, see Remark \ref{rem:identification}.
	Weakly $\Dcal$-proper morphisms considered in this paper are all explicitly constructed from $P$, and for these, it can be seen from the proofs that there are natural trivializations of the codualizing complexes.
\end{remark}
\begin{remark}
	In \cite{CLL25}, $(\infty,2)$-categorical 6-functor formalisms are studied. 
	This 6-functor formalism contains the data of $2$-morphisms $f_!\to f_*$ for nice morphisms $f$, and therefore, we can define the $\Dcal$-proper morphisms as those $f$ such that $f_!\to f_*$ is an equivalence. 
	Applying this framework to our setting would make certain arguments more natural, but we will not pursue this direction in the present paper.
\end{remark}

\begin{lemma}\label{lem:proper composition basechange}
	\begin{enumerate}
		\item Let $f\colon Y\to X$ and $g\colon X\to S$ be weakly $\Dcal$-proper morphisms.
		Then $g\circ f\colon Y\to S$ is also weakly $\Dcal$-proper.
		\item If the morphism $f$ is weakly $\Dcal$-proper, then for every morphism $g\colon S^{\prime}\to S$, the pullback $f^{\prime}\colon X^{\prime}=X\times_S S^{\prime}\to S^{\prime}$ of $f$ is also weakly $\Dcal$-proper.
	\end{enumerate}
\end{lemma}
\begin{proof}
	It easily follows from Corollary \ref{cor:prim and suave map D-local}.
\end{proof}

\begin{lemma}\label{lem:p is prim}
	Let $f\colon Y\to X$ be a morphism of solid affinoid spaces in $P$ (Definition \ref{def:suitable dcp}).
	Then the morphism $f$ is weakly $\Dcal$-proper.
\end{lemma}
\begin{proof}
	Since the diagonal morphism $\Delta \colon Y\to Y\times_X Y$ also lies in $P$, we have natural identifications $\Delta_*=\Delta_!$ and $f_*=f_!$ by construction.
	Therefore, $f$ is $\Dcal$-prim and the codualizing complex $\delta_f$ is equivalent to $\pi_{2*}\Delta_!\Ocal_Y\simeq \pi_{2*}\Delta_*\Ocal_Y=\Ocal_Y$ by Corollary \ref{cor:prim map characterization}.
\end{proof}
\begin{remark}\label{rem:identification}
	For a morphism $g$ in $P$, let $\iota_g \colon g_!\simeq g_*$ denote the natural equivalence which comes from the construction of the 6-functor formalism.
	For a morphism $f\colon Y\to X$ in $P$, three identifications of $f_*$ and $f_!$ exist.
	One is the equivalence $\iota_f\colon f_*\simeq f_!$.
	The others come from the identification $\delta_f\simeq \pi_{2*}\Delta_!\Ocal_Y\overset{\iota_{\Delta}}{\simeq} \pi_{2*}\Delta_*\Ocal_Y\simeq \Ocal_Y$ and the equivalences $f_!(\delta_f\otimes -)\simeq f_*$ and $f_!\simeq f_*\intHom(\delta_f,-)$ in Corollary \ref{cor:prim map characterization}, respectively. 
	These identifications coincide.
	Let us prove it.
	By unwinding the construction in \cite[Lemma 4.4.17]{HM24}, the equivalence $f_!(\delta_f\otimes -)\simeq f_*$ is constructed from the following morphism via adjunction:
	\begin{align*}
		&f^*f_!(\pi_{2*}\Delta_!\Ocal_Y\otimes -)\\
		\to &\pi_{1!}\pi_2^*(\pi_{2*}\Delta_!\Ocal_Y\otimes -)\\
		\to &\pi_{1!}(\pi_2^*\pi_{2*}\Delta_!\Ocal_Y\otimes \pi_2^*-)\\
		\to &\pi_{1!}(\Delta_!\Ocal_Y\otimes \pi_2^*-)\\
		\to &\pi_{1!}\Delta_!(\Ocal_Y\otimes \Delta^*\pi_2^*-)\\
		\simeq &-,
	\end{align*}
	where the first morphism follows from the proper base change, and the fourth morphism follows from the projection formula.
	By the construction of the 6-functor formalism, the proper base change and the projection formula above are induced by the adjunction under the equivalence $\iota_g\colon g_!\simeq g_*$ where $g=f,\pi_1,\Delta$.
	Therefore, we find that the equivalence $\iota_f\colon f_!\simeq f_*$ coincides with the equivalence induced from $f_!(\delta_f\otimes -)\simeq f_*$ and $\delta_f\simeq \Ocal_Y$ by standard computations of adjunctions.
	In a similar way, we can treat the equivalence induced from $f_!\simeq f_*\intHom(\delta_f,-)$.
\end{remark}

\begin{lemma}\label{lem:weak proper D-local}
	Let $f\colon Y\to X$ be a morphism of solid $\Dcal$-stacks in $\widetilde{E}$.
	We assume that $f$ is, universally $\Dcal^*$-locally on $X$, representable by a morphism in $P$, that is, there exists a universal $\Dcal^*$-cover $\{X_i\to X\}_i$ with $X_i\in \Aff_{\Zbb_{\square}}$ such that the pullback $f_i\colon Y_i=Y\times_X X_i \to X_i$ of $f$ along $X_i\to X$ lies in $P$ for every $i$.
	Then there is a natural equivalence $f_!\simeq f_*$ of functors from $\Dcal(Y)$ to $\Dcal(X)$.
\end{lemma}
\begin{proof}
	Let $\{X_i\to X\}_i$ be a universal $\Dcal^*$-cover as above.
	Let $\Ucal \subset (\Aff_{\Zbb_{\square}})_{/X}$ be the sieve generated by $\{X_i\to X\}_i$.
	Then we have natural equivalences
	\begin{align*}
	\Dcal(X)\simeq \varprojlim_{(X^{\prime}\to X)\in \Ucal}\Dcal(X^{\prime}), \;
	\Dcal(Y)\simeq \varprojlim_{(X^{\prime}\to X)\in \Ucal}\Dcal(Y\times_X X^{\prime}).
	\end{align*}
	Let $f_{X^{\prime}}\colon Y\times_X X^{\prime} \to X^{\prime}$ denote the pullback of $f$ along $(X^{\prime}\to X)\in\Ucal$.
	Then we have an equivalence $$f_!\simeq \varprojlim_{(X^{\prime}\to X)\in \Ucal} (f_{X^{\prime}})_! \colon \Dcal(Y)\simeq \varprojlim_{(X^{\prime}\to X)\in \Ucal}\Dcal(Y\times_X X^{\prime}) \to \Dcal(X)\simeq \varprojlim_{(X^{\prime}\to X)\in \Ucal}\Dcal(X^{\prime}),$$
	where we note that $\{(f_{X^{\prime}})_!\}_{(X^{\prime}\to X)\in \Ucal}$ preserves cocartesian sections by the proper base change.
	On the other hand, we have an equivalence 
	$$f_*\simeq \varprojlim_{(X^{\prime}\to X)\in \Ucal} (f_{X^{\prime}})_* \colon \Dcal(Y)\simeq \varprojlim_{(X^{\prime}\to X)\in \Ucal}\Dcal(Y\times_X X^{\prime}) \to \Dcal(X)\simeq \varprojlim_{(X^{\prime}\to X)\in \Ucal}\Dcal(X^{\prime}),$$
	where we note that $\{(f_{X^{\prime}})_*\}_{(X^{\prime}\to X)\in \Ucal}$ preserves cocartesian sections since $f_{X^{\prime}}\in P$.
	By the construction of $\Dcal\colon\Corr(\Shv_{\Dcal}(\Aff_{\Zbb_{\square}}),\widetilde{E})\to \Pr^{L,\ex}$, the diagrams over $\Delta^1\times\Ucal$
	$$\{(f_{X^{\prime}})_!\colon \Dcal(Y\times_X X^{\prime}) \to \Dcal(X^{\prime})\}_{(X^{\prime}\to X)\in \Ucal}$$
	and 
	$$\{(f_{X^{\prime}})_*\colon \Dcal(Y\times_X X^{\prime}) \to \Dcal(X^{\prime})\}_{(X^{\prime}\to X)\in \Ucal}$$
	are equivalent.
	From this, we get an equivalence $f_!\simeq f_* \colon \Dcal(Y)\to \Dcal(X)$.
\end{proof}

\begin{corollary}[{\cite[Lemma 3.1.21]{RC24}}]\label{cor:weakly proper}
	Let $f\colon Y\to X$ be a morphism of solid $\Dcal$-stacks in $\widetilde{E}$.
	We assume that $f$ is, universally $\Dcal^*$-locally on $X$, representable by a morphism in $P$.
	Then $f$ is weakly $\Dcal$-proper.
\end{corollary}
\begin{proof}
	By Corollary \ref{cor:prim and suave map D-local} and Lemma \ref{lem:p is prim}, the morphism $f$ is $\Dcal$-prim.
	Since the diagonal $\Delta \colon Y\to Y\times_X Y$ is, universally $\Dcal^*$-locally on $Y\times_X Y$, also representable by a morphism in $P$, there is an equivalence $\Delta_*\simeq \Delta_!$ by Lemma \ref{lem:weak proper D-local}.
	Therefore, we get an equivalence $\delta_f\simeq \Ocal_Y$ by Corollary \ref{cor:prim map characterization}.
\end{proof}

\subsection{Examples}

In this subsection, we examine the relationship between the notions of morphisms of analytic adic spaces and the notions of morphisms defined via the 6-functor formalism on solid $\Dcal$-stacks.

Let $\AnAdic$ denote the category of analytic adic spaces (in Huber's sense).
Then, we can construct a natural functor 
$$\AnAdic \to \Shv_{\Dcal}(\Aff_{Z_{\square}});\; X\mapsto X_{\square}$$
as follows:
Let $\AffAnAdic$ be the site of analytic affinoid adic spaces with the analytic topology.
Then we have a natural fully faithful functor $\AnAdic \to \Shv(\AffAnAdic)$, where $\Shv(\AffAnAdic)$ is the $\infty$-category of sheaves of anima on $\AffAnAdic$.
The functor $\AffAnAdic \to \Aff_{Z_{\square}} ;\; \Spa(A,A^+)\mapsto \AnSpec(A,A^+)_{\square}$ is a continuous functor. 
In fact, for an open cover $\{\Spa(A_i,A_i^+)\to \Spa(A,A^+) \}_i$, $\{\AnSpec(A_i,A_i^+)_{\square}\to \AnSpec(A,A^+)_{\square}\}_i$ is a $\Dcal$-cover by \cite[Lemma 3.2.9]{RC24}, and for every morphism $\Spa(B,B^+) \to \Spa(A,A^+)$, the natural morphism
$$\AnSpec(B_i,B_i^+)_{\square} \to \AnSpec(B,B^+)_{\square}\times_{\AnSpec(A,A^+)_{\square}}\AnSpec(A_i,A_i^+)_{\square}$$
is an equivalence by \cite[Proposition 4.12]{And21}, where $$\Spa(B_i,B_i^+)=\Spa(B,B^+)\times_{\Spa(A,A^+)}\Spa(A_i,A_i^+).$$
Therefore, we get a colimit-preserving functor 
$$\Shv(\AffAnAdic) \to \Shv_{\Dcal}(\Aff_{Z_{\square}});\; \Spa(A,A^+)\mapsto \AnSpec(A,A^+)_{\square}.$$
By composing with the functor $\AnAdic \to \Shv(\AffAnAdic)$, we obtain a functor $\AnAdic \to \Shv_{\Dcal}(\Aff_{Z_{\square}});\; X\mapsto X_{\square}$.

\begin{remark}\label{rem:explicit description}
	For an analytic adic space $X$, we have an explicit description 
	$$X_{\square}=\varinjlim_{\Spa(A,A^+)\hookrightarrow X} \AnSpec(A,A^+)_{\square},$$
	where the colimit is taken over all open immersions $\Spa(A,A^+)\hookrightarrow X$.
\end{remark}

\begin{remark}
	The author does not know whether the above functor is fully faithful. The restriction to the full subcategory of affinoid analytic adic spaces is fully faithful by \cite[Proposition 3.34]{And21}. However, it seems difficult to extend it. One of the difficulties is that ``open covers of $\AnSpec(A,A^+)_{\square}$ in $\Shv_{\Dcal}(\Aff_{Z_{\square}})$'' cannot necessarily be refined by affinoid open covers of $\Spa(A,A^+)$.
\end{remark}

We note that the functor $X\mapsto X_{\square}$ does not necessarily preserve fiber products.
However, we have the following lemma.

\begin{lemma}\label{lem:fiber product commute}
	Let $X\to S \leftarrow Y$ be a diagram in $\AnAdic$.
	We assume that for any point $x \in X \times_S Y$ and an open neighborhood $U\subset X \times_S Y$ of $x$, there exists an affinoid open neighborhood $\Spa(A,A^+)\times_{\Spa(R,R^+)}\Spa(B,B^+)\subset U$ of $x$ such that the derived tensor product $(A,A^+)_{\square}\otimes_{(R,R^+)_{\square}}(B,B^+)_{\square}$ is static.
	Then the natural morphism $(X \times_S Y)_{\square}\to X_{\square}\times_{S_{\square}} Y_{\square}$ is an equivalence.
\end{lemma}
\begin{proof}
It easily follows from Remark \ref{rem:explicit description}.
\end{proof}

\begin{example}
	Let $X\to S \leftarrow Y$ be a diagram in $\AnAdic$ such that $Y\to S$ is an open immersion.
	Then the natural morphism $(X \times_S Y)_{\square}\to X_{\square}\times_{S_{\square}} Y_{\square}$ is an equivalence.
\end{example}

\begin{lemma}\label{lem:open cov D-cov}
	Let $X$ be an analytic adic space, and $\{X_i\}_i$ be an open cover of $X$ by affinoid adic spaces.
	Then $\{X_{i\square} \to X_{\square}\}_i$ is a universal $\Dcal^*$-cover.
\end{lemma}
\begin{proof}
	By construction, $\{X_{i\square} \to X_{\square}\}_i$ is a canonical cover in $\Shv_{\Dcal}(\Aff_{Z_{\square}})$, so it is a universal $\Dcal^*$-cover by Remark \ref{rem:canonical cover *-cover}.
\end{proof}

\begin{lemma} [{\cite[Lemma 3.2.9]{RC24}}]\label{lem:open imm etale}
Let $X$ be an analytic adic space, and $U\subset X$ be an open subspace of $X$.
Then the inclusion $U_{\square}\hookrightarrow X_{\square}$ is $\Dcal$-\'{e}tale.
\end{lemma}
\begin{proof}
	Since the diagonal morphism of $U_{\square}\hookrightarrow X_{\square}$ is an equivalence, it suffices to show that $U_{\square} \hookrightarrow X_{\square}$ is $\Dcal$-suave.
	By Corollary \ref{cor:prim and suave map D-local} and Lemma \ref{lem:open cov D-cov}, we may assume that $X$ is an affinoid analytic adic space.
	Next, we take an affinoid open cover $\{U_j\}$ of $U$.
	By \cite[Lemma 3.2.9]{RC24}, the inclusion $U_{j\square}\hookrightarrow X_{\square}$ is $\Dcal$-\'{e}tale.
	Moreover, $\{U_{j\square} \to U_{\square}\}_i$ is a universal $\Dcal^*$-cover by Lemma \ref{lem:open cov D-cov}.
	Therefore, by Lemma \ref{lem:suave suave-local}, it suffices to show that $U_{j\square} \to U_{\square}$ is $\Dcal$-suave.
	By the same reason as above, it is enough to show that for an affinoid open subspace $V\subset U$, $(U_j\times_U V)_{\square} \to V_{\square}$ is $\Dcal$-suave.
	Since $U_j\times_U V \cong U_j\times_X V$ is an affinoid analytic adic space, $(U_j\times_U V)_{\square} \to V_{\square}$ is $\Dcal$-suave by \cite[Lemma 3.2.9]{RC24}.
\end{proof}	

Next, we consider \'{e}tale morphisms of locally noetherian adic spaces, where locally noetherian adic spaces are adic spaces covered by analytic affinoid adic spaces $\Spa(A,A^+)$ such that $A$ is a strongly noetherian complete Tate algebra.

\begin{proposition}\label{prop:etale static}
	Let $X=\Spa(A,A^+)$ be an analytic affinoid adic space such that $A$ is a strongly noetherian complete Tate algebra, and $B$ be a Tate algebra of the form $$B=A\langle T_1,\ldots,T_n\rangle/(f_1,\ldots,f_n).$$
	We assume that the matrix $(\partial f_i/\partial T_j)_{ij}$ is invertible in $M_n(B)$.
	Then the natural morphism 
	$$\Kos(A\langle T_1,\ldots,T_n\rangle; f_1,\ldots,f_n)\to B$$
	is an equivalence in $\Dcal(\underline{A\langle T_1,\ldots,T_n\rangle})$, where $\underline{A\langle T_1,\ldots,T_n\rangle}$ is the condensed ring associated to $A\langle T_1,\ldots,T_n\rangle$.
\end{proposition}
\begin{proof}
	Let $B^+$ be the smallest ring of integral elements of $B$ containing the image of $A^+\langle T_1,\ldots,T_n\rangle/(f_1,\ldots,f_n)$.
	By \cite[Corollary 5.11]{Zav24}, there exists a rational open cover $\{\Spa(A_i,A_i^+)\to \Spa(A\langle T_1,\ldots,T_n\rangle,A^+\langle T_1,\ldots,T_n\rangle)\}_{i=1}^m$ such that 
	$$\Spa(B_i,B_i^+)=\Spa(B,B^+)\cap \Spa(A_i,A_i^+) \to \Spa(A_i,A_i^+)$$
	is a regular immersion of pure codimension $n$ for every $i$.
	Since the definition ideal of the Zariski closed immersion $\Spa(B_i,B_i^+)\to \Spa(A_i,A_i^+)$ is generated by the image of $f_1,\ldots,f_n$ in $A_i$, the sequence $f_1,\ldots,f_n\in A_i$ is also Koszul regular, that is, 
	the morphism $\Kos(A_i; f_1,\ldots,f_n)\to B_i$
	is a quasi-isomorphism of the underlying discrete modules.
	This can be proved in the same way as in \cite[page 22]{Tate57}, see also  \cite[066A]{stacks-project}.
	By the open mapping theorem (\cite[Theorem 3.2]{And21}), the morphism 
	$$\Kos(A_i; f_1,\ldots,f_n)\to B_i$$ 
	is also an equivalence in $\Dcal((A_i,A_i^+)_{\square})$.
	Since $B$ is sheafy, the natural morphism 
	$$B\otimes_{(A\langle T_1,\ldots,T_n\rangle,A^+\langle T_1,\ldots,T_n\rangle)_{\square}}(A_i,A_i^+)_{\square} \to B_i$$
	is an equivalence in $\Dcal((A_i,A_i^+)_{\square})$ by \cite[Proposition 4.12 (ii)]{And21}.
	Combining with the above arguments, we find that the morphism 
	$$\Kos(A\langle T_1,\ldots,T_n\rangle; f_1,\ldots,f_n)\to B$$ 
	becomes an equivalence after applying $-\otimes_{(A\langle T_1,\ldots,T_n\rangle,A^+\langle T_1,\ldots,T_n\rangle)_{\square}}(A_i,A_i^+)_{\square}$.
	Therefore, this is already an equivalence by \cite[Proposition 4.12 (v)]{And21}.
\end{proof}

\begin{corollary}\label{cor:etale fiber product}
	Let $X\to S \leftarrow Y$ be a diagram of locally noetherian adic spaces such that $Y\to S$ is \'{e}tale.
	Then the natural morphism $(X \times_S Y)_{\square}\to X_{\square}\times_{S_{\square}} Y_{\square}$ is an equivalence.
\end{corollary}
\begin{proof}
	By Lemma \ref{lem:fiber product commute}, we may assume that $X$, $Y$, and $S$ are affinoid analytic adic spaces.
	We put $S=\Spa(A,A^+)$ and $Y=\Spa(B,B^+)$.
	By Proposition \cite[Proposition 1.7.1]{Hub96}, we may assume that $(B,B^+)$ is of the form 
	$$(A\langle T_1,\ldots,T_n\rangle/(f_1,\ldots,f_n), (A^+\langle T_1,\ldots,T_n\rangle/(f_1,\ldots,f_n))^{\sim})$$
	such that $(\partial f_i/\partial T_j)_{ij}$ is invertible in $M_n(B)$, where $(A^+\langle T_1,\ldots,T_n\rangle/(f_1,\ldots,f_n))^{\sim}$ is the integral closure of the image of $A^+\langle T_1,\ldots,T_n\rangle/(f_1,\ldots,f_n)$.
	Therefore, the claim follows from Proposition \ref{prop:etale static}.
\end{proof}

\begin{corollary}\label{cor:etale D-etale}
Let $f\colon X\to S$ be an \'{e}tale morphism of locally noetherian adic spaces.	
Then $f_{\square}\colon X_{\square}\to S_{\square}$ is $\Dcal$-\'{e}tale.
\end{corollary}
\begin{proof}
	Since the diagonal morphism $X\to X\times_S X$ is an open immersion, the diagonal morphism $X_{\square} \to X_{\square}\times_{S_{\square}} X_{\square}$ is $\Dcal$-\'{e}tale by Lemma \ref{lem:open imm etale} and Corollary \ref{cor:etale fiber product}.
	Therefore, it is enough to show that $f_{\square}\colon X_{\square}\to S_{\square}$ is $\Dcal$-suave. 
	By the same argument as in Lemma \ref{lem:open imm etale}, we may assume that $X$ and $S$ are affinoid analytic adic spaces.
	By \cite[Proposition 1.7,1]{Hub96} and Proposition \ref{prop:etale static}, we may assume that $X_{\square}\to S_{\square}$ is standard solid \'{e}tale in the sense of \cite[Definition 3.5.5 (2)]{RC24}.
	Therefore, it is $\Dcal$-suave by \cite[Proposition 3.6.13]{RC24}.
\end{proof}

\begin{corollary}\label{cor:smooth fiber product}
	Let $X\to S \leftarrow Y$ be a diagram of locally noetherian adic spaces such that $Y\to S$ is smooth.
	Then the natural morphism $(X \times_S Y)_{\square}\to X_{\square}\times_{S_{\square}} Y_{\square}$ is an equivalence.
\end{corollary}
\begin{proof}
	By Lemma \ref{lem:fiber product commute}, we may assume that $X$, $Y$, and $S$ are affinoid analytic adic spaces.
	We put $S=\Spa(A,A^+)$, $Y=\Spa(B,B^+)$, and $X=\Spa(C,C^+)$.
	By \cite[Corollary 1.6.10]{Hub96}, we may assume that $Y\to S$ factors as a composition
	$$Y=\Spa(B,B^+)\to \Spa(A\langle T_1,\ldots,T_d\rangle, A^+\langle T_1,\ldots,T_d\rangle)\to \Spa(A,A^+)=S$$
	such that $\Spa(B,B^+)\to \Spa(A\langle T_1,\ldots,T_d\rangle, A^+\langle T_1,\ldots,T_d\rangle)$ is \'{e}tale.
	Therefore, by Corollary \ref{cor:etale fiber product}, we may assume $\Spa(B,B^+)=\Spa(A\langle T_1,\ldots,T_d\rangle, A^+\langle T_1,\ldots,T_d\rangle)$.
	In this case, we have equivalences 
	\begin{align*}
	&(A\langle T_1,\ldots,T_d\rangle, A^+\langle T_1,\ldots,T_d\rangle)_{\square}=(A,A^+)_{\square}\otimes_{\Zbb_{\square}} (\Zbb[T_1,\ldots,T_d],\Zbb[T_1,\ldots,T_d])_{\square},\\
	&(C\langle T_1,\ldots,T_d\rangle, C^+\langle T_1,\ldots,T_d\rangle)_{\square}=(C,C^+)_{\square}\otimes_{\Zbb_{\square}} (\Zbb[T_1,\ldots,T_d],\Zbb[T_1,\ldots,T_d])_{\square}
	\end{align*}
	by \cite[Lemma 4.7]{And21}.
	The claim easily follows from these equivalences.
\end{proof}

\begin{theorem}\label{thm:smooth D-smooth}
	Let $f\colon X\to S$ be a smooth morphism of pure relative dimension $d$ of locally noetherian adic spaces.	
	Then $f_{\square} \colon X_{\square}\to S_{\square}$ is $\Dcal$-smooth, and its dualizing complex $\omega_{f_{\square}}$ is given by $\Omega_{X/S}^{d}[d]$, where we note that an invertible module on $X$ gives an invertible object in $\Dcal(X_{\square})$.
\end{theorem}
\begin{proof}
	First, we prove that $f_{\square}$ is $\Dcal$-smooth.
	By the same argument as in Lemma \ref{lem:open imm etale}, we may assume that $X$ and $S$ are affinoid analytic adic spaces.
	Moreover, by \cite[Corollary 1.6.10]{Hub96}, we may assume that $X\to S$ factors as a composition 
	$$X \to \mathbb{B}^n_S \to S$$
	such that $X \to \mathbb{B}^n_S$ is \'{e}tale, where $\mathbb{B}^n_S=S\times_{\Spa \Zbb} \Spa (\Zbb[T_1,\ldots,T_n],\Zbb[T_1,\ldots,T_n])$.
	By Corollary \ref{cor:etale D-etale}, $X_{\square} \to \mathbb{B}^n_{S\square}$ is $\Dcal$-\'{e}tale and therefore $\Dcal$-smooth (Remark \ref{rem:D-etale D-smooth}), so it is enough to show that 
	$$\mathbb{B}^n_{S\square}\simeq S_{\square} \times_{\AnSpec \Zbb_{\square}} \AnSpec (\Zbb[T_1,\ldots,T_d],\Zbb[T_1,\ldots,T_d])_{\square} \to S_{\square}$$
	is $\Dcal$-smooth.
	It follows from the proof of \cite[Proposition 3.6.13]{RC24}.

	The equivalence $\omega_{f_{\square}}\simeq \Omega_{X/S}^{d}[d]$ follows from the same argument as in \cite[Theorem 3.6.15]{RC24}, in which the dualizing complex is computed by using the deformation to normal cone.
	We note that in our setting, it is enough to consider the deformation to normal cone in $\AnAdic$, that is, it is not necessary to consider the deformation to normal cone in the derived setting.
	We omit the details.
\end{proof}

%%%%%%%%%%%%%%%%%%%%%%%%%%%%%%%%%%%%%%%%% overconvergent group %%%%%%%%%%%%%%%%%%%%%%%%%%%%%%%%%%%%
\section{Dagger groups}
\subsection{Locally analytic representations and $h\dagger$-analytic representations}\label{subsection:la repns}
Let $G$ be a compact $p$-adic Lie group (over $\Qbb_p$) of dimension $d$, and $\gfrak$ be the Lie algebra of $G$.
Let $\Lcal$ be a $\Zbb_p$-lattice of $\gfrak$ such that $[\Lcal,\Lcal]\subset p\Lcal$.
As it is explained in \cite[5.2]{Eme17}, $\Lcal$ defines a rigid analytic group $\Gbb_0$ over $\Qbb_p$ whose underlying adic space is isomorphic to the close unit ball $\Spa(\Qbb_p\langle T_1,\ldots,T_d \rangle,\Zbb_p\langle T_1,\ldots,T_d \rangle)$.
From now on, we fix such a $\Zbb_p$-lattice $\Lcal$.
By shrinking $\Lcal$, we can assume that $\Gbb_0(\Qbb_p)$ is a normal subgroup of $G$.
Moreover, as it is explained in \cite[Example 2.1.1]{RJRC23}, for a rational number $h\geq 0$, we define a rigid analytic subgroup $\Gbb_h \subset \Gbb_0$ which is a rational localization defined by $|T_1|,\ldots,|T_d| \leq |p^h|$.
We also define a rigid analytic subgroup $\mathring{\Gbb}_h$ as $\bigcup_{h^{\prime}>h} \Gbb_{h^{\prime}}$, whose underlying adic space is a Stein space.
Let $\Gbb^h$ (resp. $\Gbb^{h+}$) denote the rigid analytic group $\bigsqcup_{g\in G/\Gbb_h(\Qbb_p)}g \Gbb_h$ (resp. $\bigsqcup_{g\in G/\mathring{\Gbb}_h(\Qbb_p)}g \mathring{\Gbb}_h$), where we note that $\Gbb_h(\Qbb_p)$ (resp. $\mathring{\Gbb}_h(\Qbb_p)$) is a normal subgroup of $G$.

\begin{remark}
	In this section, the index $h$ starts at $0$, but starting from any other integer would also work.
\end{remark}

\begin{definition}
	\begin{enumerate}
		\item Let $C^h(G,\Qbb_p)$ (resp. $C^{h+}(G,\Qbb_p)$) denote the space $\Ocal(\Gbb^h)$ (resp. $\Ocal(\Gbb^{h+})$) of \textit{analytic functions} on $\Gbb^h$ (resp. $\Gbb^{h+}$) with values in $\Qbb_p$.
		Let $C^h(G,\Zbb_p)$ denote the space $\Ocal^+(\Gbb^h)$ of \textit{analytic functions} on $\Gbb^h$ with values in $\Zbb_p$.
		\item For a rational number $h>0$, let $C^{h\dagger}(G,\Qbb_p)\coloneqq \varinjlim_{h>h^{\prime}\geq0}C^{h^{\prime}}(G,\Qbb_p)$ denote the space of \textit{overconvergent analytic functions} on $\Gbb^h$, which is a Hopf algebra over $\Qbb_{p,\square}$.
		\item Let $C^{\la}(G,\Qbb_p)\coloneqq\varinjlim_{h\geq0}C^{h}(G,\Qbb_p)$ denote the space of \textit{locally analytic functions} on $G$, which is a Hopf algebra over $\Qbb_{p,\square}$.
		\item For a rational number $h>0$, we define a dagger group $\Gbb^{h\dagger}$ as a group object $\AnSpec C^{h\dagger}(G,\Qbb_p)_{\square}$ in $\Aff_{\Qbb_{p,\square}}$.
		\item We define a locally analytic group $\Gbb^{\la}$ as a group object $\AnSpec C^{\la}(G,\Qbb_p)_{\square}$ in $\Aff_{\Qbb_{p,\square}}$.
	\end{enumerate}
\end{definition}

\begin{remark}
	\begin{enumerate}
		\item We note that our notation is slightly different from \cite{RJRC23}.
		For example, the space $C^{h+}(G,\Qbb_p)$ is written by $C^h(G,\Qbb_p)$ in loc. cit.
		\item For a rational number $h\geq0$, we have a natural isomorphism 
		$$C^{h+}(G,\Qbb_p)\cong \varprojlim_{h^{\prime}>h}C^{h^{\prime}}(G,\Qbb_p),$$
		where the limit is taken in the $\infty$-category $\Dcal(\Qbb_{p,\square})$. 
		\item For a rational number $h>0$, we have a natural isomorphism $$C^{h\dagger}(G,\Qbb_p)\simeq\varinjlim_{h>h^{\prime}\geq0}C^{h^{\prime}+}(G,\Qbb_p).$$
		Moreover, we have a natural isomorphism $$C^{h\dagger}(G,\Qbb_p)_{\square}\simeq\varinjlim_{h>h^{\prime}\geq0}(C^{h^{\prime}}(G,\Qbb_p),C^{h^{\prime}}(G,\Zbb_p))_{\square}$$ as analytic rings.
		\item We have natural isomorphisms
		$$C^{\la}(G,\Qbb_p)\simeq\varinjlim_{h\geq0}C^{h+}(G,\Qbb_p)\simeq\varinjlim_{h>0}C^{h\dagger}(G,\Qbb_p).$$
	\end{enumerate}
\end{remark}
\begin{remark}
	Let $(H,\Hbb_0)$ be another pair of a compact $p$-adic Lie group $H$ and a rigid analytification $\Hbb_0$ of a compact open normal subgroup of $H$.
	We set $(\Gbb\times \Hbb)_0=\Gbb_0\times \Hbb_0$.
	Then, for $*\in \{h, h+, h\dagger, \la\}$, we have a natural isomorphism $$C^*(G,\Qbb_p)\otimes_{\Qbb_{p,\square}}C^*(H,\Qbb_p)\cong C^*(G\times H,\Qbb_p)$$
	by \cite[Lemma 3.13, Lemma 3.28 (2)]{RJRC22}.
\end{remark}

\begin{lemma}\label{lem:nuclearness}
	Let $*$ be one of $\{h,h+, h\dagger, \la\}$.
	Then $C^*(G,\Qbb_p)$ is nuclear as an object of $\Dcal(\Qbb_{p,\square})$.
	In particular, $\Qbb_{p,\square}\to C^*(G,\Qbb_p)_{\square}$ is steady by \cite[Proposition 13.14]{AG}.
\end{lemma}
\begin{proof}
	Both $C^h(G,\Qbb_p)$ and $C^{h+}(G,\Qbb_p)$ are nuclear $\Qbb_{p,\square}$-modules by \cite[Corollary 3.16, Proposition 3.29]{RJRC22}.
	The other modules can be written as colimits of $C^h(G,\Qbb_p)$, so we get the claim.
\end{proof}

\begin{definition}
	For $V\in \Dcal(\Qbb_{p,\square})$, we define the following spaces of functions with values in $V$.
	\begin{enumerate}
		\item The space of \textit{$\Gbb^{h}$-analytic functions}
		$$C^h(G,V)\coloneqq(C^{h}(G,\Qbb_p),C^{h}(G,\Zbb_p))_{\square} \otimes_{\Qbb_{p,\square}} V.$$
		\item The space of \textit{$\Gbb^{h+}$-analytic functions} 
		$$C^{h+}(G,V)\coloneqq \varprojlim_{h^{\prime}>h}C^{h^{\prime}}(G,V)\simeq \varprojlim_{h^{\prime}>h} C^{h^{\prime}}(G,\Qbb_p)\otimes_{\Qbb_{p,\square}} V.$$
		\item For $h>0$, the space of \textit{$\Gbb^{h\dagger}$-analytic functions} 
		\begin{align*}
			C^{h\dagger}(G,V)&\coloneqq C^{h\dagger}(G,\Qbb_p)\otimes_{\Qbb_{p,\square}} V\\
			&\simeq \varinjlim_{h>h^{\prime}} C^{h^{\prime}}(G,V) \\
			&\simeq \varinjlim_{h>h^{\prime}} C^{h^{\prime}+}(G,V).
		\end{align*}
		\item The space of locally analytic functions 
		\begin{align*}
			C^{\la}(G,V)&\coloneqq C^{\la}(G,\Qbb_p) \otimes_{\Qbb_{p,\square}} V\\
			&\simeq \varinjlim_{h} C^{h}(G,V) \\
			&\simeq \varinjlim_{h} C^{h+}(G,V)\\
			&\simeq \varinjlim_{h} C^{h\dagger}(G,V).
		\end{align*}
	\end{enumerate}
\end{definition}

\begin{definition}
	We define the $\infty$-category $\Rep_{\square}(G)$ of representations of $G$ as 
	$$\Rep_{\square}(G)\coloneqq \Mod_{\Qbb_{p,\square}[G]}(\Dcal(\Qbb_{p,\square})),$$ the $\infty$-category of left $\Qbb_{p,\square}[G]$-modules in $\Dcal(\Qbb_{p,\square})$.
	It is a symmetric monoidal $\infty$-category by defining the tensor product of $V,W\in \Rep_{\square}(G)$ as $V\otimes_{\Qbb_{p,\square}}W$ with the diagonal $G$-action.
\end{definition}

\begin{remark}
	We refer to objects of $\CAlg(\Rep_{\square}(G))$ as ($\Ebb_{\infty}$-)algebras over $\Qbb_{p,\square}$ with a $G$-action.
	The full subcategory $\CAlg(\Rep_{\square}(G))^{\heart}\subset \CAlg(\Rep_{\square}(G))$ consisting of static objects is equivalent to the category of static $\Qbb_{p,\square}$-algebras with $G$-actions in the usual sense (i.e., a static $\Qbb_{p,\square}$-algebra $A$ with a morphism $G\to \Aut(A)$ of condensed groups, where $\Aut(A)$ is the (condensed) automorphism group of $A$ as a $\Qbb_{p,\square}$-algebra).
\end{remark}

Next, following in \cite{RJRC22,RJRC23}, we define locally analytic vectors and locally analytic representations.
By \cite[Lemma 3.2.1]{RJRC23}, the functor $\Dcal(\Qbb_{p,\square}) \to \Dcal(\Qbb_{p,\square}) ;\; V \mapsto C^{?}(G,V)$ naturally promotes to an exact functor $\Rep_{\square}(G)\to \Rep_{\square}(G^3)$, where $C^?$ is one of the spaces of functions $C^h$, $C^{h+}$, $C^{h\dagger}$, or $C^{\la}$.
Roughly speaking, the $G^3$-action on $C^{?}(G,V)$ is given by 
$$((g_1,g_2,g_3)\star f)(h)=g_3\cdot f(g_1^{-1}hg_2).$$
For a finite set $I\subset \{1,2,3\}$, we have a group homomorphism $\iota_I\colon G\to G^3$ defined by $\iota_I(g)_j=g$ if $j\in I$ and $\iota_I(g)_j=1$ if $j\notin I$.
By using $\iota_I$, we can regard $V \in \Rep_{\square}(G^3)$ as an object in $\Rep_{\square}(G)$, and we write it $V_{\star_{I}}$.

\begin{definition}
	\begin{enumerate}
		\item The functor of (derived) $h+$-analytic vectors $$(-)^{h+\mathchar`-\an} \colon \Rep_{\square}(G) \to \Rep_{\square}(G)$$ is defined as 
		$$V^{h+\mathchar`-\an}\coloneqq \intHom_{\Qbb_{p,\square}[G]}(\Qbb_p,C^{h+}(G,V)_{\star_{1,3}}),$$
		where the $G$-action on $V^{h+\mathchar`-\an}$ is induced by the $\star_2$-action.
		Replacing $C^{h+}$ with $C^{h\dagger}$ (resp. $C^{\la}$), we also define the functor of (derived) $h\dagger$-analytic vectors (resp. locally analytic vectors) $(-)^{h\dagger\mathchar`-\an} \colon \Rep_{\square}(G) \to \Rep_{\square}(G)$ (resp. $(-)^{\la}$).
		\item A representation $V \in \Rep_{\square}(G)$ of $G$ is called \textit{$h+$-analytic} (resp. \textit{$h\dagger$-analytic}, \textit{locally analytic}) if the natural morphism $V^{h+\mathchar`-\an}\to V$ (resp. $V^{h\dagger-an}\to V$, $V^{\la}\to V$) is an equivalence.
		Let $\Rep_{\square}^{h+}(G)$ (resp. $\Rep_{\square}^{h\dagger}(G)$, $\Rep_{\square}^{\la}(G)$) denote the full subcategory of $\Rep_{\square}(G)$ consisting of $h+$-analytic representations (resp. $h\dagger$-analytic representations, locally analytic representations).
	\end{enumerate}
\end{definition}

The following two lemmas follow from the fact that $\Qbb_p$ with the trivial $G$-action is a compact object in $\Rep_{\square}(G)$, which follows from the Lazard-Serre resolution (see the proof of Lemma \ref{lem:descendable zero map}).

\begin{lemma}\label{lem:la=colim of h-an}
	Let $V\in \Rep_{\square}(G)$ be a representation of $G$.
	Then we have natural equivalences $$V^{\la}\simeq \varinjlim_{h}V^{h+\mathchar`-\an}\simeq \varinjlim_{h}V^{h\dagger\mathchar`-\an}$$ and $$V^{h\dagger\mathchar`-\an} \simeq \varinjlim_{h>h^{\prime}}V^{h^{\prime}+\mathchar`-\an}\simeq \varinjlim_{h>h^{\prime}}V^{h^{\prime}\dagger\mathchar`-\an}.$$
\end{lemma}

\begin{lemma}\label{lem:stable under colimit}
	The functors $$(-)^{h\dagger\mathchar`-\an} \colon \Rep_{\square}(G) \to \Rep_{\square}(G)$$ and $$(-)^{\la}\colon \Rep_{\square}(G) \to \Rep_{\square}(G)$$ preserve all small colimits, and the full subcategories $\Rep_{\square}^{h\dagger}(G)$ and $\Rep_{\square}^{\la}(G)$ of $\Rep_{\square}(G)$ are stable under small colimits.
\end{lemma}
\begin{proof}
It follows from the fact that the functors $C^{h\dagger}(G,-)$ and $C^{\la}(G,-)$ preserve all small colimits.
\end{proof}

\begin{remark}
 The functor $$(-)^{h+\mathchar`-\an} \colon \Rep_{\square}(G) \to \Rep_{\square}(G)$$ does not necessarily preserve small colimits.
 However, for another reason, the full subcategory $\Rep_{\square}^{h+}(G)$ is stable under small colimits, see Lemma \ref{lem:h+-an implies la}.
\end{remark}

\begin{definition}
	Let $D^{h+}(G,\Qbb_p)$ (resp. $D^{h}(G,\Qbb_p)$) denote the distribution algebra $H^0(\intHom_{\Qbb_p}(C^{h+}(G,\Qbb_p),\Qbb_p))$ (resp. $H^0(\intHom_{\Qbb_p}(C^{h}(G,\Qbb_p),\Qbb_p))$), where we note that this dual is non-derived.	
\end{definition}
\begin{remark}
	The distribution algebra $D^{h+}(G,\Qbb_p)$ is denoted by $D^h(G,\Qbb_p)$ in \cite[Definition 2.1.2]{RJRC23}.
\end{remark}

\begin{lemma}[{\cite[Lemma 4.33, Theorem 4.36]{RJRC22}}]\label{lem:h+-an implies la}
	\begin{enumerate}
		\item The distribution algebra $D^{h+}(G,\Qbb_p)$ is an idempotent $\Qbb_{p,\square}[G]$-algebra, that is, the multiplication morphism $D^{h+}(G,\Qbb_p) \otimes_{\Qbb_{p,\square}[G]} D^{h+}(G,\Qbb_p) \to D^{h+}(G,\Qbb_p)$ is an equivalence.
		\item There is a natural functorial equivalence 
		$$V^{h+\mathchar`-\an}\simeq \intHom_{\Qbb_{p,\square}[G]}(D^{h+}(G,\Qbb_p),V)$$
		for $V\in \Rep_{\square}(G)$, where the $G$-action on $\intHom_{\Qbb_{p,\square}[G]}(D^{h+}(G,\Qbb_p),V)$ is induced by the right multiplication on $D^{h+}(G,\Qbb_p)$.
		\item A representation $V \in \Rep_{\square}(G)$ of $G$ is $h+$-analytic if and only if $V$ lies in the essential image of the forgetful functor $$\Mod_{D^{h+}(G,\Qbb_p)}(\Dcal(\Qbb_{p,\square})) \to \Mod_{\Qbb_{p,\square}[G]}(\Dcal(\Qbb_{p,\square}))=\Rep_{\square}(G),$$ which is fully faithful by (1). 
		In particular, there is an equivalence of $\infty$-categories
		$$\Rep_{\square}^{h+}(G)\simeq \Mod_{D^{h+}(G,\Qbb_p)}(\Dcal(\Qbb_{p,\square})),$$
		and the full subcategory $\Rep_{\square}^{h+}(G)$ is stable under small colimits and limits.
	\end{enumerate}
\end{lemma}

\begin{remark}
	Since the trivial representation of $G$ becomes a $D^{h+}(G,\Qbb_p)$-module, it is $h+$-analytic for any $h\geq0$.
	Moreover, it is locally analytic and $h\dagger$-analytic for any $h>0$ by the following corollary.
\end{remark}

\begin{corollary}\label{cor:transitivity}
	\begin{enumerate}
		\item If a representation $V \in \Rep_{\square}(G)$ of $G$ is $h+$-analytic, then $V$ is also $h^{\prime}+$-analytic for any $h^{\prime}\geq h$. 
		\item If a representation $V \in \Rep_{\square}(G)$ of $G$ is $h\dagger$-analytic, then $V$ is also $h+$-analytic and $h^{\prime}\dagger$-analytic for any $h^{\prime}\geq h$. 
		\item If a representation $V \in \Rep_{\square}(G)$ of $G$ is $h+$-analytic, then $V$ is also $h^{\prime}\dagger$-analytic for any $h^{\prime}> h$. 
		\item If a representation $V \in \Rep_{\square}(G)$ of $G$ is $h+$-analytic or $h\dagger$-analytic, then $V$ is locally analytic.
	\end{enumerate}
\end{corollary}
\begin{proof}
	Part (1) follows from Lemma \ref{lem:h+-an implies la} (3).
	Parts (2), (3), and (4) follow from Lemma \ref{lem:la=colim of h-an}, Lemma \ref{lem:stable under colimit}, and Lemma \ref{lem:h+-an implies la} (3).
\end{proof}

\begin{corollary}\label{cor:adj la}
	Let $V\in \Rep_{\square}(G)$ be a representation of $G$.
	\begin{enumerate}
		\item The representation $V^{h+\mathchar`-\an}$ of $G$ is $h+$-analytic.
		Moreover, the functor of $h+$-analytic vectors 
		$$(-)^{h+\mathchar`-\an} \colon \Rep_{\square}(G) \to \Rep_{\square}^{h+}(G)$$
		is a right adjoint functor of the inclusion $\Rep_{\square}^{h+}(G) \hookrightarrow \Rep_{\square}(G)$.
		\item The representation $V^{h\dagger\mathchar`-\an}$ of $G$ is $h\dagger$-analytic.
		Moreover, the functor of $h\dagger$-analytic vectors 
		$$(-)^{h\dagger\mathchar`-\an} \colon \Rep_{\square}(G) \to \Rep_{\square}^{h\dagger}(G)$$
		is a right adjoint functor of the inclusion $\Rep_{\square}^{h\dagger}(G) \hookrightarrow \Rep_{\square}(G)$.
		\item The representation $V^{\la}$ of $G$ is locally analytic.
		Moreover, the functor of locally analytic vectors 
		$$(-)^{\la} \colon \Rep_{\square}(G) \to \Rep_{\square}^{\la}(G)$$
		is a right adjoint functor of the inclusion $\Rep_{\square}^{\la}(G) \hookrightarrow \Rep_{\square}(G)$.
	\end{enumerate}
\end{corollary}
\begin{proof}
	Part (1) follows from Lemma \ref{lem:h+-an implies la} (2), (3), where we note that 
	$$\intHom_{\Qbb_{p,\square}[G]}(D^{h+}(G,\Qbb_p),-)\colon \Mod_{\Qbb_{p,\square}[G]}(\Dcal(\Qbb_{p,\square})) \to \Mod_{D^{h+}(G,\Qbb_p)}(\Dcal(\Qbb_{p,\square}))$$ 
	gives a right adjoint of the forgetful functor 
	$$\Mod_{D^{h+}(G,\Qbb_p)}(\Dcal(\Qbb_{p,\square}))\to \Mod_{\Qbb_{p,\square}[G]}(\Dcal(\Qbb_{p,\square})).$$
	We prove (2).
	From (1), Lemma \ref{lem:la=colim of h-an}, and Lemma \ref{cor:transitivity} (3), we find that the representation $V^{h\dagger\mathchar`-\an}$ of $G$ is $h\dagger$-analytic.
	For $V\in \Rep_{\square}(G)$, we write $\iota_V^{h\dagger} \colon V^{h\dagger\mathchar`-\an} \to V$ and $\iota_V^{h+} \colon V^{h+\mathchar`-\an} \to V$ for the natural morphisms.
	To prove adjointness, it suffices to show that for every object $V\in \Rep_{\square}(G)$, the morphism 
	$$(\iota_V^{h\dagger})^{h\dagger\mathchar`-\an} \colon (V^{h\dagger\mathchar`-\an})^{h\dagger\mathchar`-\an} \to V^{h\dagger\mathchar`-\an}$$
	is an equivalence by the dual of \cite[Proposition 5.2.7.4]{HTT}, where we note that we have already proved that $$\iota_{V^{h\dagger\mathchar`-\an}}^{h\dagger}\colon (V^{h\dagger\mathchar`-\an})^{h\dagger\mathchar`-\an} \to V^{h\dagger\mathchar`-\an}$$ is an equivalence.
	From (1), we have already known that for every object $V\in \Rep_{\square}(G)$, the morphism 
	$$(\iota_V^{h^{\prime}+})^{h^{\prime}+\mathchar`-\an} \colon (V^{h^{\prime}+\mathchar`-\an})^{h^{\prime}+\mathchar`-\an} \to V^{h^{\prime}+\mathchar`-\an}$$
	is an equivalence for every $h^{\prime}<h$.
	Therefore, by taking colimits, we find that $(\iota_V^{h\dagger})^{h\dagger\mathchar`-\an}$ is also an equivalence.
	
	Part (3) follows from the same argument as above.
\end{proof}

\begin{corollary}
	Let $V\in \Rep_{\square}(G)$ be a representation of $G$.
	Then $V$ is $h+$-analytic (resp. $h\dagger$-analytic, locally analytic) if and only if $H^i(V)$ is $h+$-analytic (resp. $h\dagger$-analytic, locally analytic) for any $i\in \Zbb$.
	In particular, the $t$-structure on $\Rep_{\square}(G)$ induces a $t$-structure on $\Rep_{\square}^{h+}(G)$, $\Rep_{\square}^{h\dagger}(G)$, and $\Rep_{\square}^{\la}(G)$.
\end{corollary}
\begin{proof}
	It follows from the same argument as in \cite[Proposition 3.3.5]{RJRC23}.
\end{proof}

\begin{definition}[{\cite[Definition 4.32]{RJRC22}}]\label{defn:non-derived analytic vector}
	Let $V$ be a static $\Qbb_{p,\square}$-module with a continuous $G$-action.
	It is called \textit{non-derived $h$-analytic} if the natural morphism
	$$H^0(V^{h\mathchar`-\an})\coloneqq H^0(\Gamma(G,C^h(G,V))) \to V$$
	is an isomorphism.
\end{definition}
\begin{remark}
	In contrast to $h+$-analytic representations or $h\dagger$-analytic representations, there are few static representations $V$ of $G$ such that $H^i(\Gamma(G,C^h(G,V)))=0$ for $i>0$.
	For example, if $G=\Zbb_p$ and $V=\Qbb_p$ with the trivial action, then $H^1(\Gamma(G,C^h(G,V)))\neq 0$ for any $h$.
	Therefore, we do not define the notion of $h$-analytic representations for non-static representations.
\end{remark}

\begin{lemma}\label{lem:h-an implies h+-an}
	Let $V$ be a non-derived $h$-analytic representation of $G$.
	Then $V$ is an $h+$-analytic representation.
\end{lemma}
\begin{proof}
	By \cite[Theorem 4.36(1)]{RJRC22}, we have an isomorphism
	$$H^0(\Gamma(G,C^h(G,V))) \cong H^0(\intHom_G(D^h(G,\Qbb_p),V)).$$
	Therefore, $V$ becomes a $D^h(G,\Qbb_p)$-module, and therefore, it becomes a $D^{h+}(G,\Qbb_p)$-module.
	Then the claim follows from Lemma \ref{lem:h+-an implies la} (3).
\end{proof}

In the following, to handle both the cases of locally analytic representations and $h\dagger$-analytic representations simultaneously, we denote $\la$ by $\infty \dagger$.
This notation is consistent with various computations.

The following lemma and proposition are analogues of \cite[Proposition 1.62, Proposition 1.63]{Mikami24}.
Since their proofs are the same as those in \cite{Mikami24}, we omit them.

\begin{lemma}\label{lem:orbit morphism locally analytic}
	Let $h$ be a positive rational number or $\infty$.
	Let $V\in \Rep_{\square}^{h\dagger}(G)$ be an $h\dagger$-analytic representation.
	Then, there is a natural equivalence in $\Rep_{\square}(G)$
	$$C^{h\dagger}(G,V)_{\star_1}\simeq C^{h\dagger}(G,V)_{\star_{1,3}}.$$
\end{lemma}

\begin{remark}\label{rem:orbit}
	In the same way, we can also prove that for $V\in \Rep_{\square}^{h\dagger}(G)$, there is a natural equivalence in $\Rep_{\square}(G)$
	$$C^{h\dagger}(G,V)_{\star_2}\simeq C^{h\dagger}(G,V)_{\star_{2,3}}.$$
\end{remark}

\begin{proposition}\label{prop:proj formula}
	Let $h$ be a positive rational number or $\infty$.
	Let $V,W \in \Rep_{\square}(G)$ be representations of $G$.
	We assume that $W$ is $h\dagger$-analytic.
	Then there is a natural equivalence in $\Rep_{\square}(G)$
	$$(V\otimes_{\Qbb_{p,\square}}W)^{h\dagger\mathchar`-\an} \simeq V^{h\dagger\mathchar`-\an}\otimes_{\Qbb_{p,\square}}W.$$
\end{proposition}

\begin{remark}
	In the case of locally analytic representations, the above proposition is proven in \cite[Corollary 3.2.14 (3)]{RJRC23}.
\end{remark}

\begin{corollary}
	Let $h$ be a positive rational number or $\infty$.
	The full subcategory $\Rep_{\square}^{h\dagger}(G)\subset \Rep_{\square}(G)$ is stable under tensor products.
	In particular, $\Rep_{\square}^{h\dagger}(G)$ becomes a symmetric monoidal $\infty$-category.
\end{corollary}

\begin{remark}
	The analogous statement for $h+$-analytic representations is not true.
	This is one of the reasons why we introduce the notion of $h\dagger$-analytic representations.
\end{remark}

Next, we give geometric interpretations of locally analytic representations and $h\dagger$-analytic representations.
For this, we prove some lemmas.

\begin{lemma}\label{lem:tensor product dual}
	Let $h\geq 0$ be a rational number.
	\begin{enumerate}
		\item The natural morphism 
		$$C^h(G,\Qbb_p) \to \intHom_{\Qbb_p}(D^h(G,\Qbb_p),\Qbb_p)$$ 
		is an equivalence.
		\item Let $V\in \Dcal(\Qbb_{p,\square})$ be a $\Qbb_{p,\square}$-module.
		Then a $C^h(G,\Qbb_p)$-module $$\intHom_{\Qbb_p}(D^h(G,\Qbb_p),V)$$ is $(C^h(G,\Qbb_p),C^h(G,\Zbb_p))_{\square}$-complete, and the natural morphism
		$$(C^h(G,\Qbb_p),C^h(G,\Zbb_p))_{\square}\otimes_{\Qbb_{p,\square}} V \to \intHom_{\Qbb_p}(D^h(G,\Qbb_p),V)$$ 
		is an equivalence.
	\end{enumerate}
\end{lemma}
\begin{proof}
	Part (1) follows from \cite[Lemma 3.10, Remark 3.11]{RJRC22}.
	Next, we prove (2).
	If $h$ is an integer, then $C^h(G,\Qbb_p)$ is a finite direct product of affinoid $\Qbb_p$-algebras which are isomorphic to $\Qbb_p\langle T_1,\ldots,T_d \rangle$.
	Therefore, the claim follows from \cite[Corollary 2.19]{RJRC22}.
	In general, by taking a finite extension $K/\Qbb_p$ such that $p^h \in K$, we can reduce the claim to the case above.
\end{proof}

\begin{proposition}\label{prop:adjoint cofree}
	Let $h$ be a positive rational number or $\infty$.
	The functor 
	$$\Dcal(\Qbb_{p,\square}) \to \Rep_{\square}^{h\dagger}(G) ;\; V \mapsto C^{h\dagger}(G,V)_{\star_2}$$
	is a right adjoint functor of the forgetful functor $\Rep_{\square}^{h\dagger}(G)\to \Dcal(\Qbb_{p,\square})$.
\end{proposition}
\begin{proof}
	By Corollary \ref{cor:adj la}, the forgetful functor $\Rep_{\square}^{h\dagger}(G)\to\Rep_{\square}(G)$ admits a right adjoint functor 
	$$\Rep_{\square}(G) \to \Rep_{\square}^{h\dagger}(G) ;\; V \mapsto V^{h\dagger\mathchar`-\an}$$
	and the forgetful functor $\Rep_{\square}(G)\to \Dcal(\Qbb_{p,\square})$ admits a right adjoint functor 
	$$\Dcal(\Qbb_{p,\square}) \to \Rep_{\square}(G) ;\; V \mapsto \intHom_{\Qbb_p}(\Qbb_{p,\square}[G],V)_{\star_2}.$$
	Therefore, it is enough to show that for $V\in \Dcal(\Qbb_{p,\square})$, there is a natural equivalence
	$$(\intHom_{\Qbb_p}(\Qbb_{p,\square}[G],V)_{\star_2})^{h\dagger\mathchar`-\an}\simeq C^{h\dagger}(G,V)_{\star_2}.$$
	It follows from the following computation:
	\begin{align*}
		&(\intHom_{\Qbb_p}(\Qbb_{p,\square}[G],V))^{h\dagger\mathchar`-\an} \\
		\simeq &\varinjlim_{h^{\prime}<h} (\intHom_{\Qbb_p}(\Qbb_{p,\square}[G],V))^{h^{\prime}+\mathchar`-\an}\\
		\simeq &\varinjlim_{h^{\prime}<h} \intHom_{\Qbb_{p,\square}[G]}(D^{h^{\prime}+}(G,\Qbb_p),\intHom_{\Qbb_p}(\Qbb_{p,\square}[G],V))\\
		\simeq &\varinjlim_{h^{\prime}<h}\intHom_{\Qbb_p}(D^{h^{\prime}+}(G,\Qbb_p),V)\\
		\simeq &\varinjlim_{h^{\prime}<h}\intHom_{\Qbb_p}(D^{h^{\prime}}(G,\Qbb_p),V)\\
		\simeq &\varinjlim_{h^{\prime}<h}(C^{h^{\prime}}(G,\Qbb_p),C^{h^{\prime}}(G,\Zbb_p))_{\square}\otimes_{\Qbb_{p,\square}} V \\
		\simeq &C^{h\dagger}(G,V),
	\end{align*}
	where the first equivalence follows from Lemma \ref{lem:la=colim of h-an}, the second equivalence follows from Lemma \ref{lem:h+-an implies la}, the fourth equivalence follows from the fact that projective systems $\{D^{h^{\prime}+}(G,\Qbb_p)\}_{h^{\prime}<h}$ and $\{D^{h^{\prime}}(G,\Qbb_p)\}_{h^{\prime}<h}$ are pro-equivalent, and the fifth equivalence follows from Lemma \ref{lem:tensor product dual}.
\end{proof}

We want to show that the adjunction in Proposition \ref{prop:adjoint cofree} is comonadic.

\begin{lemma}\label{lem:descendable zero map}
	Let $F\colon \Rep_{\square}(G) \to \Dcal(\Qbb_{p,\square})$ be the forgetful functor.
	Then, there exists an integer $n\geq 1$ such that for morphisms $f_i \colon V_i\to V_{i+1}$ in $\Rep_{\square}(G)$ ($i=0,1,\ldots,n-1$), if $F(f_i)=0$ for all $i$, then $f_{n-1}\circ \cdots \circ f_0=0$ in $\Rep_{\square}(G)$.
\end{lemma}
\begin{proof}
	For an integer $n>0$, let $\Ccal_n$ be the full subcategory of $\Rep_{\square}(G)$ consisting of those representations $V\in \Rep_{\square}(G)$ such that for morphisms $f_i \colon V_i\to V_{i+1}$ in $\Rep_{\square}(G)$ ($i=0,1,\ldots,n-1$), if $F(f_i)=0$ for all $i$, then $$\intHom_{\Qbb_p}(V,V_0)_{\star_{1,3}}\overset{f_0\circ}{\longrightarrow} \intHom_{\Qbb_p}(V,V_1)_{\star_{1,3}} \overset{f_1\circ}{\longrightarrow} \cdots \overset{f_{n-1}\circ}{\longrightarrow} \intHom_{\Qbb_p}(V,V_n)_{\star_{1,3}}$$ is equal to $0$ in $\Rep_{\square}(G)$.
	It is enough to show that there exists an integer $n\geq 0$ such that $\Qbb_p \in \Ccal_n$.

	First, we have a functor 
	$$H \colon \Dcal(\Qbb_{p,\square}) \to \Rep_{\square}(G) ;\; V \mapsto \intHom_{\Qbb_p}(\Qbb_{p,\square}[G],V)_{\star_1}.$$
	Then the functor $H\circ F$ is given by $V\mapsto \intHom_{\Qbb_p}(\Qbb_{p,\square}[G],V)_{\star_1}$, and it is equivalent to the functor $\Rep_{\square}(G) \to \Rep_{\square}(G) ;\; V\mapsto \intHom_{\Qbb_p}(\Qbb_{p,\square}[G],V)_{\star_{1,3}}$.
	Therefore, we get $\Qbb_{p,\square}[G]\in \Ccal_0$.

	Next, by the definition, the full subcategory $\Ccal_n \subset \Rep_{\square}(G)$ is stable under retracts, shifts, and direct sums.
	Moreover, for any morphism $f \colon V\to W$ in $\Ccal_n$, $\fib(V\to W)$ lies in $\Ccal_{2n}$.
	We take a compact open normal uniform pro-$p$ subgroup $G_0 \subset G$.
	Then we have the Lazard-Serre resolution
	$$0 \to \Qbb_{p,\square}[G_0]^{\binom{d}{d}}\to \cdots\to \Qbb_{p,\square}[G_0]^{\binom{d}{0}} \to \Qbb_p\to 0$$
	(\cite[Theorem 5.7]{RJRC22}).
	By applying $\Qbb_{p,\square}[G]\otimes_{\Qbb_{p,\square}[G_0]} -$, we get a resolution
	$$0 \to \Qbb_{p,\square}[G]^{\binom{d}{d}}\to \cdots\to \Qbb_{p,\square}[G]^{\binom{d}{0}} \to \Qbb_{p,\square}[G]\otimes_{\Qbb_{p,\square}[G_0]}\Qbb_p\to 0.$$
	Since the trivial representation $\Qbb_p$ of $G$ is a direct summand of $\Qbb_{p,\square}[G]\otimes_{\Qbb_{p,\square}[G_0]}\Qbb_p$, there exists an integer $n>0$ such that $\Qbb_p\in \Ccal_n$.
\end{proof}

\begin{proposition}\label{prop:comonadic}
	Let $h$ be a positive rational number or $\infty$.
	The adjunction
	\begin{align*}
		F \colon \Rep_{\square}^{h\dagger}(G)\rightleftarrows \Dcal(\Qbb_{p,\square})\colon C^{h\dagger}(G,-)_{\star_2}, 
	\end{align*}
	where $F$ is the forgetful functor, is comonadic.
\end{proposition}
\begin{proof}
	By the dual of the Barr-Beck-Lurie theorem (\cite[Theorem 3.3]{Mat16}, \cite[Theorem 4.7.3.5]{HA}), it suffices to show that $F$ is conservative, and that for every cosimplicial object $V^{\bullet}$ in $\Rep_{\square}^{h\dagger}(G)$ such that $F(V^{\bullet})$ admits a splitting, the natural morphism $F(\Tot(V^{\bullet}))\to \Tot(F(V^{\bullet}))$ is an equivalence.
	It is clear that $F$ is conservative.
	Let $V^{\bullet}$ be a cosimplicial object such that $F(V^{\bullet})$ admits a splitting.
	We regard $V^{\bullet}$ as a cosimplicial object in $\Rep_{\square}(G)$, and let $\Tot^{\prime}(V^{\bullet})$ denote the totalization of $V^{\bullet}$ in $\Rep_{\square}(G)$.
	Since the forgetful functor $F^{\prime}\colon \Rep_{\square}(G) \to \Dcal(\Qbb_{p,\square})$ preserves all small limits, the natural morphism $F^{\prime}(\Tot^{\prime}(V^{\bullet}))\to \Tot(F(V^{\bullet}))$ is an equivalence.
	Therefore, it is enough to show $\Tot^{\prime}(V^{\bullet})\in \Rep_{\square}^{h\dagger}(G)$.
	Since $F(V^{\bullet})$ admits a splitting and $F^{\prime}(\Tot^{\prime}(V^{\bullet}))\to \Tot(F(V^{\bullet}))$ is an equivalence, the pro-object $\{F(\Tot^{\prime}_n(\fib(\Tot^{\prime}(V^{\bullet})\to V^{\bullet})))\}_n$ is pro-zero, where $\Tot^{\prime}_n$ is the $n$-truncated totalization in $\Rep_{\square}(G)$.
	By Lemma \ref{lem:descendable zero map}, $\{\Tot^{\prime}_n(\fib(\Tot^{\prime}(V^{\bullet})\to V^{\bullet}))\}_n$ is also pro-zero.
	Therefore, $\Tot^{\prime}(V^{\bullet})$ is a retract of $\Tot^{\prime}_n(V^{\bullet})$ for $n$ sufficiently large.
	Since $\Rep_{\square}^{h\dagger}(G) \subset \Rep_{\square}(G)$ is stable under finite limits and retracts, we get $\Tot^{\prime}(V^{\bullet})\in \Rep_{\square}^{h\dagger}(G)$.
\end{proof}

The following lemma is a combination of the dual of \cite[Corollary 4.7.5.3, Theorem 4.7.3.5]{HA}.

\begin{lemma}\label{lem:Barr-Beck-Lurie descent}
	Let $\Ccal^{\bullet}\colon \Delta_{+} \to \Cat_{\infty}$ be an augmented cosimplicial $\infty$-category, and set $\Ccal=\Ccal^{-1}$.
	Let $F\colon \Ccal \to \Ccal^0$ be the evident functor.
	We assume the following:
	\begin{enumerate}
		\item The functor $F$ exhibits $\Ccal$ as comonadic over $\Ccal^0$, that is, $F$ has a right adjoint functor $G\colon \Ccal^0\to\Ccal$ and the adjunction $(F,G)$ is comonadic.
		\item For every $\alpha\colon[m]\to[n]$ in $\Delta_+$, the diagram
		$$
		\xymatrix{
			\Ccal^m\ar[r]^-{d_{m+1}^0}\ar[d]_-{\alpha}& \Ccal^{m+1}\ar[d]^-{\alpha^{\prime}}\\
			\Ccal^n \ar[r]^-{d_{n+1}^0}& \Ccal^{n+1}
		}
		$$
		is right adjointable, where $\alpha^{\prime}$ is defined as $\alpha^{\prime}(0)=0$ and $\alpha^{\prime}(i)=\alpha(i-1)+1$ for $i\geq 1$.
		In other words, $d_{m+1}^0$ (resp. $d_{n+1}^0$) admits a right adjoint functor $G_{m+1}\colon \Ccal^{m+1}\to \Ccal^{m}$ (resp. $G_{n+1}\colon \Ccal^{n+1}\to \Ccal^{n}$), and the natural morphism $(\alpha \circ G_{m+1})\to (G_{n+1}\circ \alpha^{\prime})$ of functors from $\Ccal^{m+1}$ to $\Ccal^n$ is an equivalence. 
	\end{enumerate}
	Then the canonical morphism $\Ccal\to \varprojlim_{[n]\in\Delta}\Ccal^n$ is an equivalence of $\infty$-categories.
\end{lemma}

\begin{proposition}\label{prop:geometric interpretation}
Let $h$ be a positive rational number or $\infty$.
There is a natural equivalence of symmetric monoidal $\infty$-categories
$$\Dcal(\AnSpec\Qbb_{p,\square}/\Gbb^{h\dagger})\simeq \Rep_{\square}^{h\dagger}(G).$$
\end{proposition}

\begin{proof}
	We construct a symmetric monoidal functor $$\Rep_{\square}^{h\dagger}(G)\to \Dcal(\AnSpec\Qbb_{p,\square}/\Gbb^{h\dagger}).$$
	By taking the \v{C}ech nerve of $\AnSpec\Qbb_{p,\square}\to \AnSpec\Qbb_{p,\square}/\Gbb^{h\dagger}$, we get a cosimplicial $\Qbb_{p,\square}$-algebra $\{C^{h\dagger}(G^n,\Qbb_p)\}_{[n]\in \Delta}$.
	Roughly speaking, this cosimplicial algebra is given as follows:
	\begin{itemize}
		\item For $1\leq n$ and $0\leq i\leq n$, the coface map $d_n^i$ is given by 
		$$\begin{array}{ccl}
			C^{h\dagger}(G^{n-1},\Qbb_p) &\longrightarrow &  C^{h\dagger}(G^{n},\Qbb_p)\\
	        \rotatebox{90}{$\in$}       &                &  \quad\quad\rotatebox{90}{$\in$}\\
			f                            &\longmapsto     & \left( (g_1,\ldots,g_n) \mapsto
						                                                   \begin{cases}
							                                                f(g_2,\ldots,g_n) &(i=0) \\
							                                                f(g_1,\ldots,g_{i}g_{i+1},\ldots, g_n) &(0<i<n)\\
						                                                    f(g_{1},\ldots,g_{n-1}) &(i=n) 
						                                                  \end{cases}\right).
		\end{array}$$
		\item For $0\leq n$ and $0\leq i\leq n$, the codegeneracy map $s_n^i$ is given by 
		$$\begin{array}{ccl}
			C^{h\dagger}(G^{n+1},\Qbb_p) &\longrightarrow &  C^{h\dagger}(G^{n},\Qbb_p)\\
	        \rotatebox{90}{$\in$}       &                &  \quad\quad\rotatebox{90}{$\in$}\\
			f                            &\longmapsto     &  \left((g_1,\ldots,g_n) \mapsto f(g_1,\ldots,g_{i},1,g_{i+1},\ldots,g_n)\right).	                                                
		\end{array}$$
	\end{itemize}
	By composing with the functor $\iota\colon\Delta\to \Delta ;\; [n]\mapsto \{*\}\sqcup[n]$, where the well-order on $\{*\}\sqcup[n]$ is defined so that $*$ becomes the smallest element, we get a cosimplicial algebra $\{C^{h\dagger}(G^{n+1},\Qbb_p)\}_{[n]\in \Delta}$.
	Moreover, the natural transformation $\id \to \iota$ defines a morphism $\{C^{h\dagger}(G^{n},\Qbb_p)\}_{[n]\in \Delta}\to \{C^{h\dagger}(G^{n+1},\Qbb_p)\}_{[n]\in \Delta}$.
	We endow $\{C^{h\dagger}(G^{n},\Qbb_p)\}_{[n]\in \Delta}$ with the trivial $G$-action and $\{C^{h\dagger}(G^{n+1},\Qbb_p)\}_{[n]\in \Delta}$ with a $G$-action defined by the left translation on the first component of $G^{n+1}$, that is, the action is given by $(h\cdot f)(g_1,\ldots,g_{n+1})=f(h^{-1}g_1,g_2,\ldots,g_{n+1})$ for $h\in G$ and $f\in C^{h\dagger}(G^{n+1},\Qbb_p)$.
	Then 
	$$\{C^{h\dagger}(G^{n},\Qbb_p)\}_{[n]\in \Delta}\to \{C^{h\dagger}(G^{n+1},\Qbb_p)\}_{[n]\in \Delta}$$ 
	becomes a morphism of cosimplicial objects in $\CAlg(\Rep_{\square}(G))$.
    This morphism induces an equivalence $\{C^{h\dagger}(G^{n},\Qbb_p)\}_{[n]\in \Delta}\to \{C^{h\dagger}(G^{n+1},\Qbb_p)^G\}_{[n]\in \Delta}$ of cosimplicial $\Qbb_{p,\square}$-algebras, where $(-)^G=\Gamma(G,-)$. 
	Let $\Mod_{\{C^{h\dagger}(G^{n},\Qbb_p)\}}(\Dcal(\Qbb_{p,\square}))$ (resp. $\Mod_{\{C^{h\dagger}(G^{n+1},\Qbb_p)\}}(\Rep_{\square}(G))$) be the $\infty$-category of cosimplicial $\{C^{h\dagger}(G^{n},\Qbb_p)\}_{[n]\in \Delta}$-modules in $\Dcal(\Qbb_{p,\square})$ (resp. cosimplicial $\{C^{h\dagger}(G^{n+1},\Qbb_p)\}_{[n]\in \Delta}$-modules in $\Rep_{\square}(G)$).
	Then we have functors 
	$$\begin{array}{rccc}
	\alpha \colon &\Rep_{\square}^{h\dagger}(G)&\longrightarrow &\Mod_{\{C^{h\dagger}(G^{n+1},\Qbb_p)\}}(\Rep_{\square}(G))\\
	 &\rotatebox{90}{$\in$} & &\rotatebox{90}{$\in$}\\
	 &V &\longmapsto &\{C^{h\dagger}(G^{n+1},\Qbb_p)\otimes_{\Qbb_{p,\square}}V\}_{[n]\in \Delta},
	\end{array}$$
	and 
	$$\begin{array}{rccc}
	\beta \colon &\Mod_{\{C^{h\dagger}(G^{n+1},\Qbb_p)\}}(\Rep_{\square}(G))&\longrightarrow &\Mod_{\{C^{h\dagger}(G^{n},\Qbb_p)\}}(\Dcal(\Qbb_{p,\square}))\\
	 &\rotatebox{90}{$\in$} & &\rotatebox{90}{$\in$}\\
	 &\{V_n\}_{[n]\in\Delta} &\longmapsto &\{V_n^G\}_{[n]\in \Delta}.
	\end{array}$$
	We denote the $m$-th projection by
	$$p_m \colon \Mod_{\{C^{h\dagger}(G^{n},\Qbb_p)\}}(\Dcal(\Qbb_{p,\square})) \to \Dcal(C^{h\dagger}(G^{m},\Qbb_p)_{\square});\; \{V_n\}_{[n]\in \Delta}\mapsto V_m$$
	for $m\geq 0$.
	We prove the following:
	\begin{enumerate}
		\item The lax symmetric monoidal functor 
		$$\begin{array}{rccc}
		p_m\circ\beta\circ\alpha \colon &\Rep_{\square}^{h\dagger}(G)& \longrightarrow &\Dcal(C^{h\dagger}(G^{m},\Qbb_p)_{\square})\\
		 & \rotatebox{90}{$\in$}& &\rotatebox{90}{$\in$}\\
		 & V &\longmapsto& C^{h\dagger}(G^{m+1},V)^G
		\end{array}$$
		is symmetric monoidal.
		\item For every morphism $\gamma \colon [l]\to[m]$ in $\Delta$, the natural morphism
		$$\gamma^*(p_l\beta\alpha(V))\to p_m\beta\alpha(V)$$
		is an equivalence for every $V\in \Rep_{\square}^{h\dagger}(G)$, where $\gamma^*$ is the scalar extension functor along $\gamma\colon C^{h\dagger}(G^{l},\Qbb_p)\to C^{h\dagger}(G^{m},\Qbb_p)$.
	\end{enumerate}
	Let us prove (1).
	The morphism $s_n^0\colon C^{h\dagger}(G^{n+1},\Qbb_p) \to C^{h\dagger}(G^{n},\Qbb_p)$ induces
	\begin{align*}
	C^{h\dagger}(G,C^{h\dagger}(G^{n},V))^G\simeq C^{h\dagger}(G^{n+1},V)^G 
	\to C^{h\dagger}(G^{n},V)
	\end{align*} 
	and it is an equivalence since $C^{h\dagger}(G^{n},V)$ is an $h\dagger$-analytic $G$-representation.
	Therefore, $p_m\circ\beta\circ\alpha$ is equivalent to the scalar extension functor
	$$\Rep_{\square}^{h\dagger}(G)\to \Dcal(C^{h\dagger}(G^{m},\Qbb_p)_{\square}) ;\; V \mapsto C^{h\dagger}(G^{n},\Qbb_p)\otimes_{\Qbb_{p,\square}}V,$$ 
	and it is symmetric monoidal.
	Next, let us prove (2).
	It suffices to prove the claim for $\gamma=d_n^i, s_m^j$.
	For $m\geq 0$ and $0\leq j \leq m$, we have the following commutative diagram:
	$$
	\xymatrix{
		C^{h\dagger}(G^{m+2},V)^G\ar[r]^-{s_{m+1}^0}_-{\simeq}\ar[d]_-{s_{m+1}^{j+1}} & C^{h\dagger}(G^{m+1},V)\ar[d]^-{s_m^j\otimes \id_V}\\
		C^{h\dagger}(G^{m+1},V)^G \ar[r]^-{s_{m}^0}_-{\simeq}& C^{h\dagger}(G^{m},V).
	}
	$$
	Therefore, we get (2) in the case when $\gamma=s_m^j$. By the same argument, we get (2) in the case when $\gamma=d_n^i$ for $n\geq 1$ and $0< i \leq n$.
	We prove (2) in the case when $\gamma=d_n^0$.
	In the same way as in the proof of \cite[Proposition 1.62]{Mikami24}, we can construct an equivalence of $C^{h\dagger}(G^{n},\Qbb_p)$-modules 
	$$
	\begin{array}{rccc}
	\delta_V\colon &C^{h\dagger}(G^{n},V)& \overset{\sim}{\longrightarrow}&C^{h\dagger}(G^{n},V)\\
	& \rotatebox{90}{$\in$} & &\rotatebox{90}{$\in$} \\
	& f &\longmapsto &((g_1,\ldots,g_n)\mapsto g_1f(g_1,\ldots,g_n)).
	\end{array}
	$$
	Then we get the following commutative diagram:
	$$
	\xymatrix{
		C^{h\dagger}(G^{n},V)^G\ar[r]^-{s_{n-1}^0}_-{\simeq}\ar[dd]_-{d_{n+1}^1} & C^{h\dagger}(G^{n-1},V)\ar[d]^-{d_n^0\otimes \id_V}\\
		& C^{h\dagger}(G^{n-1},V)\ar[d]^-{\delta_V}_{\simeq}\\
		C^{h\dagger}(G^{n+1},V)^G \ar[r]^-{s_{n}^0}_-{\simeq}&C^{h\dagger}(G^{n},V).
	}
	$$
	Therefore, we get (2) in the case when $\gamma=d_n^0$.

	From (1) and (2), we get a symmetric monoidal functor $$\Rep_{\square}^{h\dagger}(G)\to \varprojlim_{[n]\in \Delta}\Dcal(C^{h\dagger}(G^{n},\Qbb_p)_{\square})\simeq \Dcal(\AnSpec\Qbb_{p,\square}/\Gbb^{h\dagger}).$$
	We check the conditions (1) and (2) in Lemma \ref{lem:Barr-Beck-Lurie descent}.
	The condition (1) follows from Proposition \ref{prop:comonadic}.
	Let us check the condition (2).
	For any morphism $\epsilon \colon [m]\to [n]$ in $\Delta$, the morphism $\epsilon \colon C^{h\dagger}(G^{m},\Qbb_p)_{\square} \to C^{h\dagger}(G^{n},\Qbb_p)_{\square}$ is steady by \cite[Proposition 2.3.19 (ii)]{Mann22}, where we note that the source and target are steady analytic $\Qbb_{p,\square}$-algebra by Lemma \ref{lem:nuclearness}.
	Therefore, the condition (2) is satisfied when $m\geq 0$.
	Moreover, the diagram
	$$
	\xymatrix{
		\Rep_{\square}^{h\dagger}(G)\ar[r]\ar[d]& \Dcal(\Qbb_{p,\square})\ar[d]\\
		\Dcal(\Qbb_{p,\square})\ar[r]^-{d_{1}^0} & \Dcal(C^{h\dagger}(G^{1},\Qbb_p)_{\square})
	}
	$$
	is right adjointable by Proposition \ref{prop:adjoint cofree}.
	By combining the above, we find that the condition (2) is satisfied for every $\alpha\colon [m]\to[n]$ in $\Delta_+$. 
	Therefore, $\Rep_{\square}^{h\dagger}(G)\to \Dcal(\AnSpec\Qbb_{p,\square}/\Gbb^{h\dagger})$
	is an equivalence of $\infty$-categories, and consequently, it is also an equivalence of symmetric monoidal $\infty$-categories.
\end{proof}

\begin{remark}
	Roughly speaking, the equivalence $\Rep_{\square}^{h\dagger}(G) \overset{\simeq}{\to} \Dcal(\AnSpec\Qbb_{p,\square}/\Gbb^{h\dagger})$ is given by $V\mapsto \{C^{h\dagger}(G^{n},\Qbb_p)\otimes_{\Qbb_{p,\square}}V\}_{[n]\in\Delta}=\{C^{h\dagger}(G^{n},V)\}_{[n]\in\Delta}$, where the coface maps and the codegeneracy maps are described as follows:
	\begin{itemize}
		\item For $1\leq n$ and $0\leq i\leq n$, the coface map $d_n^i$ is given by 
		$$\begin{array}{ccl}
			C^{h\dagger}(G^{n-1},V) &\longrightarrow &  C^{h\dagger}(G^{n},V)\\
	        \rotatebox{90}{$\in$}       &                &  \quad\quad\rotatebox{90}{$\in$}\\
			f                            &\longmapsto     & \left( (g_1,\ldots,g_n) \mapsto
						                                                   \begin{cases}
							                                                g_1f(g_2,\ldots,g_n) &(i=0) \\
							                                                f(g_1,\ldots,g_{i}g_{i+1},\ldots, g_n) &(0<i<n)\\
						                                                    f(g_{1},\ldots,g_{n-1}) &(i=n) 
						                                                  \end{cases}\right).
		\end{array}$$
		\item For $0\leq n$ and $0\leq i\leq n$, the codegeneracy map $s_n^i$ is given by 
		$$\begin{array}{ccl}
			C^{h\dagger}(G^{n+1},V) &\longrightarrow &  C^{h\dagger}(G^{n},V)\\
	        \rotatebox{90}{$\in$}       &                &  \quad\quad\rotatebox{90}{$\in$}\\
			f                            &\longmapsto     &  \left((g_1,\ldots,g_n) \mapsto f(g_1,\ldots,g_{i},1,g_{i+1},\ldots,g_n)\right).	                                                
		\end{array}$$
	\end{itemize}
	The proof of Proposition \ref{prop:geometric interpretation} provides a precise formulation
\end{remark}

\subsection{Locally analytic representations and six functors}
We use the same notation as in the previous subsection.
In this subsection, we interpret functors related to locally analytic representations in terms of the 6-functor formalism.
Continuing from the previous subsection, let $h$ be a positive rational number or $\infty$. 
Let $p_h\colon \AnSpec\Qbb_{p,\square} \to \AnSpec\Qbb_{p,\square}/\Gbb^{h\dagger}$ and $f_h\colon \AnSpec\Qbb_{p,\square}/\Gbb^{h\dagger}\to \AnSpec\Qbb_{p,\square}$ denote the natural morphisms of solid $\Dcal$-stacks.
If there is no room for confusion, we will drop the index $h$.
For $h^{\prime}\geq h$, let $f_{h^{\prime},h}\colon \AnSpec\Qbb_{p,\square}/\Gbb^{h^{\prime}\dagger}\to\AnSpec\Qbb_{p,\square}/\Gbb^{h\dagger}$ denote the natural morphism.
In the following, we denote $\AnSpec\Qbb_{p,\square}$ by $*$ and fiber products $- \times_{\AnSpec\Qbb_{p,\square}}-$ by $-\times -$.

\begin{proposition}\label{prop:descendable cover}
	\begin{enumerate}
		\item The functor $$p^*\colon \Rep_{\square}^{h\dagger}(G)\simeq \Dcal(*/\Gbb^{h\dagger}) \to \Dcal(\Qbb_{p,\square})$$ is the forgetful functor.
		\item The functor 
		$$p_*\colon \Dcal(\Qbb_{p,\square})\to \Dcal(*/\Gbb^{h\dagger})\simeq \Rep_{\square}^{h\dagger}(G)$$ 
		is given by $V\mapsto C^{h\dagger}(G,V)_{\star_2}$.
		\item The morphism $p\colon * \to */\Gbb^{h\dagger}$ is weakly $\Dcal$-proper. In particular, we have an equivalence $p_*\simeq p_!$.
		\item The $\Ebb_{\infty}$-algebra $p_*\Qbb_p=C^{h\dagger}(G,\Qbb_p)_{\star_2}$ in $\Rep_{\square}^{h\dagger}(G)$ is descendable.
		\item The morphism $p\colon * \to */\Gbb^{h\dagger}$ is a $\Dcal$-cover.
	\end{enumerate}
\end{proposition}
\begin{proof}
	Part (1) easily follows from the construction in Proposition \ref{prop:geometric interpretation}.
	Since $p_*$ is a right adjoint functor of $p^*$, (2) follows from Proposition \ref{prop:adjoint cofree}.
	Part (3) follows from Corollary \ref{cor:weakly proper}.
	Let us prove (4).
	It suffices to show that $C^{h\dagger}(G,\Qbb_p)_{\star_2}$ is descendable in $\Rep_{\square}(G)$.
	By Lemma \ref{lem:descendable zero map}, it is enough to show that $C^{h\dagger}(G,\Qbb_p)_{\star_2}$ is descendable in $\Dcal(\Qbb_{p,\square})$, where we use a characterization of descendable algebras in \cite[Proposition 3.20]{Mat16}.
	In this case, the natural morphism $\Qbb_p\to C^{h\dagger}(G,\Qbb_p)$ admits a splitting, so we get the claim. 
	Part (5) follows from (3) and (4) by Lemma \ref{lem:prim descendable local}. 
\end{proof}

\begin{proposition}\label{prop:compact underlying module}
	\begin{enumerate}
		\item An $h\dagger$-analytic representation $V\in \Rep_{\square}^{h\dagger}(G)$ is dualizable if and only if the underlying $\Qbb_{p,\square}$-module $V\in \Dcal(\Qbb_{p,\square})$ is dualizable.
		\item An $h\dagger$-analytic representation $V\in \Rep_{\square}^{h\dagger}(G)$ is compact if and only if the underlying $\Qbb_{p,\square}$-module $V\in \Dcal(\Qbb_{p,\square})$ is compact.
	\end{enumerate}
\end{proposition}
\begin{proof}
Part (1) follows from Remark \ref{rem:suave and primand dualizable}, Lemma \ref{lem:prim and suave obj D-local}, and Proposition \ref{prop:descendable cover}.
Let us prove (2).
Since $p_*\colon \Dcal(\Qbb_{p,\square})\to \Dcal(*/\Gbb^{h\dagger})\simeq \Rep_{\square}^{h\dagger}(G)$ preserves all small colimits, the left adjoint functor $p^*$ of $p_*$ preserves compact objects, which proves the ``only if'' direction.
We take $V\in \Rep_{\square}^{h\dagger}(G)$ such that the underlying $\Qbb_{p,\square}$-module $V\in \Dcal(\Qbb_{p,\square})$ is compact.
We need to show that for any filtered diagram $\{M_{\lambda}\}_{\lambda\in \Lambda}$ in $\Rep_{\square}^{h\dagger}(G)$, the natural morphism
$$\varinjlim_{\lambda}\Hom_G(V,M_{\lambda})\to \Hom_G(V,\varinjlim_{\lambda}M_{\lambda})$$
is an equivalence.
Let $\Ccal\subset \Rep_{\square}^{h\dagger}(G)$ be the full subcategory consisting of $h\dagger$-analytic representations $N$ such that 
$$\varinjlim_{\lambda}\Hom_G(V,M_{\lambda}\otimes_{\Qbb_{p,\square}}N)\to \Hom_G(V,\varinjlim_{\lambda}M_{\lambda}\otimes_{\Qbb_{p,\square}}N)$$
is an equivalence for any filtered diagram $\{M_{\lambda}\}_{\lambda\in \Lambda}$ in $\Rep_{\square}^{h\dagger}(G)$.
It suffices to show that the trivial representation $\Qbb_p$ lies in $\Ccal$.
We note that the full subcategory $\Ccal$ is stable under finite limits, finite colimits, tensor products, and direct summands.
By Proposition \ref{prop:adjoint cofree} and Remark \ref{rem:orbit}, we have 
\begin{align*}
	&\varinjlim_{\lambda}\Hom_G(V,M_{\lambda}\otimes_{\Qbb_{p,\square}}C^{h\dagger}(G,\Qbb_p)_{\star_2})\\
	\simeq &\varinjlim_{\lambda}\Hom_G(V,C^{h\dagger}(G,M_{\lambda})_{\star_2})\\
	\simeq &\varinjlim_{\lambda}\Hom_{\Qbb_p}(V,M_{\lambda})\\
	\simeq &\Hom_{\Qbb_p}(V,\varinjlim_{\lambda}M_{\lambda})\\
	\simeq &\Hom_G(V,C^{h\dagger}(G,\varinjlim_{\lambda}M_{\lambda})_{\star_2})\\
	\simeq &\Hom_G(V,\varinjlim_{\lambda}M_{\lambda}\otimes_{\Qbb_{p,\square}}C^{h\dagger}(G,\Qbb_p)_{\star_2}).
\end{align*}
Therefore, we find that $C^{h\dagger}(G,\Qbb_p)_{\star_2} \in \Ccal$.
Since $C^{h\dagger}(G,\Qbb_p)_{\star_2}$ is a descendable $\Ebb_{\infty}$-algebra in $\Rep_{\square}^{h\dagger}(G)$ by Proposition \ref{prop:descendable cover} (4), the trivial representation $\Qbb_p$ also lies in $\Ccal$.
\end{proof}

\begin{lemma}\label{lem:compact h-analytic}
	Let $V \in \Rep_{\square}^{\la}(G)$ be a compact object.
	Then there exists a positive rational number $h$ such that $V$ lies in $\Rep_{\square}^{h\dagger}(G)$.
\end{lemma}
\begin{proof}
	We have an equivalence $V\simeq \varinjlim_{h>0}V^{h\dagger\mathchar`-\an}$ in $\Rep_{\square}^{\la}(G).$
	Since $V$ is compact, there exists a positive rational number $h$ such that $V$ is a direct summand of $V^{h\dagger\mathchar`-\an}\in \Rep_{\square}^{h\dagger}(G)$.
	The full subcategory $\Rep_{\square}^{h\dagger}(G)\subset \Rep_{\square}^{\la}(G)$ is stable under direct summands, so we get the claim.
\end{proof}

\begin{proposition}\label{prop:6-ff classifying stack}
	\begin{enumerate}
		\item The functor 
		$$f^*\colon \Dcal(\Qbb_{p,\square})\to \Dcal(*/\Gbb^{h\dagger})\simeq \Rep_{\square}^{h\dagger}(G)$$ 
		is equivalent to the functor sending $V\in \Dcal(\Qbb_{p,\square})$ to $V$ with the trivial $G$-action.
		\item The functor 
		\begin{align*}
			f_{h^{\prime},h}^*\colon \Rep_{\square}^{h\dagger}(G)\simeq &\Dcal(*/\Gbb^{h\dagger})\\
			\to &\Dcal(*/\Gbb^{h^{\prime}\dagger})\simeq \Rep_{\square}^{h^{\prime}\dagger}(G)
		\end{align*}
		is equivalent to the natural inclusion $\Rep_{\square}^{h\dagger}(G) \hookrightarrow \Rep_{\square}^{h^{\prime}\dagger}(G)$.
		\item The functor $f_*\colon \Rep_{\square}^{h\dagger}(G)\simeq \Dcal(*/\Gbb^{h\dagger}) \to \Dcal(\Qbb_{p,\square})$ is given by $V\mapsto V^G=\Gamma(G,V)$.
		\item The functor 
		\begin{align*}
			(f_{h^{\prime},h})_*\colon \Rep_{\square}^{h^{\prime}\dagger}(G)\simeq &\Dcal(*/\Gbb^{h^{\prime}\dagger})\\
			\to &\Dcal(*/\Gbb^{h\dagger})\simeq \Rep_{\square}^{h\dagger}(G)
		\end{align*}
		is given by $V\mapsto V^{h\dagger\mathchar`-\an}$.
		\item The morphisms 
		$$f_h\colon */\Gbb^{h\dagger}\to *$$ 
		and 
		$$f_{h^{\prime},h}\colon */\Gbb^{h^{\prime}\dagger}\to */\Gbb^{h\dagger}$$ 
		are weakly $\Dcal$-proper.
		In particular, we have equivalences $(f_h)_*\simeq (f_h)_!$ and $(f_{h^{\prime},h})_*\simeq (f_{h^{\prime},h})_!$.
		\item There exists a positive rational number $h^{\prime}$ such that the morphism 
		$$f_h\colon */\Gbb^{h\dagger}\to *$$
		is $\Dcal$-smooth for $h\geq h^{\prime}$.
		The dualizing complex $\omega_f\in \Rep_{\square}^{h\dagger}(G)$ of $f$ is given by $\wedge^d \gfrak[d]$, where we endow $\gfrak$ with the adjoint action of $G$.
		Moreover, if $G$ is abelian, then $f_h$ is $\Dcal$-smooth for every positive rational number $h$.
	\end{enumerate}
\end{proposition}
\begin{proof}
	Parts (1) and (2) easily follow from the construction in Proposition \ref{prop:geometric interpretation}.
	Since $f_*$ (resp. $(f_{h^{\prime},h})_*$) is a right adjoint functor of $f^*$ (resp. $f_{h^{\prime},h}^*$), we get (3) and (4).

	Let us prove (5). 
	First, the diagonal morphism $\Delta_h\colon */\Gbb^{h\dagger} \to */\Gbb^{h\dagger}\times */\Gbb^{h\dagger}\simeq */(\Gbb\times \Gbb)^{h\dagger}$ of $f_h$ is, universally $\Dcal^*$-locally on the target, representable by a morphism in $P$, so $\Delta_h$ is weakly $\Dcal$-proper by Corollary \ref{cor:weakly proper}. 
	Next, we show that $f_h$ is $\Dcal$-prim.
	It follows from Proposition \ref{prop:descendable cover} by applying Lemma \ref{lem:prim descendable local} to 
	$$* \overset{p}{\to} */\Gbb^{h\dagger}\overset{f_h}{\to} *,$$
	where we note that the identity morphism is always $\Dcal$-prim.
	Since there is an equivalence $\delta_{f_h}\simeq \pi_{2*}\Delta_{h!}\Qbb_p\simeq \pi_{2*}\Delta_{h*}\Qbb_p\simeq \Qbb_p$, where $\pi_{2}$ is the second projection $*/\Gbb^{h\dagger}\times */\Gbb^{h\dagger}\to */\Gbb^{h\dagger}$, we find that $f_h$ is weakly $\Dcal$-proper.
	Since $\ast/\Gbb^{h^{\prime}\dagger} \to \ast$ and the diagonal morphism $\Delta_h$ of $\ast/\Gbb^{h\dagger} \to \ast$ are weakly $\Dcal$-proper, $f_{h^{\prime},h}$ is also weakly $\Dcal$-proper by Lemma \ref{lem:proper composition basechange} and the standard argument.

	Let us prove (6). 
	If $h=\infty$, then the claim is proven in \cite[Proposition 6.2.4]{RC24}.
	By Proposition \ref{prop:compact underlying module} and Lemma \ref{lem:compact h-analytic}, we can take a positive rational number $h^{\prime}$ such that $\wedge^d \gfrak[d]\in \Rep_{\square}^{h\dagger}(G)$ for $h\geq h^{\prime}$.
	We note that if $G$ is abelian, $\wedge^d \gfrak[d]$ lies in $\Rep_{\square}^{h\dagger}(G)$ for every positive rational number $h$, since the adjoint action of $G$ on $\gfrak$ is trivial.
	Then (6) can be obtained by applying Lemma \ref{lem:suave and prim composition} to 
	$(Y\overset{g}{\to}X\overset{f}{\to}S)=(*/\Gbb^{\la}\to */\Gbb^{h\dagger} \to *)$ and $P=Q=\Qbb_p$, where we note $(f_{\infty,h})_! \wedge^d \gfrak[d] =\wedge^d \gfrak[d]$ since $\wedge^d \gfrak[d]$ is $h\dagger$-analytic.
\end{proof}

\subsection{Semilinear representations}
We use the same notation as in the previous subsection.
In this subsection, we define a locally analytic action of $G$ on $\Acal\in \AffRing_{\Qbb_{p,\square}}$ and study locally analytic semilinear representations over $\Acal$.
Continuing from the previous subsection, let $h$ be a positive rational number or $\infty$. 

\begin{definition}
	Let $\Acal$ be a solid affinoid animated $\Qbb_{p,\square}$-algebra.
	\textit{An $h\dagger$-analytic action of $G$ on $\Acal$} is a (right) action of the group object $\Gbb^{h\dagger}=\AnSpec C^{h\dagger}(G,\Qbb_p)_{\square}$ in $\Aff_{\Qbb_{p,\square}}$ on $\AnSpec \Acal\in \Aff_{\Qbb_{p,\square}}$, or equivalently, a groupoid object $\{ \AnSpec \Acal\times (\Gbb^n)^{h\dagger}\}_{[n]\in \Delta}$ such that $d_1^1 \colon \AnSpec \Acal\times \Gbb^{h\dagger} \to \AnSpec \Acal$ is the natural projection, and such that the projection gives a morphism 
	$$\{ \AnSpec \Acal\times (\Gbb^n)^{h\dagger}\}_{[n]\in \Delta}\to \{(\Gbb^n)^{h\dagger}\}_{[n]\in \Delta}$$ 
	of groupoid objects (see \cite[Definition 3.1]{NSS14}).
	Let $\AnSpec \Acal/ \Gbb^{h\dagger}$ denote its geometric realization in $\Shv_{\Dcal}(\Aff_{\Qbb_{p,\square}})$.
\end{definition}

\begin{lemma}\label{lem:steady semilinear}
	Let $\Acal$ be a solid affinoid animated $\Qbb_{p,\square}$-algebra with an $h\dagger$-analytic action of $G$.
	Let $\{ \AnSpec \Acal\times (\Gbb^n)^{h\dagger}\}_{[n]\in \Delta}$ be the corresponding cosimplicial object.
	Then, for any morphism $\alpha \colon [m]\to [n]$ in $\Delta$, the morphism 
	$$\AnSpec \Acal\times(\Gbb^n)^{h\dagger} \to \AnSpec \Acal\times(\Gbb^m)^{h\dagger}$$
	corresponding to $\alpha$ is steady.
\end{lemma}
\begin{proof}
	The morphism $\alpha \colon (\Gbb^n)^{h\dagger} \to (\Gbb^m)^{h\dagger}$ is steady by Lemma \ref{lem:nuclearness} and \cite[Proposition 2.3.19 (ii)]{Mann22}.
	Therefore, the claim follows from the fact that the natural morphism 
	$$\{ \AnSpec \Acal\times (\Gbb^n)^{h\dagger}\}_{[n]\in \Delta}\to \{(\Gbb^n)^{h\dagger}\}_{[n]\in \Delta}$$
	of groupoid objects is cartesian in the sense of \cite[definition 6.1.3.1]{HTT}.
\end{proof}

\begin{corollary}\label{cor:underlying module}
	Let $f\colon \AnSpec\Acal\to\AnSpec\Bcal$ be a $\Gbb^{h\dagger}$-equivariant morphism of solid affinoid spaces over $\AnSpec\Qbb_{p,\square}$ with an $h\dagger$-analytic action of $G$.
	Let us consider the cartesian diagram 
$$
\xymatrix{
	\AnSpec \Acal \ar[r]^-{f}\ar[d]_-{p_{\Acal}} & \AnSpec \Bcal\ar[d]^-{p_{\Bcal}}\\
	\AnSpec \Acal/ \Gbb^{h\dagger}\ar[r]^-{\overline{f}} & \AnSpec \Bcal/ \Gbb^{h\dagger}.
}
$$
Then the natural morphism $p_{\Bcal}^*\overline{f}_*\to f_*p_{\Acal}^{*}$ of functors from $\Dcal(\AnSpec \Acal/ \Gbb^{h\dagger})$ to $\Dcal(\Bcal)$ is an equivalence.
\end{corollary}
\begin{proof}
Let $\{f_n\}_{[n]\in \Delta}\colon \{ \AnSpec \Acal\times (\Gbb^n)^{h\dagger}\}_{[n]\in \Delta}\to\{ \AnSpec \Bcal\times (\Gbb^n)^{h\dagger}\}_{[n]\in \Delta}$ denote the morphism corresponding to the $\Gbb^{h\dagger}$-equivariant morphism $f\colon \AnSpec\Acal\to\AnSpec\Bcal$.
By Lemma \ref{lem:steady semilinear}, for any morphism $\alpha \colon [m]\to [n]$ in $\Delta$, the morphism 
$$\AnSpec \Bcal\times(\Gbb^n)^{h\dagger} \to \AnSpec \Bcal\times(\Gbb^m)^{h\dagger}$$
corresponding to $\alpha$ is steady.
Therefore, the functor 
$$\{f_{n*}\}_{[n]\in \Delta}\colon \{\Dcal(\AnSpec \Acal\times(\Gbb^n)^{h\dagger})\}_{[n]\in \Delta}\to \{\Dcal(\AnSpec \Bcal\times(\Gbb^n)^{h\dagger})\}_{[n]\in \Delta}$$ 
preserves cocartesian sections.
The claim follows from this.
\end{proof}

Let $\AnSpec\Acal$ be a solid affinoid space with an $h\dagger$-analytic action of $G$, and let $f\colon \AnSpec\Acal/\Gbb^{h\dagger}\to \AnSpec\Qbb_{p,\square}/\Gbb^{h\dagger}$ denote the natural morphism.
We set $(A,A^+)=(\underline{\Acal},\pi_0(\Acal^+))$.
By applying Corollary \ref{cor:underlying module} to $\Bcal=\Qbb_{p,\square}$, we find that the underlying $\Ebb_{\infty}$-algebra of $f_*A$ is equal to $A$.
Therefore, we can regard an animated $\Qbb_{p,\square}$-algebra $A$ as a (connective) $\Ebb_{\infty}$-algebra in $\Dcal(\AnSpec \Qbb_{p,\square}/ \Gbb^{h\dagger})=\Rep_{\square}^{h\dagger}(G)$.

Conversely, we want to construct an $h\dagger$-analytic action of $G$ on $\AnSpec(A,A^+)_{\square}$ from a connective $\Ebb_{\infty}$-algebra in $\Rep_{\square}^{h\dagger}(G)$ and a suitable subring $A^+\subset \pi_0(A)$.
We note that, since we are working with $\Qbb_{p,\square}$-algebras, there is no difference between connective $\Ebb_{\infty}$-algebras and animated algebras.

\begin{lemma}\label{lem:la action on animated rings}
	Let $A$ be a connective $\Ebb_{\infty}$-algebra in $\Rep_{\square}^{h\dagger}(G)$.
	Then it gives rise to an $h\dagger$-analytic action of $G$ on $\AnSpec A_{\square}$.
\end{lemma}
\begin{proof}
	By the proof of Proposition \ref{prop:geometric interpretation}, we get a morphism 
	$$\{C^{h\dagger}(G^n,\Qbb_p)\}_{[n]\in \Delta}\to \{C^{h\dagger}(G^n,\Qbb_p)\otimes_{\Qbb_{p,\square}}A\}_{[n]\in \Delta}$$ 
	of cosimplicial connective $\Ebb_{\infty}$-algebras in $\Dcal(\Qbb_{p,\square})$.
	Since $\Qbb_p$ is a $\Qbb$-algebra, the above morphism becomes a morphism of cosimplicial animated $\Qbb_{p,\square}$-algebras.
	By endowing with the induced analytic ring structure from $\Qbb_{p,\square}$, we get a morphism 
	$$\{C^{h\dagger}(G^n,\Qbb_p)_{\square}\}_{[n]\in \Delta}\to \{C^{h\dagger}(G^n,\Qbb_p)_{\square}\otimes_{\Qbb_{p,\square}}A_{\square}\}_{[n]\in \Delta}$$
	of cosimplicial objects in $\AffRing_{\Qbb_{p,\square}}$, which gives an $h\dagger$-analytic action of $G$ on $\AnSpec A_{\square}$.
\end{proof}

\begin{proposition}\label{prop:la action on solid affinoid space}
	Let $A$ be a connective $\Ebb_{\infty}$-algebra in $\Rep_{\square}^{h\dagger}(G)$, and let $A^+\subset \pi_0(A)$ be a discrete subring such that for every $x\in A^+$, $A$ is $\Zbb[x]_{\square}$-complete.
	Let $\{C^{h\dagger}(G^n,\Qbb_p)\otimes_{\Qbb_{p,\square}}A\}_{[n]\in \Delta}$ be the cosimplicial animated $\Qbb_{p,\square}$-algebra corresponding to the $h\dagger$-analytic action of $G$ on $A$.
	We assume that the morphism 
	$$d_1^0 \colon A \to C^{h\dagger}(G,\Qbb_p)\otimes_{\Qbb_{p,\square}}A$$
	of animated $\Qbb_{p,\square}$-algebras becomes a morphism 
	$$d_1^0 \colon (A,A^+)_{\square} \to C^{h\dagger}(G,\Qbb_p)_{\square}\otimes_{\Qbb_{p,\square}}(A,A^+)_{\square}$$
	of solid affinoid animated $\Qbb_{p,\square}$-algebras.
	Then the cosimplicial animated $\Qbb_{p,\square}$-algebra 
	$$\{C^{h\dagger}(G^n,\Qbb_p)\otimes_{\Qbb_{p,\square}}A\}_{[n]\in \Delta}$$ 
	becomes a cosimplicial solid affinoid animated $\Qbb_{p,\square}$-algebra 
	$$\{C^{h\dagger}(G^n,\Qbb_p)_{\square}\otimes_{\Qbb_{p,\square}}(A,A^+)_{\square}\}_{[n]\in \Delta},$$ 
	and therefore, we obtain an $h\dagger$-analytic action of $G$ on $\AnSpec(A,A^+)_{\square}$.
	Conversely, any solid affinoid space over $\AnSpec \Qbb_{p,\square}$ with an $h\dagger$-analytic action of $G$ is uniquely obtained by the above construction. 
\end{proposition}
\begin{proof}
	Let $\alpha \colon [m]\to [n]$ in $\Delta$ be one of the morphisms $d_n^i\colon [n-1]\to [n]$ ($n\geq 1$, $0\leq i \leq n$) or $s_n^i\colon [n+1]\to [n]$  ($n\geq 0$, $0\leq i \leq n$).
	To prove the first claim, it suffices to show that  the morphism 
	\begin{align}\label{mor1}
	\alpha \colon C^{h\dagger}(G^m,\Qbb_p)\otimes_{\Qbb_{p,\square}}A \to C^{h\dagger}(G^n,\Qbb_p)\otimes_{\Qbb_{p,\square}}A
	\end{align}
	of animated $\Qbb_{p,\square}$-algebras becomes a morphism 
	$$\alpha \colon C^{h\dagger}(G^m,\Qbb_p)_{\square}\otimes_{\Qbb_{p,\square}}(A,A^+)_{\square} \to C^{h\dagger}(G^n,\Qbb_p)_{\square}\otimes_{\Qbb_{p,\square}}(A,A^+)_{\square}$$
	of solid affinoid $\Qbb_{p,\square}$-algebras.
	First, we assume that $\alpha$ is not $d_n^0$ ($n\geq 1$).
	In this case, the morphism \eqref{mor1} is the tensor product of $\alpha \colon C^{h\dagger}(G^m,\Qbb_p)_{\square} \to C^{h\dagger}(G^n,\Qbb_p)_{\square}$ and $\id_A$ by the proof of Proposition \ref{prop:geometric interpretation}, so we get the claim.
	Next, we assume $\alpha=d_n^0 \colon [n-1]\to [n]$ for some $n\geq 1$.
	In this case, we have a cocartesian diagram
	$$
	\xymatrix{
		C^{h\dagger}(G^{n-1},\Qbb_p)_{\square}\otimes_{\Qbb_{p,\square}}A\ar[r]^-{d_n^0} & C^{h\dagger}(G^n,\Qbb_p)_{\square}\otimes_{\Qbb_{p,\square}}A \\
		A\ar[r]_-{d_1^0}\ar[u] & C^{h\dagger}(G,\Qbb_p)_{\square}\otimes_{\Qbb_{p,\square}}A,\ar[u]
	}
	$$
	where the vertical morphisms correspond to the natural inclusion $[0]\to [n-1]$ and $[1]\to[n]$ respectively.
	Therefore, the claim follows from the assumption on $d_1^0$.

	Conversely, let $\AnSpec\Acal$ be a solid affinoid space over $\AnSpec \Qbb_{p,\square}$ with an $h\dagger$-analytic action of $G$.
	Then $(A,A^+)=(\underline{\Acal}, \pi_0(\Acal^+))$ satisfies the assumption of the proposition.
	It is clear that these correspondences give bijections.
\end{proof}

Let us give an example of solid affinoid spaces over $\AnSpec \Qbb_{p,\square}$ with $h\dagger$-analytic actions of $G$.

\begin{lemma}\label{lem:banach algebra with la action}
	Let $A$ be a Banach $\Qbb_p$-algebra with a continuous $G$-action (in the classical sense).
	We assume that the action of $G$ on $A$ is non-derived $h^{\prime}$-analytic for some $h^{\prime}<h$.
	Then it defines an $h\dagger$-analytic action of $G$ on $\AnSpec (A,A^{\circ})_{\square}$, where $A^{\circ}$ is the ring of power-bounded elements.
\end{lemma}
\begin{proof}
	By Lemma \ref{lem:h-an implies h+-an} and Corollary \ref{cor:transitivity}, the action of $G$ on $A$ is $h\dagger$-analytic.
	Therefore, we can regard $A$ as a connective $\Ebb_{\infty}$-algebra in $\Rep_{\square}^{h\dagger}(G)$.
	Let us check the assumption in Proposition \ref{prop:la action on solid affinoid space}.
	By the definition of non-derived $h'$-analytic representations, we have a coaction morphism $A \to C^{h^{\prime}}(G,A)$ of Banach $\Qbb_p$-algebras.
	It defines a morphism 
	$$(A, A^{\circ})_{\square} \to (C^{h^{\prime}}(G,A),C^{h^{\prime}}(G,A)^{\circ})_{\square}=(C^{h^{\prime}}(G,\Qbb_p),C^{h^{\prime}}(G,\Zbb_p))_{\square}\otimes_{\Qbb_{p,\square}}(A,A^{\circ})_{\square}$$ 
	of analytic $\Qbb_{p,\square}$-algebras.
	Therefore, we get a morphism 
	$$(A, A^{\circ})_{\square}\to (C^{h^{\prime}}(G,\Qbb_p),C^{h^{\prime}}(G,\Zbb_p))_{\square}\otimes_{\Qbb_{p,\square}}(A,A^{\circ})_{\square} \to C^{h\dagger}(G,\Qbb_p)_{\square}\otimes_{\Qbb_{p,\square}}(A, A^{\circ})_{\square}$$
	of analytic $\Qbb_{p,\square}$-algebras.
	By construction, its underlying morphism is the coaction morphism $d_1^0\colon A\to  C^{h\dagger}(G,\Qbb_p)\otimes_{\Qbb_{p,\square}}A$, so we get the claim.
\end{proof}

Let us give a presentation of $\Dcal(\AnSpec \Acal/\Gbb^{h\dagger})$ for a solid affinoid space $\AnSpec\Acal$ with an $h\dagger$-analytic action of $G$. 
We put $(A,A^+)=(\underline{\Acal}, \pi_0(\Acal^+))$.
\begin{definition}
Let $\Rep_{\Acal}^{h\dagger}(G)$ denote the full subcategory of $\Mod_A(\Rep_{\square}^{h\dagger}(G))$, where $\Mod_A(\Rep_{\square}^{h\dagger}(G))$ is the $\infty$-category of $A$-module objects in $\Rep_{\square}^{h\dagger}(G)$, consisting of those objects whose underlying $A$-modules are $\Acal$-complete.
We will refer to objects of $\Rep_{\Acal}^{h\dagger}(G)$ as \textit{semilinear $h\dagger$-analytic representations of $G$ over $\Acal$}.
\end{definition}

\begin{proposition}\label{prop:semilinear geometric interpretation}
Let $\AnSpec\Acal$ be a solid affinoid space with an $h\dagger$-analytic action of $G$.
Then there is a natural equivalence of symmetric monoidal $\infty$-categories
$$\Dcal(\AnSpec\Acal/\Gbb^{h\dagger})\simeq \Rep_{\Acal}^{h\dagger}(G).$$
\end{proposition}
\begin{proof}
We put $(A,A^+)=(\underline{\Acal}, \pi_0(\Acal^+))$.
First, we prove the proposition when $\Acal=A_{\square}$.
By Proposition \ref{prop:geometric interpretation}, we have an augmented cosimplicial diagram of symmetric monoidal $\infty$-categories
$\Rep_{\square}^{h\dagger}(G) \to \{\Dcal(C^{h\dagger}(G^n,\Qbb_p)_{\square})\}_{[n]\in \Delta}.$
By taking the image of $A\in \CAlg(\Rep_{\square}^{h\dagger}(G))$, we get the cosimplicial $\{C^{h\dagger}(G^n,\Qbb_p)\}_{[n]\in \Delta}$-algebra $\{C^{h\dagger}(G^n,A)\}_{[n]\in \Delta}$, which corresponds to the $h\dagger$-analytic action of $G$ on $A$.
From this, we get an augmented cosimplicial diagram of symmetric monoidal $\infty$-categories
\begin{align*}
\Mod_A(\Rep_{\square}^{h\dagger}(G)) \to &\{\Mod_{C^{h\dagger}(G^n,A)}(\Dcal(C^{h\dagger}(G^n,\Qbb_p)_{\square}))\}_{[n]\in \Delta}\\
\simeq &\{\Dcal(C^{h\dagger}(G^n,A)_{\square})\}_{[n]\in \Delta}.
\end{align*}
By taking a limit, we get a symmetric monoidal functor $$\Mod_A(\Rep_{\square}^{h\dagger}(G)) \to \Dcal(\AnSpec A_{\square}/\Gbb^{h\dagger}).$$
Let us check the conditions in Lemma \ref{lem:Barr-Beck-Lurie descent}.
By Lemma \ref{lem:steady semilinear}, the condition (2) is satisfied when $m\geq 0$.
The forgetful functor 
$$F\colon\Mod_A(\Rep_{\square}^{h\dagger}(G)) \to \Dcal(A_{\square})$$
admits a right adjoint functor 
$$C^{h\dagger}(G,-)_{\star_2}\colon\Dcal(A_{\square}) \to \Mod_A(\Rep_{\square}^{h\dagger}(G));\; V \mapsto C^{h\dagger}(G,V)_{\star_2},$$
where we regard the $C^{h\dagger}(G,A)_{\star_2}$-module $C^{h\dagger}(G,V)_{\star_2}$ as an $A$-module by using $d_1^0\colon A\to C^{h\dagger}(G,A)_{\star_2}$.
Therefore, it follows from this presentation that the diagram
$$
\xymatrix{
	\Mod_A(\Rep_{\square}^{h\dagger}(G))\ar[r]\ar[d]& \Dcal(A_{\square})\ar[d]\\
	\Dcal(A_{\square})\ar[r]^-{d_{1}^0} & \Dcal(C^{h\dagger}(G,A)_{\square})
}
$$
is right adjointable, which proves that the condition (2) is satisfied.
Since we have a commutative diagram
$$
\xymatrix{
	\Dcal(A_{\square})\ar[rr]^-{C^{h\dagger}(G,-)_{\star_2}}\ar[d]& & \Mod_A(\Rep_{\square}^{h\dagger}(G))\ar[d]\\
	\Dcal(\Qbb_{p,\square}) \ar[rr]^-{C^{h\dagger}(G,-)_{\star_2}}& & \Rep_{\square}^{h\dagger}(G),
}
$$
and the vertical morphisms preserve all small limits and are conservative.
Therefore, the adjunction $(F, C^{h\dagger}(G,-)_{\star_2})$ is comonadic by Proposition \ref{prop:comonadic}, which proves that the condition (1) is satisfied.

Next, let us consider a general $\Acal$.
let 
$$\{f_{n}\}_{[n]\in \Delta}\colon \{\AnSpec\Acal \times (\Gbb^n)^{h\dagger}\}_{n}\to \{\AnSpec A_{\square}\times (\Gbb^n)^{h\dagger}\}_{n}$$ 
denote the natural morphism of the cosimplicial solid affinoid spaces, and let 
$$f_{-1} \colon\AnSpec\Acal/\Gbb^{h\dagger}\to \AnSpec A_{\square}/\Gbb^{h\dagger}$$
denote the natural morphism.
Then the functor $$(f_{-1})_*\colon \Dcal(\AnSpec\Acal/\Gbb^{h\dagger})\to \Dcal(\AnSpec A_{\square}/\Gbb^{h\dagger})$$
is equivalent to the fully faithful functor
$$\varprojlim_{[n]\in \Delta}(f_n)_*\colon \varprojlim_n \Dcal(\AnSpec\Acal \times (\Gbb^n)^{h\dagger})\to \varprojlim_n\Dcal(\AnSpec A_{\square}\times (\Gbb^n)^{h\dagger}),$$
where we implicitly use that $\{(f_n)_*\}_{[n]\in \Delta}$ preserves cocartesian sections by Lemma \ref{lem:steady semilinear}.
The essential image of the above functor is the full subcategory consisting of cocartesian sections $\{V_n\}_{[n]\in \Delta} \in \varprojlim_n\Dcal(\AnSpec A_{\square}\times (\Gbb^n)^{h\dagger})$ such that $V_0\in \Dcal(\AnSpec A_{\square})$ is $\Acal$-complete.
Therefore, we get the equivalence of symmetric monoidal $\infty$-categories
$\Dcal(\AnSpec\Acal/\Gbb^{h\dagger})\simeq \Rep_{\Acal}^{h\dagger}(G).$
\end{proof}

\begin{corollary}\label{cor:semilinear completion}
	Let $f\colon \AnSpec\Bcal\to \AnSpec\Acal$ be a $\Gbb^{h\dagger}$-equivariant morphism of solid affinoid spaces with an $h\dagger$-analytic action of $G$, and let $\overline{f}\colon \AnSpec\Bcal/\Gbb^{h\dagger}\to \AnSpec \Acal/\Gbb^{h\dagger}$ denote the quotient morphism.
	\begin{enumerate}
		\item The functor 
		\begin{align*}
			\overline{f}_*\colon \Rep_{\Bcal}^{h\dagger}(G)\simeq &\Dcal(\AnSpec\Bcal/\Gbb^{h\dagger})\\
			\to &\Dcal(\AnSpec \Acal/\Gbb^{h\dagger})\simeq \Rep_{\Acal}^{h\dagger}(G)
		\end{align*}
		is equivalent to the restriction of scalars $\Rep_{\Bcal}^{h\dagger}(G) \to \Rep_{\Acal}^{h\dagger}(G)$.
		\item The functor 
		\begin{align*}
			\overline{f}^*\colon \Rep_{\Acal}^{h\dagger}(G)\simeq &\Dcal(\AnSpec \Acal/\Gbb^{h\dagger})\\
			\to &\Dcal(\AnSpec \Bcal/\Gbb^{h\dagger})\simeq \Rep_{\Bcal}^{h\dagger}(G)
		\end{align*}
		is a left adjoint functor of the restriction of scalars $\Rep_{\Bcal}^{h\dagger}(G) \to \Rep_{\Acal}^{h\dagger}(G)$.
		Moreover, for $V\in \Rep_{\Acal}^{h\dagger}(G)$, the underlying $\Bcal$-module of $\overline{f}^*V\in \Rep_{\Bcal}^{h\dagger}(G)$ is equivalent to $\Bcal\otimes_{\Acal}V$.
		\item If $f$ lies in $P$ (i.e., the analytic ring structure of $\Bcal$ is induced from $\Acal$), then $\overline{f}\colon \AnSpec\Bcal/\Gbb^{h\dagger}\to \AnSpec \Acal/\Gbb^{h\dagger}$ is weakly $\Dcal$-proper.
	\end{enumerate}
\end{corollary}
\begin{proof}
	Part (1) follows from the construction of Proposition \ref{prop:semilinear geometric interpretation}.
	We have a commutative diagram
	\begin{equation}\label{cartesian diagram}
	\vcenter{
	\xymatrix{
		\AnSpec \Bcal \ar[r]^-{f}\ar[d]_-{p_{\Bcal}}&\AnSpec \Acal\ar[d]^-{p_{\Acal}}\\
		\AnSpec \Bcal/\Gbb^{h\dagger}\ar[r]^-{\overline{f}}&\AnSpec \Acal/\Gbb^{h\dagger},
	}}
	\end{equation}
	so there is an equivalence $p_{\Bcal}^*\circ \overline{f}^*\simeq f^*\circ g_{\Acal}^*$ of functors from $\Dcal(\AnSpec\Acal/\Gbb^{h\dagger})$ to $\Dcal(\AnSpec \Bcal)$.
	Part (2) follows from this.
	Part (3) follows from Corollary \ref{cor:weakly proper} and the cartesian diagram \eqref{cartesian diagram}.
\end{proof}

\begin{remark}
\begin{enumerate}
	\item If the analytic ring structure of $\Bcal$ is induced from $\Acal$, then for $V\in \Rep_{\Acal}^{h\dagger}(G)$, the action of $G$ on $\Bcal\otimes_{\Acal}V=B\otimes_{\Acal}V$ is the diagonal action, where $B$ is the underlying animated $\Qbb_{p,\square}$-algebra of $\Bcal$.
	\item Let $\AnSpec\Acal$ be a solid affinoid space with an $h\dagger$-analytic action of $G$, and let $A$ denote the underlying animated $\Qbb_{p,\square}$-algebra of $\Acal$.
	Then, by applying the above corollary to the natural $G$-equivariant morphism $\AnSpec\Acal \to \AnSpec A_{\square}$, we find that for $V\in \Rep_{A_{\square}}^{h\dagger}(G)$, the $\Acal$-completion $\Acal\otimes_{A_{\square}} V$ of $V$ admits a natural $h\dagger$-analytic action of $G$.
\end{enumerate}
\end{remark}

\begin{corollary}\label{cor:semilinear analytic vector}
	Let $\AnSpec\Acal$ be a solid affinoid space with an $h\dagger$-analytic action of $G$, and let $h^{\prime}\geq h$ be a positive rational number or $\infty$.
	Let $f\colon \AnSpec\Acal/\Gbb^{h^{\prime}\dagger} \to \AnSpec\Acal/\Gbb^{h\dagger}$ denote the natural morphism.
	\begin{enumerate}
		\item The functor 
		\begin{align*}
			f^{*}\colon \Rep_{\Acal}^{h\dagger}(G)\simeq &\Dcal(\AnSpec\Acal/\Gbb^{h\dagger}) \\
			\to&\Dcal(\AnSpec\Acal/\Gbb^{h^{\prime}\dagger})\simeq \Rep_{\Acal}^{h^{\prime}\dagger}(G)
		\end{align*} 
		is equivalent to the natural inclusion $\Rep_{\Acal}^{h\dagger}(G) \hookrightarrow \Rep_{\Acal}^{h^{\prime}\dagger}(G)$.
		\item The functor 
		\begin{align*}
			f_{*}\colon \Rep_{\Acal}^{h^{\prime}\dagger}(G)\simeq &\Dcal(\AnSpec\Acal/\Gbb^{h^{\prime}\dagger}) \\
			\to&\Dcal(\AnSpec\Acal/\Gbb^{h\dagger})\simeq \Rep_{\Acal}^{h\dagger}(G)
		\end{align*} 
		is a right adjoint functor of the natural inclusion $\Rep_{\Acal}^{h\dagger}(G) \hookrightarrow \Rep_{\Acal}^{h^{\prime}\dagger}(G)$.
		It can be described explicitly as follows:
		For $V\in \Rep_{\Acal}^{h^{\prime}\dagger}(G)$, $V^{h\dagger\mathchar`-\an}$ becomes an $A=A^{h\dagger\mathchar`-\an}$-module in $\Rep_{\square}^{h\dagger}(G)$, where $A$ is the underlying animated $\Qbb_{p,\square}$-algebra of $\Acal$.
		Then $V^{h\dagger\mathchar`-\an}$ is $\Acal$-complete and $f_*V\in \Rep_{\Acal}^{h\dagger}(G)$ is equivalent to this.
		\item The morphism $f\colon \AnSpec\Acal/\Gbb^{h^{\prime}\dagger} \to \AnSpec\Acal/\Gbb^{h\dagger}$ is weakly $\Dcal$-proper.
	\end{enumerate}
\end{corollary}
\begin{proof}
	Part (1) follows from the construction of Proposition \ref{prop:semilinear geometric interpretation}.
	Part (2) is obvious when $\Acal=A_{\square}$. 
	Let us consider general cases.
	We have a commutative diagram
	$$
	\xymatrix{
		\AnSpec\Acal/\Gbb^{h^{\prime}\dagger}\ar[r]^-{f}\ar[d]_-{p_{h^{\prime}}}& \AnSpec\Acal/\Gbb^{h\dagger}\ar[d]^-{p_h}\\
		\AnSpec A_{\square}/\Gbb^{h^{\prime}\dagger}\ar[r]^-{f^{\prime}}& \AnSpec A_{\square}/\Gbb^{h\dagger},
	}
	$$
	so there is an equivalence $(p_h)_*\circ f_*\simeq f^{\prime}_*\circ (p_{h^{\prime}})_*$ of functors from $\Dcal(\AnSpec\Acal/\Gbb^{h^{\prime}\dagger})$ to $\Dcal(\AnSpec A_{\square}/\Gbb^{h\dagger})$.
	Therefore, (2) follows from Corollary \ref{cor:semilinear completion}.
	Part (3) follows from Lemma \ref{lem:proper composition basechange}, Proposition \ref{prop:6-ff classifying stack} (5), and a cartesian diagram
	$$
	\xymatrix{
		\AnSpec\Acal/\Gbb^{h^{\prime}\dagger}\ar[r]^-{f}\ar[d] &\AnSpec\Acal/\Gbb^{h\dagger}\ar[d]\\
		\AnSpec\Qbb_{p,\square}/\Gbb^{h^{\prime}\dagger} \ar[r] &\AnSpec\Qbb_{p,\square}/\Gbb^{h\dagger}.
	}
	$$
\end{proof}

Finally, we note the following lemma.
\begin{lemma}\label{lem:suave smooth classifying stack}
	Let $X\in \Shv_{\Dcal}(\Aff_{\Qbb_{p,\square}})$ (resp. $Y$) be a solid $\Dcal$-stack over $\AnSpec\Qbb_{p,\square}$ with an action of $\Gbb^{h\dagger}$ (resp. the trivial action of $\Gbb^{h\dagger}$), and let $f\colon X\to Y$ be a $\Gbb^{h\dagger}$-equivariant morphism.
	\begin{enumerate}
		\item If $f$ is weakly $\Dcal$-proper, then the natural morphism $f^{\prime}\colon X/\Gbb^{h\dagger} \to Y$ is also weakly $\Dcal$-proper.
		\item If $f$ is $\Dcal$-suave (resp. $\Dcal$-smooth) and $h$ is sufficiently large, then the natural morphism $f^{\prime}\colon X/\Gbb^{h\dagger} \to Y$ is also $\Dcal$-suave (resp. $\Dcal$-smooth).
		If $G$ is abelian, this holds for arbitrary $h$.
	\end{enumerate}
\end{lemma}
\begin{proof}
	We have a factorization
	$$f^{\prime}\colon X/\Gbb^{h\dagger} \overset{\overline{f}}{\to} Y/\Gbb^{h\dagger} \overset{g}{\to}Y.$$
	By Corollary \ref{cor:prim and suave map D-local} (2), Lemma \ref{lem:proper composition basechange}, Proposition \ref{prop:6-ff classifying stack} (5), (6), and a cartesian diagram
	$$
	\xymatrix{
		Y/\Gbb^{h\dagger}  \ar[r]^-{g}\ar[d] & Y \ar[d]\\
		\AnSpec\Qbb_{p,\square}/\Gbb^{h\dagger} \ar[r] &\AnSpec\Qbb_{p,\square},
	}
	$$
	the morphism $g\colon Y/\Gbb^{h\dagger} \to Y$ is weakly $\Dcal$-proper and $\Dcal$-smooth (when $h$ is sufficiently large).
	Moreover, by Corollary \ref{cor:prim and suave map D-local} (3) and a cartesian diagram
	$$
	\xymatrix{
		X  \ar[r]^-{f}\ar[d] & Y \ar[d]\\
		X/\Gbb^{h\dagger} \ar[r]^-{\overline{f}} &Y/\Gbb^{h\dagger},
	}
	$$
	the morphism $\overline{f}\colon X/\Gbb^{h\dagger}\to Y/\Gbb^{h\dagger}$ is weakly $\Dcal$-proper (resp. $\Dcal$-suave, $\Dcal$-smooth).
	Therefore, we get the claim from Corollary \ref{cor:prim and suave map D-local} (1).
\end{proof}

%%%%%%%%%%%%%%%%%%%%%%%%%%%%%%%%%%%%%%%%%%%%%%%%%%%%%%%%%%%%%%%%%%%%%%%%%%%%%%%%%%%%%%%%%%%%%%%%%

\section{Cohomology of $(\varphi,\Gamma)$-modules}
\subsection{Locally analytic actions and suaveness}
Let $G$ be a compact open subgroup of $\Zbb_p^{\times}$.
We take the smallest integer $n\geq 1$ ($n\geq 2$ if $p=2$) such that $1+p^n\Zbb_p \subset G$.
By using the exponential $\exp\colon p^n\Zbb_p \overset{\sim}{\to} 1+p^n\Zbb_p$, we define a rigid analytic group $\Gbb_n=\Spa\Qbb_p\langle T/p^n\rangle$ with $\Gbb_n(\Qbb_p)=1+p^n\Zbb_p$.
Let $\Gbb^n$ denote the rigid analytic group $\bigsqcup_{g\in G/\Gbb_n(\Qbb_p)}g \Gbb_n$.
In the same way as in the subsection \ref{subsection:la repns}, we also define $\Gbb_h$, $\Gbb^h$, $\mathring{\Gbb}_h$, and $\Gbb^{h+}$ for every rational number $h\geq n$, and define $\Gbb^{h\dagger}$ for every rational number $h>n$.
For a rational number $h\geq n$, we write $G_h=\Gbb_h(\Qbb_p)$. 
Let $i$ be the smallest integer such that $i\geq h$.
Then, we have $G_h=G_i=1+p^i\Zbb_p$.
We write $\gamma_i=\exp(p^i)\in G$, which is a topological generator of $G_i$.

\begin{remark}
	Let $h\geq n$ be a rational number, and $i$ be the smallest integer such that $i\geq h$.
	By construction, we have a natural isomorphism 
	$$C^h(G,\Qbb_p)\cong\prod_{G/G_h}\Qbb_p\langle T/p^h\rangle,$$ 
	and $\gamma_i \in G$ acts on it as $T\mapsto T+p^i$.
\end{remark}

We work with the following Tate-Sen setting.

\begin{setting}\label{setting suave}
Let $A$ be a Banach $\Qbb_p$-algebra with a unitary $G$-action, and let $A_0\subset A_1\subset \cdots \subset A$ be closed subalgebras stable under the $G$-action.
We normalize the norm as $\lvert p\rvert=p^{-1}$.
We assume that:
\begin{enumerate}[label=(C\arabic*)]
	\item There exist $c_2>0$ and continuous $A_i$-linear and $G$-equivariant morphisms 
	$$R_i\colon A\to A_i$$
	for any $i\geq 0$ satisfying the following:
	\begin{itemize}
		\item For any $i\geq 0$, $R_i|_{A_i}=\id_{A_i}$.
		\item For any $i\geq 0$ and $x\in A$, $\lvert R_i(x)\rvert \leq p^{c_2}\lvert x\rvert$.
		\item For any $x\in A$, $\lim_{i\to\infty}R_i(x)=x$.
	\end{itemize}
	\item\label{TS3}
	We set $X_i=\Ker(R_i \colon A\to A_i)$ for $i\geq 0$.
	Then there exist $c_3>0$ and an integer $N\geq 0$ such that for any $l\geq N$ and $\gamma \in G_n\setminus G_{n+l+1}$, the morphism 
	$$\gamma-\id \colon X_l \to X_l$$ 
	is invertible and 
	$\lvert(\gamma-\id)^{-1}(x)\rvert \leq p^{c_3}\lvert x\rvert$
	for any $x\in X_l$.
	\item Let $N$ be as in \ref{TS3}.
	Then, for an integer $i$ sufficiently large and an integer $l\geq N$, there exists $t=t(i,l)>0$ such that for any $\gamma \in G_i$ and any $x\in A_l$,
	$$\lvert (\gamma-\id)x \rvert \leq p^{-t}\lvert x\rvert.$$ 
	\item For any integer $i\geq n$, there exists an integer $j$ such that $H^0(A^{i\mathchar`-\an})\subset A_j$ (for the definition of $H^0(A^{i\mathchar`-\an})$, see Definition \ref{defn:non-derived analytic vector}).
	\item For any $j\geq i$, $A_j$ is a finite $A_i$-algebra.
	\item For any integer $i\geq 0$, the morphism $\AnSpec (A_i,A_i^{\circ})_{\square}\to \AnSpec \Qbb_{p,\square}$ is $\Dcal$-suave.
\end{enumerate}
We write $A_{\infty}=\varinjlim_i A_i\subset A$, where we take the colimit as $\Qbb_{p,\square}$-algebras.
We note that $A_{\infty}\subset A$ is stable under the action of $G$.

\end{setting}
\begin{remark}
	The conditions (C1), (C2), (C3) correspond to the Tate-Sen conditions (TS2), (TS3), and (TS4), respectively (\cite[Definition 3.1.3]{BC08}, \cite[5A]{Por24}).
	Since we are working with a compact open subgroup of $\Zbb_p^{\times}$, there is no condition corresponding to (TS1).
\end{remark}

\begin{remark}\label{rem:defn of h}
It follows from the condition (C3) that there exists a sequence $n\leq h(0)\leq h(1)\leq \cdots$ of integers such that $\lim_{i\to \infty}h(i)=\infty$ and the action of $G$ on $A_i$ induces an $h(i)\dagger$-analytic action on $\AnSpec (A_i,A_i^{\circ})_{\square}$ by \cite[Example 2.1.9]{Pan22} and Lemma \ref{lem:banach algebra with la action}.
\end{remark}

We write $(A_{\infty},A_{\infty}^{\circ})_{\square}=\varinjlim_i (A_i,A_i^{\circ})_{\square}$.
From the above remark, we find that the action of $G$ on $A_{\infty}$ induces a locally analytic action on $\AnSpec(A_{\infty},A_{\infty}^{\circ})_{\square}$.

The goal of this subsection is to prove the following theorem.
\begin{theorem}\label{thm:suave Tate-Sen}
The  natural morphism 
$$p_{\infty}\colon \AnSpec(A_{\infty},A_{\infty}^{\circ})_{\square}/\Gbb^{\la}\to \AnSpec\Qbb_{p,\square}$$ 
is $\Dcal$-suave.
\end{theorem}

We set $Y_i=\Ker(R_i\colon A_{i+1}\to A_i)$.
We note that $Y_i\subset R_{i+1}$ is stable under the action of $G$.
The following lemma is essential.
\begin{lemma}\label{lem:Tate-Sen H1 vanishing}
	For any rational number $h\geq n$, we have $H^1(G, C^h(G,Y_l))=0$ for any integer $l\geq h+c_3+2$.
\end{lemma} 
\begin{proof}
	There is an equivalence 
	$$\Gamma(G,-)\simeq \Gamma(G/G_h, \Gamma(G_h,-))$$ 
	and the functor $\Gamma(G/G_h,-)$ is $t$-exact, so it is enough to show that $$H^1(G_h, C^h(G,Y_l))=0.$$
	As a $G_h$-representation, $C^h(G,Y_l)$ is a direct sum of copies of $C(\Gbb_h,\Qbb_p)\hotimes_{\Qbb_p}Y_l$, where we note that $G$ is abelian and that for Banach $\Qbb_p$-modules, the solid tensor product $-\otimes_{\Qbb_{p,\square}}-$ coincides with the usual completed tensor product (\cite[Lemma 3.13]{RJRC22}).
	Therefore, it suffices to show $H^1(G_h, C(\Gbb_h,\Qbb_p)\hotimes_{\Qbb_{p}}Y_l)=0$ for $l\geq h+c_3+2$.
	Let $i$ be the smallest integer such that $i\geq h$.
	Then we have $\Zbb_p\cong G_i=G_h;\; a\mapsto \gamma_i^a$, so we have an isomorphism
	\begin{align*}
	&H^1(G_h, C(\Gbb_h,\Qbb_p)\hotimes_{\Qbb_{p}}Y_l)\\
	\cong &\Coker(\gamma_i-\id\colon C(\Gbb_h,\Qbb_p)\hotimes_{\Qbb_{p}}Y_l \to C(\Gbb_h,\Qbb_p)\hotimes_{\Qbb_{p}}Y_l).
	\end{align*}
	We note $C(\Gbb_h,\Qbb_p)\cong \Qbb_p\langle T/p^h\rangle$ and $\gamma_i$ acts on it as $T\mapsto T+p^i$.
	We endow $\Qbb_p\langle T/p^h\rangle$ with the usual Gauss norm and $\Qbb_p\langle T/p^h\rangle\hotimes_{\Qbb_{p}}Y_l$ with the tensor product norm.
	Then for any integer $r\geq 0$ and $x\in \Qbb_p\langle T/p^h\rangle$, we have 
	$$\lvert(\gamma_{i+r}-\id)x\rvert \leq p^{-r}\lvert x\rvert.$$
	We take the smallest integer $r$ such that $r>c_3$.
	Let us prove that 
	$$\gamma_i-\id\colon \Qbb_p\langle T/p^h\rangle\hotimes_{\Qbb_{p}}Y_l \to \Qbb_p\langle T/p^h\rangle\hotimes_{\Qbb_{p}}Y_l$$
	is surjective for $l\geq i+r$, where we note that $h+c_3+2>i+r$.
	Since we have $\gamma_{i+r}-\id=(\gamma_i-\id)(\gamma_i^{p^{r}-1}+\gamma_i^{p^{r}-2}+\cdots+\id)$, it suffices to show that $\gamma_{i+r}-\id$ is surjective.
	We note that by the condition (C2), $\gamma_{i+r}-\id \colon Y_l\to Y_l$ is invertible and for any $y\in Y_l$, $\lvert(\gamma_{i+r}-\id)^{-1}(y)\rvert \leq p^{c_3}\lvert y\rvert$.
	It is enough to show that for any $y\in Y_l$ and $f\in \Qbb_p\langle T/p^h\rangle$, there exists $x\in \Qbb_p\langle T/p^h\rangle\hotimes_{\Qbb_{p}}Y_l$ such that $(\gamma_{i+r}-\id)(x)=f\otimes y$ and $\lvert x\rvert \leq p^{c_3}\lvert f\rvert\cdot\lvert y\rvert $.
	First, we take $y_1 \in Y_l$ such that $(\gamma_{i+r}-\id)(y_1)=y$.
	Then we have 
	$$f\otimes y=(\gamma_{i+r}-\id)(f\otimes y_1)-(\gamma_{i+r}(f)-f)\otimes \gamma_{i+r}(y_1),$$
	and 
	$$\lvert\gamma_{i+r}(f)-f\rvert \cdot \lvert\gamma_{i+r}(y_1) \rvert \leq p^{-r}\lvert f\rvert \cdot p^{c_3}\lvert y\rvert = p^{c_3-r}\lvert f\rvert\cdot\lvert y\rvert.$$
	Next, we take $y_2 \in Y_l$ such that $(\gamma_{i+r}-\id)(y_2)=\gamma_{i+r}(y_1)$.
	By repeating this construction, we define $y_j\in Y_l$ for $j=1,2,\ldots$, and we set
	$$x=f\otimes y_1 +(\gamma_{i+r}(f)-f)\otimes y_2+\cdots.$$
	Then we have $(\gamma_{i+r}-\id)(x)=f\otimes y$ and  $\lvert x\rvert \leq p^{c_3}\lvert f\rvert\cdot\lvert y\rvert$.
\end{proof}

\begin{corollary}\label{cor:key for suave}
	For any rational number $h>n$, there exists an integer $l\geq 0$ such that the natural morphism $A_l^{h\dagger\mathchar`-\an}\to A_{\infty}^{h\dagger\mathchar`-\an}$ becomes an equivalence.
\end{corollary}
\begin{proof}
	Since $C^{h\dagger}(G,\Qbb_p)=\varinjlim_{h^{\prime}<h} C^{h^{\prime}}(G,\Qbb_p)$ is a filtered colimit of flat $\Qbb_{p,\square}$-modules (\cite[Lemma 3.21]{RJRC22}), the functor 
	$$\Rep_{\square}(G)\to \Rep_{\square}(G);\; V\mapsto (C^{h\dagger}(G,\Qbb_p)\otimes_{\Qbb_{p,\square}} V)_{\star_{1,3}}=C^{h\dagger}(G,V)_{\star_{1,3}}$$ 
	is $t$-exact.
	By the proof of Lemma \ref{lem:Tate-Sen H1 vanishing}, the cohomological dimension of $\Gamma(G,-)$ is $1$, so the cohomological dimension of $(-)^{h\dagger\mathchar`-\an}=\Gamma(G,C^{h\dagger}(G,-)_{\star_{1,3}})$ is $1$.
	Therefore, it suffices to find an integer $l\geq 0$ such that the natural morphism $H^i(A_l^{h\dagger\mathchar`-\an})\to H^i(A_{\infty}^{h\dagger\mathchar`-\an})$ is an isomorphism for $i=0,1$.
	Let $l$ be an integer such that $l\geq h+c_3+2$.
	Then for any $l^{\prime}\geq l$ and $n\leq h^{\prime}<h$, we have $H^1(G, C^{h^{\prime}}(G,Y_{l^{\prime}}))=0$.
	By taking a colimit with respect to $h^{\prime}$, we get $H^1(Y_{l^{\prime}}^{h\dagger\mathchar`-\an})=0$ for any $l^{\prime}\geq l$.
	Since we have a decomposition $A_{\infty}=A_l\oplus Y_l \oplus Y_{l+1}\oplus \cdots$, we find that the natural morphism  $H^1(A_{l}^{h\dagger\mathchar`-\an})\to H^1(A_{\infty}^{h\dagger\mathchar`-\an})$ is an isomorphism.
	
	Next, let us consider the morphism 
	$$H^0(A_{l}^{h\dagger\mathchar`-\an})=\varinjlim_{h^{\prime}<h} H^0(A_l^{h^{\prime}\mathchar`-\an})\hookrightarrow \varinjlim_{h^{\prime}<h} H^0(A_{\infty}^{h^{\prime}\mathchar`-\an})=H^0(A_{\infty}^{h\dagger\mathchar`-\an}).$$
	From the condition (C4), we may assume $H^0(A^{j\mathchar`-\an})\subset A_l$ for some integer $j$ such that $j\geq h$ by replacing $l$ with a larger integer.
	Then, for any $h^{\prime}<h\leq j$, we have $H^0(A_{\infty}^{h^{\prime}\mathchar`-\an})\subset A_l$, and therefore, $H^0(A_l^{h^{\prime}\mathchar`-\an})\hookrightarrow H^0(A_{\infty}^{h^{\prime}\mathchar`-\an})$ becomes an isomorphism.
	Hence, we find that $H^0(A_l^{h\dagger\mathchar`-\an})\to H^0(A_{\infty}^{h\dagger\mathchar`-\an})$ is an isomorphism.
\end{proof}

\begin{proof}[The proof of Theorem \ref{thm:suave Tate-Sen}]
	We take $h(0)\leq h(1)\leq \cdots $ as in Remark \ref{rem:defn of h}.
	We apply Proposition \ref{prop:suave criterion} to 
	$$(f\colon X\to S)=(p_{\infty}\colon\AnSpec(A_{\infty},A_{\infty}^{\circ})_{\square}/\Gbb^{\la}\to \AnSpec\Qbb_{p,\square}),$$
	and 
	\begin{align*}
	&(X\overset{g_i}{\to}X_i\overset{f_i}{\to} S)\\
	=&(\AnSpec(A_{\infty},A_{\infty}^{\circ})_{\square}/\Gbb^{\la}\overset{p_{\infty, i}}{\longrightarrow} \AnSpec(A_{i},A_{i}^{\circ})_{\square}/\Gbb^{h(i)\dagger}\overset{p_i}{\longrightarrow}\AnSpec\Qbb_{p,\square}).
	\end{align*}
	We check the conditions (1), (2), and (3) in Proposition \ref{prop:suave criterion}.
	First, by Lemma \ref{lem:suave smooth classifying stack} and the condition (C6), the morphism $$p_i \colon \AnSpec(A_{i},A_{i}^{\circ})_{\square}/\Gbb^{h(i)\dagger} \to \AnSpec\Qbb_{p,\square}$$ is $\Dcal$-suave.
	For any $j\geq i$, $A_j$ is finite $A_i$-algebra by the condition (C5), so the analytic ring structure on $(A_j,A_j^{\circ})_{\square}$ is induced from $(A_i,A_i^{\circ})_{\square}$.
	Therefore, the analytic ring structure on $(A_{\infty},A_{\infty}^{\circ})_{\square}$ is also induced from $(A_i,A_i^{\circ})_{\square}$.
	In other words, the morphism $\AnSpec (A_{\infty},A_{\infty}^{\circ})_{\square} \to \AnSpec (A_i,A_i^{\circ})_{\square}$ lies in $P$. 
	Therefore, by Corollary \ref{cor:semilinear completion} and Corollary \ref{cor:semilinear analytic vector}, the morphism
	$$p_{\infty,i}\colon \AnSpec(A_{\infty},A_{\infty}^{\circ})_{\square}/\Gbb^{\la} \to \AnSpec(A_{i},A_{i}^{\circ})_{\square}/\Gbb^{h(i)\dagger}$$
	is weakly $\Dcal$-proper, and in particular, $\Dcal$-prim.

	Next, we check the condition (2) in Proposition \ref{prop:suave criterion}.
	We want to show that the family of functors
	\begin{align*}
	\{(p_{\infty,i}\times\id)_!\colon &\Dcal(\AnSpec(A_{\infty},A_{\infty}^{\circ})_{\square}/\Gbb^{\la}\times \AnSpec(A_{\infty},A_{\infty}^{\circ})_{\square}/\Gbb^{\la})\\
	\to &\Dcal(\AnSpec(A_i,A_i^{\circ})_{\square}/\Gbb^{h(i)\dagger}\times \AnSpec(A_{\infty},A_{\infty}^{\circ})_{\square}/\Gbb^{\la})\}_i
	\end{align*}
	is conservative.
	% There is a cartesian diagram
	% $$
	% \xymatrix{
	% 	\AnSpec(A_{\infty},A_{\infty}^{\circ})_{\square}/\Gbb^{\la}\times \AnSpec(A_{\infty},A_{\infty}^{\circ})_{\square}\ar[r]\ar[d] & \AnSpec(A_i,A_i^{\circ})_{\square}/\Gbb^{h(i)\dagger}\times \AnSpec(A_{\infty},A_{\infty}^{\circ})_{\square}\ar[d]\\
	% 	\AnSpec(A_{\infty},A_{\infty}^{\circ})_{\square}/\Gbb^{\la}\times \AnSpec(A_{\infty},A_{\infty}^{\circ})_{\square}/\Gbb^{\la}\ar[r] & \AnSpec(A_i,A_i^{\circ})_{\square}/\Gbb^{h(i)\dagger}\times \AnSpec(A_{\infty},A_{\infty}^{\circ})_{\square}/\Gbb^{\la},
	% }
	% $$
	We set $p\colon \AnSpec(A_{\infty},A_{\infty}^{\circ})_{\square} \to \AnSpec(A_{\infty},A_{\infty}^{\circ})_{\square}/\Gbb^{\la}$.
	Then, by the proper base change, we have an equivalence 
	$$(\id\times p)^*\circ(p_{\infty,i}\times\id)_!\simeq (p_{\infty,i}\times\id)_!\circ (\id\times p)^*$$ 
	of functors from 
	$$\Dcal(\AnSpec(A_{\infty},A_{\infty}^{\circ})_{\square}/\Gbb^{\la}\times \AnSpec(A_{\infty},A_{\infty}^{\circ})_{\square}/\Gbb^{\la})$$
	to 
	$$\Dcal(\AnSpec(A_i,A_i^{\circ})_{\square}/\Gbb^{h(i)\dagger}\times \AnSpec(A_{\infty},A_{\infty}^{\circ})_{\square}).$$
	Since the morphism 
	\begin{align*}
	\id\times p \colon &\AnSpec(A_{\infty},A_{\infty}^{\circ})_{\square}/\Gbb^{\la}\times \AnSpec(A_{\infty},A_{\infty}^{\circ})_{\square}/\Gbb^{\la} \\
	\to &\AnSpec(A_{\infty},A_{\infty}^{\circ})_{\square}/\Gbb^{\la}\times \AnSpec(A_{\infty},A_{\infty}^{\circ})_{\square}
	\end{align*}
	is a universal $\Dcal^*$-cover, $(\id\times p)^*$ is conservative.
	Therefore, it is enough to show that the family of functors
	\begin{align*}
		\{(p_{\infty,i}\times\id)_!\colon &\Dcal(\AnSpec(A_{\infty},A_{\infty}^{\circ})_{\square}/\Gbb^{\la}\times \AnSpec(A_{\infty},A_{\infty}^{\circ})_{\square})\\
		\to &\Dcal(\AnSpec(A_i,A_i^{\circ})_{\square}/\Gbb^{h(i)\dagger}\times \AnSpec(A_{\infty},A_{\infty}^{\circ})_{\square})\}_i
	\end{align*}
	is conservative.
	Since $p_{\infty,i}$ is weakly $\Dcal$-proper, we have an equivalence $(p_{\infty,i}\times\id)_!\simeq (p_{\infty,i}\times\id)_*$.
	We have a commutative diagram
	$$
	\xymatrix{
		\AnSpec(A_{\infty},A_{\infty}^{\circ})_{\square}/\Gbb^{\la}\times \AnSpec(A_{\infty},A_{\infty}^{\circ})_{\square}\ar[r]^-{q_{\infty}}\ar[d]_-{p_{\infty,i}\times \id} & \AnSpec \Qbb_{p,\square}/\Gbb^{\la}\ar[d]^-{p^{\prime}_{\infty,h(i)}}\\
		\AnSpec(A_i,A_i^{\circ})_{\square}/\Gbb^{h(i)\dagger}\times \AnSpec(A_{\infty},A_{\infty}^{\circ})_{\square}\ar[r]^-{q_i} & \AnSpec \Qbb_{p,\square}/\Gbb^{h(i)\dagger},
	}
	$$
	and there is an equivalence $q_{i*}\circ (p_{\infty,i}\times \id)_*\simeq p^{\prime}_{\infty,h(i)*}\circ q_{\infty*}$ of functors from 
	$$\Dcal(\AnSpec(A_{\infty},A_{\infty}^{\circ})_{\square}/\Gbb^{\la}\times \AnSpec(A_{\infty},A_{\infty}^{\circ})_{\square})$$ 
	to $\Dcal(\AnSpec \Qbb_{p,\square}/\Gbb^{h(i)\dagger})$.
	Since the functor $q_{\infty*}$ is conservative, it is enough to show that the family of functors
	$$\{(p^{\prime}_{\infty,h(i)})_*\colon\Dcal(\AnSpec \Qbb_{p,\square}/\Gbb^{\la})\to \Dcal(\AnSpec \Qbb_{p,\square}/\Gbb^{h(i)\dagger})\}_i$$
	is conservative.
	Since we have $\lim_{i\to \infty}h(i)=\infty$, the natural morphism 
	$$\varinjlim_i ((p^{\prime}_{\infty,h(i)})^*\circ (p^{\prime}_{\infty,h(i)})_*) \to \id$$ 
	of endofunctors of $\Dcal(\AnSpec \Qbb_{p,\square}/\Gbb^{\la})$ is an equivalence, which proves the conservativity.

	Finally, we check the condition (3) in Proposition \ref{prop:suave criterion}.
	We want to show that $(p_{\infty,i})_*A_{\infty}\in \Dcal(\AnSpec(A_{i},A_{i}^{\circ})_{\square}/\Gbb^{h(i)\dagger})$ is $p_i$-suave for every $i$, where we note that the morphism $p_{\infty,i}$ is weakly $\Dcal$-proper.
	By Corollary \ref{cor:key for suave}, for an integer $l$ sufficiently large, the natural morphism $(p_{l,i})_*A_l\to (p_{\infty,i})_*A_{\infty}$ becomes an equivalence, where $p_{l,i}$ is the natural morphism $\AnSpec(A_{l},A_{l}^{\circ})_{\square}/\Gbb^{h(l)\dagger}\to \AnSpec(A_{i},A_{i}^{\circ})_{\square}/\Gbb^{h(i)\dagger}$.
	Therefore, it is enough to show that $(p_{l,i})_*A_l$ is $p_i$-suave.
	Since $p_{l,i}$ is weakly $\Dcal$-proper and $p_l$ is $\Dcal$-suave, it follows from Lemma \ref{lem:suave and prim composition} (2).
\end{proof}

\subsection{Locally analytic actions and smoothness}
We work with the same setting as in the previous subsection, and we strengthen the conditions (C5) and (C6) in Setting \ref{setting suave} as follows:
\begin{enumerate}[label=(C\arabic*${}^{\prime}$), start=5]
	\item For any $j\geq i$, $A_j$ is a finite \'{e}tale $A_i$-algebra of constant rank.
	\item For any integer $i\geq 0$, the morphism $\Spa (A_i,A_i^{\circ})\to \Spa \Qbb_{p}$ is smooth of pure dimension $d$ (in the sense of Huber).
\end{enumerate}

\begin{remark}
	By the condition (C${6^{\prime}}$), $A_i$ is an affinoid $\Qbb_p$-algebra.
	Therefore, by Corollary \ref{cor:etale D-etale}, $\AnSpec(A_j,A_j^{\circ})_{\square}\to \AnSpec(A_i,A_i^{\circ})_{\square}$ is $\Dcal$-\'{e}tale for any $j\geq i$, and by Theorem \ref{thm:smooth D-smooth}, $\AnSpec(A_i,A_i^{\circ})_{\square}\to \AnSpec \Qbb_{p,\square}$ is $\Dcal$-smooth for any $i\geq0$.
\end{remark}

\begin{theorem}\label{thm:smooth Tate-Sen}
	The  natural morphism 
	$$p_{\infty}\colon \AnSpec(A_{\infty},A_{\infty}^{\circ})_{\square}/\Gbb^{\la}\to \AnSpec\Qbb_{p,\square}$$ 
	is $\Dcal$-smooth.
\end{theorem}

By Theorem \ref{thm:suave Tate-Sen}, we have already shown that the morphism $p_{\infty}$ is $\Dcal$-suave.
Let us compute the dualizing complex $\omega_{p_{\infty}}$ of $p_{\infty}$.

Let $\omega_i=\Omega_{A_i/\Qbb_p}^d\in \Dcal((A_i,A_i^{\circ})_{\square})$ denote the dualizing module of $\Spa (A_i,A_i^{\circ})\to \Spa \Qbb_{p}$.
For $g\in G$, the automorphism $g\colon \Spa (A_i,A_i^{\circ}) \to \Spa (A_i,A_i^{\circ})$ induces an automorphism of $\omega_i$, and it defines a continuous $A_i$-semilinear action of $G$ on $\omega_i$.
\begin{lemma}\label{lem:dualizing complex smooth quotient}
	The morphism $p_i\colon \AnSpec(A_{i},A_{i}^{\circ})_{\square}/\Gbb^{h(i)\dagger}\to\AnSpec\Qbb_{p,\square}$ is $\Dcal$-smooth, and its dualizing complex is $\omega_i[1+d]$.
\end{lemma}
\begin{proof}
	By Lemma \ref{lem:suave smooth classifying stack}, the morphism $p_i$ is $\Dcal$-smooth.
	We compute its dualizing complex.
	We have a factorization of $p_i$ as follows:
	$$p_i\colon \AnSpec(A_{i},A_{i}^{\circ})_{\square}/\Gbb^{h(i)\dagger} \overset{p_i^{\prime}}{\longrightarrow} \AnSpec\Qbb_{p,\square}/\Gbb^{h(i)\dagger} \overset{f_{h(i)}}{\longrightarrow} \AnSpec\Qbb_{p,\square}.$$
	First, by Proposition \ref{prop:6-ff classifying stack}, the dualizing complex of $f_{h(i)}$ is $\Qbb_p[1]$ with the trivial action of $G$, where we note that $G$ is an abelian $p$-adic Lie group of dimension $1$.
	Next, we compute the dualizing complex of $p^{\prime}_i$.
	We have a cartesian diagram 
	$$
	\xymatrix{
		\AnSpec(A_{i},A_{i}^{\circ})_{\square} \ar[r]\ar[d] & \AnSpec\Qbb_{p,\square}\ar[d]\\
		\AnSpec(A_{i},A_{i}^{\circ})_{\square}/\Gbb^{h(i)\dagger} \ar[r]^-{p^{\prime}_i}& \AnSpec\Qbb_{p,\square}/\Gbb^{h(i)\dagger},
	}
	$$
	and therefore,  by Lemma \ref{lem:suave characterization} and Theorem \ref{thm:smooth D-smooth}, the underlying $(A_{i},A_{i}^{\circ})_{\square}$-module of the dualizing complex $\omega_{p^{\prime}_i}=p^{\prime !}_i\Qbb_p$ is equivalent to $\omega_i[d]$.
	We determine the action of $G$ on $\omega_i[d]$.
	The action of $g\in G$ on $\omega_i[d]$ is induced by the morphism $g\colon \AnSpec (A_i,A_i^{\circ})_{\square} \to \AnSpec (A_i,A_i^{\circ})_{\square}$.
	Since $\omega_i[d]$ is a shift of a Banach $\Qbb_p$-module, the action of $G$ on $\omega_i[d]$ is uniquely determined from the action of each $g\in G$.
	Therefore, we get the claim.
\end{proof}

Let $p_{j,i}\colon \AnSpec(A_j,A_j^{\circ})_{\square}/\Gbb^{h(j)\dagger} \to \AnSpec(A_{i},A_{i}^{\circ})_{\square}/\Gbb^{h(i)\dagger}$ denote the natural morphism.
Then by Lemma \ref{lem:dualizing complex smooth quotient} and the condition (C$5^{\prime}$), we get an equivalence $\omega_j\simeq p_{j,i}^!\omega_i\simeq A_j\otimes_{(A_{i},A_{i}^{\circ})_{\square}}\omega_i$.
By adjunction, we get a trace morphism 
\begin{align*}
\Tr_{j,i}\colon \omega_j^{h(i)\dagger\mathchar`-\an} \simeq (p_{j,i})_*p_{j,i}^!\omega_i \to\omega_i
\end{align*}
in $\Rep_{(A_{i},A_{i}^{\circ})_{\square}}^{h(i)\dagger}(G)$.
Since $A_j$ is a finite locally free $A_i$-module, we have a usual trace morphism $\Tr_{A_j/A_i}\colon A_j\to A_i$, and therefore, we get a morphism 
$$\omega_j\simeq A_j\otimes_{(A_{i},A_{i}^{\circ})_{\square}}\omega_i \to A_i\otimes_{(A_{i},A_{i}^{\circ})_{\square}}\omega_i=\omega_i.$$
By construction, we find that the trace morphism $\Tr_{j,i}$ coincides with the composition of the above morphism and the natural morphism $\omega_j^{h(i)\dagger\mathchar`-\an}\to \omega_j$.

Let $p_{\infty,i}\colon \AnSpec(A_{\infty},A_{\infty}^{\circ})_{\square}/\Gbb^{\la} \to \AnSpec(A_{i},A_{i}^{\circ})_{\square}/\Gbb^{h(i)\dagger}$ denote the natural morphism.
We set 
\begin{align*}
\omega_{\infty}&=p_{\infty,i}^*\omega_i=(A_{\infty},A_{\infty}^{\circ})_{\square}\otimes_{(A_{i},A_{i}^{\circ})_{\square}}\omega_i = A_{\infty}\otimes_{(A_{i},A_{i}^{\circ})_{\square}}\omega_i \\
&\in \Dcal(\AnSpec(A_{\infty},A_{\infty}^{\circ})_{\square}/\Gbb^{\la})\simeq \Rep_{(A_{\infty},A_{\infty}^{\circ})_{\square}}^{\la}(G),
\end{align*}
where we note that the analytic ring structure of $(A_{\infty},A_{\infty}^{\circ})_{\square}$ is induced from $(A_{i},A_{i}^{\circ})_{\square}$.
Since we have an isomorphism 
$$(A_j,A_j^{\circ})_{\square}\otimes_{(A_{i},A_{i}^{\circ})_{\square}}\omega_i \cong \omega_j$$
in $\Rep_{(A_j,A_j^{\circ})_{\square}}^{h(j)\dagger}(G)$, the definition of $\omega_{\infty}$ does not depend on the choice of $i$.
Since $\omega_i$ is invertible, $\omega_{\infty}$ is also invertible in $\Rep_{(A_{\infty},A_{\infty}^{\circ})_{\square}}^{\la}(G)$.

For the proof of Theorem \ref{thm:smooth Tate-Sen}, it is enough to show the following theorem:
\begin{theorem}\label{thm:smooth Tate-Sen dualizing complex}
	The dualizing complex $\omega_{p_{\infty}}=p_{\infty}^!\Qbb_p$ of the morphism 
	$$p_{\infty}\colon \AnSpec(A_{\infty},A_{\infty}^{\circ})_{\square}/\Gbb^{\la}\to \AnSpec\Qbb_{p,\square}$$ 
	is $\omega_{\infty}[1+d]$.
	In particular, it is invertible. 
\end{theorem}
\begin{proof}
	Let $r(i)$ denote the rank of the finite locally free $A_0$-module $A_i$. 
	First, we construct a morphism 
	$$\alpha\colon \omega_{\infty}[1+d] \to p_{\infty}^!\Qbb_p$$
	in $\Rep_{(A_{\infty},A_{\infty}^{\circ})_{\square}}^{\la}(G)$.
	Since we have equivalences $p_{\infty}^!\Qbb_p\simeq p_{\infty,i}^!p_i^!\Qbb_p\simeq p_{\infty,i}^! \omega_i[1+d]$, it is enough to construct a morphism $\alpha_i \colon (p_{\infty,i})_!\omega_{\infty} \to \omega_i$ in $\Rep_{(A_i,A_i^{\circ})_{\square}}^{h(i)\dagger}(G)$.
	For $j\geq i$ sufficiently large, we have equivalences
	\begin{align*}
	(p_{\infty,i})_!\omega_{\infty} &\simeq (p_{\infty,i})_!(A_{\infty}\otimes_{(A_{i},A_{i}^{\circ})_{\square}}\omega_i)\\
	&\simeq (A_{\infty}\otimes_{(A_{i},A_{i}^{\circ})_{\square}}\omega_i)^{h(i)\dagger\mathchar`-\an}\\
	&\simeq A_{\infty}^{h(i)\dagger\mathchar`-\an}\otimes_{(A_{i},A_{i}^{\circ})_{\square}}\omega_i\\
	&\simeq A_j^{h(i)\dagger\mathchar`-\an}\otimes_{(A_{i},A_{i}^{\circ})_{\square}}\omega_i \\
	&\simeq (A_j\otimes_{(A_{i},A_{i}^{\circ})_{\square}}\omega_i)^{h(i)\dagger\mathchar`-\an} \\
	&\simeq \omega_j^{h(i)\dagger\mathchar`-\an} 
    \end{align*}
	where the second equivalence follow from Corollary \ref{cor:semilinear completion} and Corollary \ref{cor:semilinear analytic vector}, the third and fifth equivalence follow from the projection formula, and the fourth equivalence follows from Corollary \ref{cor:key for suave}.
	Therefore, we get a morphism 
	$$\beta_{j,i}\coloneqq r(j)^{-1}\Tr_{j,i}\colon (p_{\infty,i})_!\omega_{\infty}\simeq \omega_j^{h(i)\dagger\mathchar`-\an} \to \omega_i,$$
	and it induces $\alpha_{j,i}\colon \omega_{\infty}[1+d] \to p_{\infty}^!\Qbb_p$.
	We show that it does not depend on the choices of $i,j$.
	First, we show that it does not depend on the choice of $j$.
	We take $j^{\prime}\geq j$.
	It is enough to show that the diagram
	$$
	\xymatrix{
		\omega_j^{h(i)\dagger\mathchar`-\an}\ar[rrrd]^-{r(j)^{-1}\Tr_{j,i}}\ar[d] & & & \\
		\omega_{j^{\prime}}^{h(i)\dagger\mathchar`-\an}\ar[rrr]_-{r(j')^{-1}\Tr_{j',i}}& & & \omega_i
	}
	$$
	is commutative, where the vertical morphism is induced by the inclusion $\omega_j\hookrightarrow \omega_{j^{\prime}}$.
	By the description of $\Tr_{j,i}$, it reduces to showing that a diagram
	$$
	\xymatrix{
		A_j\ar[rrrd]^-{r(j)^{-1}\Tr_{A_j/A_i}}\ar[d]  & & &\\
		A_{j^{\prime}}\ar[rrr]_-{r(j')^{-1}\Tr_{A_{j'}/A_i}} & &&A_i
	}
	$$
	is commutative, which is clear.
	Next, we show that $\alpha_{j,i}$ does not depend on the choice of $i$.
	We take integers $i^{\prime}\geq i\geq 0$ and a sufficiently large integer $j\geq i^{\prime}$.
	It is enough to show that a diagram
	$$
	\xymatrix{
		\omega_j^{h(i)\dagger\mathchar`-\an} \ar[rrrd]_-{r(j)^{-1}\Tr_{j,i}}\ar[rrr]^-{r(j)^{-1}\Tr_{j,i^{\prime}}}& & &\omega_{i^{\prime}}^{h(i)\dagger\mathchar`-\an}\ar[d]^-{\Tr_{i^{\prime},i}}\\
		 & & & \omega_i
	}
	$$
	is commutative, which is clear.
	Therefore, $\alpha_{j,i}\colon \omega_{\infty}[1+d] \to p_{\infty}^!\Qbb_p$ does not depend on the choices of $i,j$.
	We write it $\alpha \colon \omega_{\infty}[1+d] \to p_{\infty}^!\Qbb_p$. 
	We show that it is an equivalence.
	Since the family of functors 
	$$\{(-)^{h(i)\dagger\mathchar`-\an}\colon \Rep_{(A_{\infty},A_{\infty}^{\circ})_{\square}}^{\la}(G)\to \Rep_{(A_i,A_i^{\circ})_{\square}}^{h(i)\dagger}(G)\}_{i}$$
	is conservative, it is enough to show that 
	$$\alpha^{h(i)\dagger\mathchar`-\an} \colon \omega_{\infty}[1+d]^{h(i)\dagger\mathchar`-\an} \to (p_{\infty}^!\Qbb_p)^{h(i)\dagger\mathchar`-\an}$$
	is an equivalence for every $i\geq 0$.
	We set 
	\begin{align*}
	\beta_i=\alpha^{h(i)\dagger\mathchar`-\an}[-1-d] \colon &\omega_{\infty}^{h(i)\dagger\mathchar`-\an} \\
	\to &(p_{\infty}^!\Qbb_p)^{h(i)\dagger\mathchar`-\an}[-1-d]\simeq  (p_{\infty,i}^!\omega_i)^{h(i)\dagger\mathchar`-\an}=(p_{\infty,i})_!p_{\infty,i}^!\omega_i.
	\end{align*}
	Let us prove that $\beta_i$ is an equivalence.
	First, we show that for $j\geq i$ sufficiently large, the morphism $(p_{\infty,i})_!p_{\infty,i}^!\omega_i\to (p_{j,i})_!\omega_j$ induced by the adjunction is an equivalence. 
	By Corollary \ref{cor:counit proper}, we have the following commutative diagram:
	$$
	\xymatrix{
		(p_{\infty,j})_!p_{\infty,i}^!\omega_i\ar[d]\ar[r]^-{\simeq} &(p_{\infty,j})_*\intHom_{A_{\infty}}(A_{\infty},p_{\infty,i}^!\omega_i)\ar[d]^-{\simeq} \\
		\omega_j\ar[d]^-{\simeq} & \intHom_{A_j}((p_{\infty,j})_!A_{\infty},\omega_j)\ar[d]^-{\simeq} \\
		\intHom_{A_j}(A_j,\omega_j) &\intHom_{A_j}((p_{\infty,j})_*A_{\infty},\omega_j),\ar[l]
	}
	$$
	where the lower horizontal morphism is induced from the natural morphism $A_j \to (p_{\infty,j})_*A_{\infty}$, and $\intHom_{A_{\infty}}$ (resp. $\intHom_{A_{j}}$) is the internal hom in $\Rep_{(A_{\infty},A_{\infty}^{\circ})_{\square}}^{\la}(G)$ (resp. $\Rep_{(A_j,A_j^{\circ})_{\square}}^{h(j)\dagger}(G)$).
	Therefore, it is enough to show that the morphism
	$$(p_{j,i})_!\intHom_{A_j}((p_{\infty,j})_*A_{\infty},\omega_j) \to (p_{j,i})_!\intHom_{A_j}(A_j,\omega_j)$$
	is an equivalence.
	In the same way as above, we find that the above morphism is equivalent to the morphism 
	$$\intHom_{A_i}((p_{\infty,i})_*A_{\infty},\omega_i) \to \intHom_{A_j}((p_{j,i})_*A_j,\omega_j),$$
	which is an equivalence by Corollary \ref{cor:key for suave}.

	From the above, to prove that  $\beta_i$ is an equivalence, it is enough to show that 
	$$\omega_j^{{h(i)\dagger\mathchar`-\an}}\simeq \omega_{\infty}^{h(i)\dagger\mathchar`-\an} \to (p_{\infty,i})_!p_{\infty,i}^!\omega_i\simeq (p_{j,i})_!\omega_j=\omega_j^{{h(i)\dagger\mathchar`-\an}}$$
	is an equivalence.
	By the construction, this morphism is multiplication by $r(j)^{-1}$.
	Therefore, it is an equivalence. 
\end{proof}

\subsection{Application to the cohomology of $(\varphi,\Gamma)$-modules}
In this subsection, we apply the previous theorem to the cohomology of $(\varphi,\Gamma)$-modules.
Let us recall the coefficient algebras of $(\varphi,\Gamma)$-modules used in \cite[3.1, 3.2]{Mikami24}.
Let $K$ be a finite extension of $\Qbb_p$, and let $K_{\infty}=K(\zeta_{p^{\infty}})$ denote the $p$-adic cyclotomic extension of $K$.
We write $\Gamma_K=\Gal(K_{\infty}/K)$, and write $\chi\colon\Gamma_K\to \Zbb_p^{\times}$ for the $p$-adic cyclotomic character of $K$.
By using this, we can regard $\Gamma_K$ as a compact open subgroup of $\Zbb_p^{\times}$.
Let $\pi^{\flat}\in \Ocal_{\hat{K}_{\infty}^{\flat}}$ be a pseudo-uniformizer such that $\pi^{\flat\#}/p \in \Ocal_{\hat{K}_{\infty}}^{\times}$.
We set 
\begin{align*}
Y_{K_{\infty}}=\Spa(W(\Ocal_{\hat{K}_{\infty}^{\flat}}))\setminus \{p[\pi^{\flat}]=0\},
\end{align*}
which is an analytic adic space over $\Qbb_p$ come equipped with a Frobenius automorphism $\varphi$ and a continuous action of $\Gamma_K$.
There exist a surjective continuous morphism 
$$\kappa \colon Y_{K_{\infty}} \to (0,\infty)$$ 
defined by 
$$\kappa(x)=\frac{\log \lvert[\pi^{\flat}](\tilde{x})\rvert}{\log \lvert p(\tilde{x})\rvert},$$
where $\tilde{x}$ is the maximal generalization of $x$ (cf. \cite[12.2]{SW20}).
For a closed interval $[r,s]\subset (0,\infty)$ (such that $r,s \in\Qbb$), let $Y_{K_{\infty}}^{[r,s]}$ denote the interior of the preimage of $[r,s]$ under $\kappa$.
For a closed interval $[r,s]\subset (0,\infty)$, we set
\begin{align*}
\tilde{B}_{K_{\infty}}^{[r,s],(+)}=\Ocal_{Y_{K_{\infty}}}^{(+)}(Y_{K_{\infty}}^{[r,s]}).
\end{align*}
It is easy to check that the action of $\Gamma_K$ on $Y_{K_{\infty}}$ induces the action on 
$\tilde{B}_{K_{\infty}}^{[r,s]}$, and the Frobenius automorphism $\varphi \colon Y_{K_{\infty}} \to Y_{K_{\infty}}$ induces 
$\varphi \colon \tilde{B}_{K_{\infty}}^{[r,s]} \to \tilde{B}_{K_{\infty}}^{[r/p,s/p]}$.

Let $K_0$ (resp. $K_0^{\prime}$) be the maximal unramified extension of $\Qbb_p$ contained in $K$ (resp. $K_{\infty}$), and let $e$ be the ramification index of the extension $K_{\infty}/K_{0}(\zeta_{p^{\infty}})$.
Let $\Bbb_{K_0^{\prime}}^{[r,s]}$ denote the rational localization of $\Spa({K_0^{\prime}}\langle T\rangle, \Ocal_{K_0^{\prime}}\langle T\rangle)$ defined by $\lvert p^s\rvert \leq \lvert T \rvert \leq \lvert p^r\rvert$, and let $\Bbb_{K_0^{\prime}}^{(0,s]}$ denote the union of $\Bbb_{K_0^{\prime}}^{[r,s]}$ for all $0<r\leq s$.
Then as the construction in \cite[Chapitre 1]{Ber08B-pair}, for some rational number $s^{\prime}>0$, we can construct the following:
\begin{itemize}
\item An action of $\Gamma_K$ on $\Bbb_{K_0^{\prime}}^{(0,\frac{ps^{\prime}}{(p-1)e}]}$.
\item A Frobenius endomorphism $\varphi\colon \Bbb_{K_0^{\prime}}^{(0,\frac{ps^{\prime}}{(p-1)e}]} \to \Bbb_{K_0^{\prime}}^{(0,\frac{s^{\prime}}{(p-1)e}]}$.
\item A $\Gamma_K$-equivariant morphism $Y_{K_{\infty}}^{(0,s^{\prime}]} \to \Bbb_{K_0^{\prime}}^{(0,\frac{ps^{\prime}}{(p-1)e}]}$ which is compatible with the Frobenius endomorphisms.
Moreover it satisfies that for every closed interval $[r,s]\subset (0,s^{\prime}]$, the preimage of $\Bbb_{K_0^{\prime}}^{[\frac{pr}{(p-1)e},\frac{ps}{(p-1)e}]}$ in $Y_{K_{\infty}}^{(0,s^{\prime}]}$ is $Y_{K_{\infty}}^{[r,s]}$, and the induced morphism $\Ocal(\Bbb_{K_0^{\prime}}^{[\frac{pr}{(p-1)e},\frac{ps}{(p-1)e}]})\to \tilde{B}_{K_{\infty}}^{[r,s]}$ becomes a closed embedding.
\end{itemize}
Let $B_K^{[r,s],(+)}$ denote the image of the closed embedding $\Ocal^{(+)}(\Bbb_{K_0^{\prime}}^{[\frac{pr}{(p-1)e},\frac{ps}{(p-1)e}]})\to \tilde{B}_{K_{\infty}}^{[r,s]}$.
From the above, $B_K^{[r,s]}\subset \tilde{B}_{K_{\infty}}^{[r,s]}$ is stable under the action of $\Gamma_K$, and $\varphi \colon \tilde{B}_{K_{\infty}}^{[r,s]} \to \tilde{B}_{K_{\infty}}^{[r/p,s/p]}$ induces the Frobenius endomorphism $\varphi \colon B_K^{[r,s]} \to B_K^{[r/p,s/p]}$.
Let $B_{K,n}^{[r,s]}$ be the preimage of $B_{K}^{[r/p^n,s/p^n]}$ under the morphism $\varphi^n \colon \tilde{B}_{K_{\infty}}^{[r,s]}\to \tilde{B}_{K_{\infty}}^{[r/p^n,s/p^n]}.$
We define $B_{K,\infty}^{[r,s]}$ as $\varinjlim_n B_{K,n}^{[r,s]}\subset \tilde{B}_{K_{\infty}}^{[r,s]}$, where we regard it as a colimit of $\Qbb_{p,\square}$-algebras.
We set $B_{K,\infty}^{[r,s],+}\coloneqq B_{K,\infty}^{[r,s]} \cap \tilde{B}_{K_{\infty}}^{[r,s],+}=\varinjlim_n B_{K,n}^{[r,s],+}\subset B_{K,\infty}^{[r,s]}$.

\begin{example}\label{ex:unramified}
    If $K$ is an unramified extension of $\Qbb_p$, we can construct the embeddings explicitly.
    We explain the construction.
    First, we fix a system $\varepsilon=(\zeta_{p^n})_{n\geq0}$ such that $\zeta_{p^n}\in\bar{\Qbb_p}$ is a primitive $p^n$-th root of unity and that $\zeta_{p^{n+1}}^p=\zeta_{p^n}$.
    Then we can regard $\varepsilon$ as an element of $\Ocal_{\hat{K}_{\infty}^{\flat}}=\varprojlim_{x\mapsto x^p} \Ocal_{\hat{K}_{\infty}}$.
    We write $\varpi=[\varepsilon]-1 \in W(\Ocal_{\hat{K}_{\infty}^{\flat}})$, where $[\varepsilon]$ is a Teichm\"{u}ller lift of $\varepsilon$.
    Then for $[r,s]$ such that $s< 1$, a morphism $\Ocal_{\Bbb_{K}}(\Bbb_{K}^{[\frac{pr}{p-1},\frac{ps}{p-1}]}) \to \tilde{B}_{K_{\infty}}^{[r,s]} ;\; T\mapsto \varpi$ is well-defined and it gives a desired embedding.
    The action of $\varphi$ and $\gamma\in\Gamma_{K}$ are given by $\varphi(\varpi)=(1+\varpi)^p-1$ and $\gamma(\varpi)=(1+\varpi)^{\chi(\gamma)}-1,$ where $\chi$ is the $p$-adic cyclotomic character.
\end{example}
\begin{remark}\label{rem:deperfection}
	By the construction in \cite{Col08}, $B_{\Qbb_p}^{[r,s]}$ is contained in $B_K^{[r,s]}$, where we regard $\tilde{B}_{\Qbb_p(\zeta_{p^{\infty}})}^{[r,s]}$ as a subalgebra of $\tilde{B}_{K_{\infty}}^{[r,s]}$ by the natural way, and $B_K^{[r,s]}$ is a finite free $B_{\Qbb_p}^{[r,s]}$-module (of rank $[K_{\infty}\colon \Qbb_p(\zeta_{p^{\infty}})]$).
	Moreover, the natural morphism 
	$$B_K^{[r,s]} \otimes_{B_{\Qbb_p}^{[r,s]}}B_{\Qbb_p,n}^{[r,s]}\to B_{K,n}^{[r,s]}$$ 
	becomes an isomorphism for every $n$.
\end{remark}
\begin{remark}\label{rem:open subspace}
	For closed intervals $[r_1,s_1]\subset [r_2,s_2]\subset (0,s^{\prime}]$, the natural morphism
	$$(B_K^{[r_1,s_1]},B_K^{[r_1,s_1]+})_{\square} \otimes_{(B_K^{[r_2,s_2]},B_K^{[r_2,s_2],+})_{\square}}(B_{K,n}^{[r_2,s_2]},B_{K,n}^{[r_2,s_2],+})_{\square}\to (B_{K,n}^{[r_1,s_1]},B_{K,n}^{[r_1,s_1],+})_{\square}$$
	is an isomorphism by the construction in \cite{Col08}.
	Therefore, we get an isomorphism
	$$(B_K^{[r_1,s_1]},B_K^{[r_1,s_1]+})_{\square} \otimes_{(B_K^{[r_2,s_2]},B_K^{[r_2,s_2],+})_{\square}}(B_{K,\infty}^{[r_2,s_2]},B_{K,\infty}^{[r_2,s_2],+})_{\square}\to (B_{K,\infty}^{[r_1,s_1]},B_{K,\infty}^{[r_1,s_1],+})_{\square}.$$
	Since the morphism 
	$$\AnSpec (B_K^{[r_1,s_1]},B_K^{[r_1,s_1],+})_{\square} \to \AnSpec (B_K^{[r_2,s_2]},B_K^{[r_2,s_2],+})_{\square}$$ 
	is an open immersion (in the associated locale), so is the morphism 
	$$\AnSpec (B_{K,\infty}^{[r_1,s_1]},B_{K,\infty}^{[r_1,s_1],+})_{\square} \to \AnSpec (B_{K,\infty}^{[r_2,s_2]},B_{K,\infty}^{[r_2,s_2],+})_{\square}.$$
\end{remark}

We have the normalized trace morphisms $R_n\colon \tilde{B}_{K_{\infty}}^{[r,s]} \to B_{K,n}^{[r,s]}$ which satisfy the Tate-Sen axioms (C1), (C2), and (C3) by \cite[Proposition 1.1.12]{Ber08B-pair}, \cite[Corollary 9.5]{Col08}, and \cite[Example 5.5 (2)]{Por24}.
We assume that the closed interval $[r,s]$ is of the form $[p^{-k},p^{-l}]$ for some integers $k\geq l$.
Then, by \cite[Theorem 4.4]{Ber16} and \cite{Bergererrata}, they also satisfy the condition (C4), see also the proof of \cite[Proposition 3.23]{Mikami24}.
By the description of Example \ref{ex:unramified}, the condition (C$5^{\prime}$) is satisfied when $K/\Qbb_p$ is unramified. 
In general, (C$5^{\prime}$) is satisfied by Remark \ref{rem:deperfection}.
It is clear that they satisfy the condition (C$6^{\prime}$).
Therefore, we get the following theorem:
\begin{theorem}\label{thm:Tate-Sen smooth Robba}
	Let $k\geq 0$ be an integer.
	Then, the morphism 
	$$\AnSpec(B_{K,\infty}^{[p^{-k-l},p^{-l}]},B_{K,\infty}^{[p^{-k-l},p^{-l}],+})_{\square}/\Gamma_K^{\la}\to \AnSpec\Qbb_{p,\square}$$
	is $\Dcal$-smooth for an integer $l$ sufficiently large.
\end{theorem}

\begin{corollary}\label{cor:Tate-Sen smooth Robba 1}
For a closed interval $[r,s]\subset (0,p^{-1}s^{\prime}]$, the morphism 
	$$\AnSpec(B_{K,\infty}^{[r,s]},B_{K,\infty}^{[r,s],+})_{\square}/\Gamma_K^{\la}\to \AnSpec\Qbb_{p,\square}$$
	is $\Dcal$-smooth.
\end{corollary}
\begin{proof}
	By applying $\varphi^n$ for $n$ sufficiently large, we may assume that $[r,s]\subset [p^{-k-l},p^{-l}]$ for some integers $k,l$ as in Theorem \ref{thm:Tate-Sen smooth Robba}.
	Then the claim easily follows from Theorem \ref{thm:Tate-Sen smooth Robba}, Remark \ref{rem:open subspace}, Remark \ref{rem:D-etale D-smooth}, and Lemma \ref{lem:open etale}.
\end{proof}

For any closed interval $I=[r,s]\subset (0,\infty)$, we define $B_{K,\infty}^{I}=(\tilde{B}_{K_{\infty}}^{I})^{\la}$.
If $I=[r,s]\subset (0,s^{\prime}]$, then this definition coincides with the previous definition by \cite[Corollary 3.24]{Mikami24}.
We set $B_{K,\infty}^{[r,s],+}\coloneqq B_{K,\infty}^{[r,s]} \cap \tilde{B}_{K_{\infty}}^{[r,s],+}\subset B_{K,\infty}^{[r,s]}$.

\begin{corollary}\label{cor:Tate-Sen smooth Robba}
	The morphism 
	$$\AnSpec(B_{K,\infty}^{I},B_{K,\infty}^{I,+})_{\square}/\Gamma_K^{\la}\to \AnSpec\Qbb_{p,\square}$$
	is $\Dcal$-smooth for any closed interval $I=[r,s]\subset (0,\infty)$.
\end{corollary}
\begin{proof}
	By applying $\varphi^n$ for $n$ sufficiently large, we reduce to the case where $I=[r,s]\subset (0,p^{-1}s^{\prime}]$.
\end{proof}

\begin{definition}\label{defn:locally analytic Fargues-Fontaine curve}
	For a closed interval $I=[r,s]\subset (0,\infty)$, we set 
	$$Y_K^{I,\la}=\AnSpec(B_{K,\infty}^{I},B_{K,\infty}^{I,+})_{\square}.$$
	We define a solid $\Dcal$-stack $Y_K^{\la}$ as 
	$$Y_K^{\la} =\varinjlim_{I} Y_K^{I,\la}.$$
	The action of $\Gamma_K$ and the Frobenius automorphism on $Y_{K_{\infty}}$ induce a locally analytic $\Gamma_K$-action on $Y_K^{\la}$ and the Frobenius automorphism 
	$\varphi \colon Y_K^{\la} \to Y_K^{\la}$, respectively.
	We define a solid $\Dcal$-stack $X_K^{\la}$ by $X_K^{\la}=Y_K^{\la}/\varphi^{\Zbb}$.
	Since the locally analytic action of $\Gamma_K$ and the Frobenius automorphism on $Y_K^{\la}$ commute with each other, $X_K^{\la}$ has a locally analytic action of $\Gamma_K$.
\end{definition}

\begin{remark}\label{rem:open cover}
	From the construction, $X_K^{\la}$ is obtained by identifying open subspaces $Y_K^{[p,p],\la}, Y_K^{[1,1],\la}\subset Y_K^{[1,p],\la}$ via the Frobenius automorphism $$\varphi\colon Y_K^{[1,1],\la}\to Y_K^{[p,p],\la}.$$
	%In particular, $X_K^{\la}$ is an analytic space in the sense of \cite[Definition 13.5]{AG}.
	Moreover $X_K^{\la}$ has a universal $\Dcal^*$-cover 
	$$\{Y_K^{[1,3/2],\la}\hookrightarrow X_K^{\la}, Y_K^{[3/2,p],\la} \hookrightarrow X_K^{\la}\}$$
	such that each morphism is a monomorphism (i.e., its diagonal morphism is an equivalence) and $\Dcal$-\'{e}tale.
\end{remark}

\begin{corollary}\label{cor:X D-smooth}
	The morphism $X_K^{\la}/\Gamma_K^{\la}\to \AnSpec \Qbb_{p,\square}$ is $\Dcal$-smooth.
\end{corollary}
\begin{proof}
	It easily follows from Corollary \ref{cor:Tate-Sen smooth Robba} and Remark \ref{rem:open cover}.
\end{proof}

We will compute the dualizing complex later by comparing it with the known duality.
It can be also computed directly by unwinding the proof of Theorem \ref{thm:smooth Tate-Sen dualizing complex}, but we omit this.

Let $f\colon X_K^{\la} \to \AnSpec\Qbb_{p,\square}$ and $g\colon X_K^{\la}/\Gamma_K^{\la} \to \AnSpec\Qbb_{p,\square}$ be the structure morphisms, and $\overline{f}\colon X_K^{\la}/\Gamma_K^{\la} \to \AnSpec\Qbb_{p,\square}/\Gamma_K^{\la}$ be the quotient of $f$.
For $\Acal\in \AffRing_{\Qbb_{p,\square}}$, let $(-)_{\Acal}$ denote the base change to $\AnSpec\Acal$.
For example, $X_{K,\Acal}^{\la}$ means $X_K^{\la}\times_{\AnSpec\Qbb_{p,\square}}\AnSpec\Acal$.
Let us compute $f_{\Acal *}$, $\overline{f}_{\Acal *}$, and $g_{\Acal *}$.

\begin{lemma}\label{lem:quasi-coherent sheaves on FF-curve}
	For $\Acal\in \AffRing_{\Qbb_{p,\square}}$, there exists a natural equivalence of $\infty$-categories
	$$\Dcal(X_{K,\Acal}^{\la})\simeq \Eq(\Dcal(Y_{K,\Acal}^{[1,p],\la})\rightrightarrows \Dcal(Y_{K,\Acal}^{[1,1],\la})),$$
	where two functors $\Dcal(Y_{K,\Acal}^{[1,p],\la})\to \Dcal(Y_{K,\Acal}^{[1,1],\la})$ are the pull-back functors along the open immersion 
	$$Y_{K,\Acal}^{[1,1],\la}\hookrightarrow Y_{K,\Acal}^{[1,p],\la}$$ 
	and the Frobenius morphism
	$$\varphi \colon Y_{K,\Acal}^{[1,1],\la}\hookrightarrow Y_{K,\Acal}^{[1,p],\la}.$$
	Similarly, there exists a natural equivalence of $\infty$-categories
	$$\Dcal(X_{K,\Acal}^{\la}/\Gamma_K^{\la})\simeq \Eq(\Dcal(Y_{K,\Acal}^{[1,p],\la}/\Gamma_K^{\la})\rightrightarrows \Dcal(Y_{K,\Acal}^{[1,1],\la}/\Gamma_K^{\la})).$$
\end{lemma}
\begin{proof}
	It directly follows from the description in Remark \ref{rem:open cover}.
\end{proof}

\begin{notation}
	Let $\Acal\in \AffRing_{\Qbb_{p,\square}}$ be a solid affinoid animated $\Qbb_{p,\square}$-algebra.
	We set $(B_{K,\infty}^{[r,s]},B_{K,\infty}^{[r,s],+})_{\square,\Acal}=(B_{K,\infty}^{[r,s]},B_{K,\infty}^{[r,s],+})_{\square}\otimes_{\Qbb_{p,\square}}\Acal.$
	For an object $\Mcal \in \Dcal(X_{K,\Acal}^{\la})$ (resp. $\Mcal \in \Dcal(X_{K,\Acal}^{\la}/\Gamma_K^{\la})$), let $\Mcal^{[r,s]}\in \Dcal(Y_{K,\Acal}^{[r,s],\la})$ (resp. $\Mcal^{[r,s]}\in \Dcal(Y_{K,\Acal}^{[r,s],\la}/\Gamma_K^{\la})$) denote the pull-back of $\Mcal$ to $Y_{K,\Acal}^{[r,s],\la}$ (resp. $Y_{K,\Acal}^{[r,s],\la}/\Gamma_K^{\la}$).
	We let $\Phi_{\Mcal} \colon \Mcal^{[r,s]} \to \Mcal^{[r/p,s/p]}$ denote the natural $\varphi$-semilinear morphism. In other words, it is $(B_{K,\infty}^{[r,s]},B_{K,\infty}^{[r,s],+})_{\square,\Acal}$-linear, where we regard $\Mcal^{[r/p,s/p]}$ as a $(B_{K,\infty}^{[r,s]},B_{K,\infty}^{[r,s],+})_{\square,\Acal}$-module via the Frobenius morphism 
	$$\varphi\colon (B_{K,\infty}^{[r,s]},B_{K,\infty}^{[r,s],+})_{\square,\Acal} \to (B_{K,\infty}^{[r/p,s/p]},B_{K,\infty}^{[r/p,s/p],+})_{\square,\Acal}.$$
	We also let $1\colon \Mcal^{[r,s]} \to \Mcal^{[r^{\prime},s^{\prime}]}$ denote the restriction morphism for $[r^{\prime},s^{\prime}]\subset [r,s]$, which is $(B_{K,\infty}^{[r,s]},B_{K,\infty}^{[r,s],+})_{\square,\Acal}$-linear, where we regard $\Mcal^{[r^{\prime},s^{\prime}]}$ as a $(B_{K,\infty}^{[r,s]},B_{K,\infty}^{[r,s],+})_{\square,\Acal}$-module via the restriction morphism 
	$$\varphi\colon (B_{K,\infty}^{[r,s]},B_{K,\infty}^{[r,s],+})_{\square,\Acal} \to (B_{K,\infty}^{[r^{\prime},s^{\prime}]},B_{K,\infty}^{[r^{\prime},s^{\prime}],+})_{\square,\Acal}.$$ 
	We also write $\Phi_{\Mcal}$ for the composition of $\Phi_{\Mcal}$ and $1$.
\end{notation}

\begin{corollary}\label{cor:cohomology explicit}
	Let $\Acal\in \AffRing_{\Qbb_{p,\square}}$ be a solid affinoid animated $\Qbb_{p,\square}$-algebra.
	\begin{enumerate}
		\item For an object $\Mcal \in \Dcal(X_{K,\Acal}^{\la})$, there is a natural equivalence 
		$$f_{\Acal *}\Mcal\simeq \fib(\Phi_{\Mcal}-1\colon \Mcal^{[1,p]}\to \Mcal^{[1,1]})$$
		in $\Dcal(\Acal)$.
		\item For an object $\Mcal \in \Dcal(X_{K,\Acal}^{\la}/\Gamma_K^{\la})$, there is a natural equivalence 
		$$\overline{f}_{\Acal *}\Mcal\simeq \fib(\Phi_{\Mcal}-1\colon \Mcal^{[1,p]}\to \Mcal^{[1,1]})$$
		in $\Dcal(\Acal/\Gamma_K^{\la})=\Rep_{\Acal}^{\la}(\Gamma_K)$.
		\item For an object $\Mcal \in \Dcal(X_{K,\Acal}^{\la}/\Gamma_K^{\la})$, there is a natural equivalence 
		$$g_{\Acal *}\Mcal\simeq \Gamma(\Gamma_K,\fib(\Phi_{\Mcal}-1\colon \Mcal^{[1,p]}\to \Mcal^{[1,1]}))$$
		in $\Dcal(\Acal)$.
	\end{enumerate}
\end{corollary}
\begin{proof}
	It follows from Lemma \ref{lem:quasi-coherent sheaves on FF-curve} and Proposition \ref{prop:6-ff classifying stack} (3).
\end{proof}

\begin{notation}
	For an object $\Mcal \in \Dcal(X_{K,\Acal}^{\la})$, we write
	$$\Gamma(X_{K,\Acal}^{\la},\Mcal)=f_{\Acal *}\Mcal\in \Dcal(\Acal).$$
	Similarly, for an object $\Ncal \in \Dcal(X_{K,\Acal}^{\la}/\Gamma_K^{\la})$, we write 
	\begin{align*}
		&\Gamma(X_{K,\Acal}^{\la},\Ncal)=\overline{f}_{\Acal *}\Ncal \in \Rep_{\Acal}^{\la}(\Gamma_K),\\
		&\Gamma(X_{K,\Acal}^{\la}/\Gamma_K^{\la},\Ncal)=g_{\Acal *}\Ncal \in \Dcal(\Acal).
	\end{align*}
	Moreover, for $i\in \Zbb$, we write
	\begin{align*}
	&H^i(X_{K,\Acal}^{\la},\Mcal)=H^i(\Gamma(X_{K,\Acal}^{\la},\Mcal))\in \Dcal(\Acal)^{\heart},\\
	&H^i(X_{K,\Acal}^{\la}/\Gamma_K^{\la},\Ncal)=H^i(\Gamma(X_{K,\Acal}^{\la}/\Gamma_K^{\la},\Ncal))\in \Dcal(\Acal)^{\heart}
    \end{align*}
	for $\Mcal \in \Dcal(X_{K,\Acal}^{\la})$ and $\Ncal \in \Dcal(X_{K,\Acal}^{\la}/\Gamma_K^{\la})$.
\end{notation}

\begin{lemma}\label{lem:FF curve proper}
	The morphism $X_K^{\la} \to \AnSpec\Qbb_{p,\square}$ is weakly $\Dcal$-proper.
\end{lemma}
\begin{proof}
	The diagonal morphism $\Delta \colon X_K^{\la} \to X_K^{\la}\times X_K^{\la}$ satisfies the condition in Corollary \ref{cor:weakly proper}, so it is weakly $\Dcal$-proper.
	Therefore, it is enough to show that $X_K^{\la} \to \AnSpec\Qbb_{p,\square}$ is $\Dcal$-prim by the description of the codualizing complex in Corollary \ref{cor:prim map characterization}.
	We write $Z_1=\AnSpec (B_{K,\infty}^{[1,3/2]})_{\square}$ and $Z_2=\AnSpec (B_{K,\infty}^{[3/2,p]})_{\square}$.
	Then $\{Z_1\to X_K^{\la}, Z_2 \to X_K^{\la}\}$ is a universal $\Dcal^*$-cover because it admits a refinement by the universal $\Dcal^*$-cover in Remark \ref{rem:open cover}.
	Since $Z_i\to \AnSpec \Qbb_{p,\square}$ and the diagonal morphism of $X_K^{\la} \to \AnSpec\Qbb_{p,\square}$ are weakly $\Dcal$-proper, $Z_i\to X_K^{\la}$ is also weakly $\Dcal$-proper for $i=1,2$ by Lemma \ref{lem:proper composition basechange} and the standard argument.
	Moreover, the diagonal morphism of $Z_i\to X_K^{\la}$ is an equivalence (i.e., $Z_i\to X_K^{\la}$ is a monomorphism).
	Therefore, the claim follows from Lemma \ref{lem:suave suave-local} and Remark \ref{rem:mono cech cover finite}.
\end{proof}

\begin{corollary}\label{cor:X weakly proper}
	The morphism $X_K^{\la}/\Gamma_K^{\la} \to \AnSpec\Qbb_{p,\square}$ is weakly $\Dcal$-proper.
\end{corollary}
\begin{proof}
	It follows from Lemma \ref{lem:FF curve proper} and Lemma \ref{lem:suave smooth classifying stack}.
\end{proof}

\begin{corollary}
	Let $\Acal\to \Bcal$ be a morphism in $\AffRing_{\Qbb_{p,\square}}$.
	Then, for $\Mcal \in \Dcal(X_{K,\Acal}^{\la})$, the natural morphism
	\begin{align*}
		\Gamma(X_{K,\Acal}^{\la}/\Gamma_K^{\la},\Mcal)\otimes_{\Acal}\Bcal \to \Gamma(X_{K,\Bcal}^{\la}/\Gamma_K^{\la},\Mcal\otimes_{\Acal} \Bcal)
	\end{align*}
	is an equivalence, where $\Mcal\otimes_{\Acal} \Bcal$ is the pullback of $\Mcal$ along $X_{K,\Bcal}^{\la}/\Gamma_K^{\la}\to X_{K,\Acal}^{\la}/\Gamma_K^{\la}$.
\end{corollary}
\begin{proof}
	It follows from the proper base change and Corollary \ref{cor:X weakly proper}.
\end{proof}

We compare quasi-coherent sheaves on $X_{K,\Acal}^{\la}$ and their cohomology with the results in \cite{Mikami24}.
First, we introduce the notion of vector bundles on solid $\Dcal$-stacks.

\begin{definition}[{\cite[Definition 4.1.1]{RC24}}]
	Let $X\in \Shv_{\Dcal}(\Aff_{Z_{\square}})$ be a solid $\Dcal$-stack.
	An object $\Fcal \in \Dcal(X)$ is a \textit{vector bundle on $X$} if there is a canonical cover $\{f_i\colon \AnSpec \Acal_i \to X\}_i$ of $X$ such that for every $i$, $f_i^*\Fcal\in \Dcal(\Acal_i)$ is a finite projective $\Acal_i$-module, that is, it is a direct summand of $(\underline{\Acal_i})^n$ for some integer $n\geq 0$.
	Let $\Vect(X)$ denote the full subcategory of $\Dcal(X)$ consisting of vector bundles on $X$.
\end{definition}
\begin{remark}
	Let $X$ be a solid $\Dcal$-stack with a locally analytic action of a $p$-adic Lie group $G$.
	Since $X\to X/\Gbb^{\la}$ is a canonical cover, an object $\Fcal \in \Dcal(X/\Gbb^{\la})$ is a vector bundle on $X/\Gbb^{\la}$ if and only if its pullback to $X$ is a vector bundle on $X$.
\end{remark}

This definition is very subtle.
For example, the author does not know whether for $\Acal\in \AffRing_{\Zbb_{\square}}$, an object $\Fcal \in \Dcal(\AnSpec \Acal)$ is a vector bundle if and only if $\Fcal$ comes from a finite projective $\underline{\Acal}$-module.
However it behaves well for the following nice analytic rings.

\begin{definition}[{\cite[Definition 9.7]{CC}}]\label{def:Fredholm}
    A solid affinoid animated $\Zbb_{\square}$-algebra $\Acal$ is \textit{Fredholm} if any dualizable object of $\Dcal(\Acal)$ is relatively discrete.
\end{definition}
\begin{example}
	Let $(A,A^+)$ be a complete Tate affinoid pair.
    Then the analytic ring $(A,A^+)_{\square}$ is Fredholm, which follows from \cite[Theorem 5.50]{And21}.
\end{example}

\begin{lemma}\label{lem:Freholm colimit}
	Let $\{\Acal_{\lambda}\}_{\lambda\in \Lambda}$ be a filtered diagram of Fredholm solid affinoid animated $\Zbb_{\square}$-algebras.
	Then $\Acal=\varinjlim_{\lambda}\Acal_{\lambda}$ is also Fredholm.
\end{lemma}
\begin{proof}
	Let $V\in \Dcal(\Acal)$ be a dualizable object.
	Then by \cite[Lemma 2.7.4]{Mann22}, there exists a dualizable object $M_{\lambda} \in \Dcal(\Acal_{\lambda})$ for some $\lambda$ such that $M_{\lambda}\otimes_{\Acal_{\lambda}}\Acal \simeq M$.
	Since $\Acal_{\lambda}$ is Fredholm, $M_{\lambda}$ is relatively discrete.
	Therefore, $M$ is also relatively discrete.
\end{proof}
\begin{corollary}\label{cor:la vector Fredholm}
	For a closed interval $I=[r,s]\subset (0,\infty)$, $(B_{K,\infty}^{I},B_{K,\infty}^{I,+})_{\square}$ is Fredholm.
	In general, for a Banach $\Qbb_{p}$-algebra $A$, $(B_{K,\infty}^{I},B_{K,\infty}^{I,+})_{\square,A_{\square}}$ is also Fredholm.
\end{corollary}

\begin{lemma}[{\cite[Corollary IV.1.4]{GRLLC24}}]\label{lem:Fredholm vector bundle}
Let $\Acal$ be a Fredholm solid affinoid animated $\Zbb_{\square}$-algebra.
Then an object $\Fcal \in \Dcal(\AnSpec \Acal)$ is a vector bundle if and only if $\Fcal$ comes from a finite projective $\underline{\Acal}$-module.
\end{lemma}

Let $A$ be a Banach $\Qbb_p$-algebra or an algebraic-affinoid $\Qbb_{p,\square}$-algebra.
We write $B_{K,\infty,A}^{[r,s]}=B_{K,\infty}^{[r,s]}\otimes_{\Qbb_{p,\square}}A$.
In \cite[Definition 3.4, Definition 3.27, Definition 3.29]{Mikami24}, the author introduced the category $\VB_{\tilde{B}_{K_{\infty},A}}^{\varphi,\Gamma_K}$ (resp. $\VB_{B_{K,\infty,A}}^{\varphi,\Gamma_K}$, $\VB_{B_{K,A}}^{\varphi,\Gamma_K}$) of $(\varphi,\Gamma_K)$-modules over $\tilde{B}_{K_{\infty}, A}$ (resp. $B_{K,\infty,A}$, $B_{K,A}$).
Let us recall the definition of $(\varphi,\Gamma_K)$-modules over $B_{K,\infty,A}$.

\begin{definition}[{\cite[Definition 3.27]{Mikami24}}]
	A \textit{$(\varphi,\Gamma_K)$-module} $\Mcal$ over $B_{K,\infty,A}$ is a family of finite projective $B_{K,\infty,A}^{[r,s]}$-modules $\{M^{[r,s]}\}_{r,s}$ for $[r,s]\subset (0,\infty)$ with a continuous semilinear action of $\Gamma_K$ on $M^{[r,s]}$ and $\Gamma_K$-equivariant isomorphisms 
\begin{align*}
    &\tau_{r,r^{\prime},s^{\prime},s}\colon M^{[r,s]}\otimes_{(B_{K,\infty,A}^{[r,s]})_{\square}} B_{K,\infty,A}^{[r^{\prime},s^{\prime}]} \overset{\sim}{\to} M^{[r^{\prime},s^{\prime}]}\\
&\Phi_{r,s}\colon M^{[r,s]}\otimes_{(B_{K,\infty,A}^{[r,s]})_{\square},\varphi} B_{K,\infty,A}^{[r/p,s/p]}\cong M^{[r/p,s/p]}
\end{align*}
for $[r^{\prime},s^{\prime}]\subset [r,s]$ which satisfy the usual cocycle conditions, see \cite[Definition 3.4]{Mikami24}.
\end{definition}
\begin{remark}
	By \cite[Theorem 3.37]{Mikami24}, the continuous semilinear action of $\Gamma_K$ on a finite projective $B_{K,\infty,A}^{[r,s]}$-module $M^{[r,s]}$ is automatically locally analytic.
\end{remark}

By the description in Lemma \ref{lem:quasi-coherent sheaves on FF-curve}, we have a natural fully faithful functor 
$$\VB_{B_{K,\infty,A}}^{\varphi,\Gamma_K} \hookrightarrow \Dcal(X_{K,A_{\square}}^{\la}/\Gamma_K^{\la}).$$

\begin{proposition}
	Let $A$ be a Banach $\Qbb_p$-algebra.
	Then the fully faithful functor 
	$$\VB_{B_{K,\infty,A}}^{\varphi,\Gamma_K}\hookrightarrow \Vect(X_{K,A_{\square}}^{\la}/\Gamma_K^{\la})$$
	is an equivalence of $\infty$-categories.
\end{proposition}
\begin{proof}
	It easily follows from Corollary \ref{cor:la vector Fredholm} and Lemma \ref{lem:Fredholm vector bundle}.
\end{proof}

\begin{remark}
	If $A$ is not Fredholm, then the author does not know whether the natural fully faithful functor 
	$$\VB_{B_{K,\infty,A}}^{\varphi,\Gamma_K}\hookrightarrow \Vect(X_{K,A_{\square}}^{\la}/\Gamma_K^{\la})$$
	is an equivalence.
\end{remark}

Let us recall the definition cohomology of $(\varphi,\Gamma_K)$-module $\Mcal$ over $B_{K,\infty,A}$ in \cite{Mikami24}.

\begin{definition}[{\cite[Definition 4.1]{Mikami24}}]
	Let $\Mcal=\{M^{[r,s]}\}_{r,s}$ be a $(\varphi,\Gamma_K)$-module over $B_{K,\infty,A}$.
	Then we define the $(\varphi,\Gamma_K)$-cohomology of $\Mcal$ as 
	$$\Gamma_{\varphi,\Gamma_K} (\Mcal)=\Gamma(\Gamma_K,\fib(\Phi-1\colon\varinjlim_{0<s}\varprojlim_{0<r<s}M^{[r,s]}\to \varinjlim_{0<s}\varprojlim_{0<r<s}M^{[r,s]}))\in \Dcal(A_{\square}),$$
	where $\Phi$ is induced by the $\varphi$-semilinear morphisms
	$$\Phi\colon M^{[r,s]}\to M^{[r/p,s/p]}.$$
\end{definition}

\begin{proposition}\label{prop:comparison cohomology}
	Let $\Mcal=\{M^{[r,s]}\}_{r,s}$ be a $(\varphi,\Gamma_K)$-module over $B_{K,\infty,A}$.
	We also denote the image of $\Mcal$ under the fully faithful functor $\VB_{B_{K,\infty,A}}^{\varphi,\Gamma_K} \hookrightarrow \Dcal(X_{K,A_{\square}}^{\la}/\Gamma_K^{\la})$ by the same letter $\Mcal$.
	Then there is a natural equivalence
	\begin{align*}
		\Gamma_{\varphi,\Gamma_K} (\Mcal) \simeq \Gamma(X_{K,A_{\square}}^{\la}/\Gamma_K^{\la},\Mcal)
	\end{align*}
	in $\Dcal(A_{\square})$.
\end{proposition}
\begin{proof}
	It follows from the description in Corollary \ref{cor:cohomology explicit} and \cite[Lemma 4.2]{Mikami24}.
\end{proof}

Finally, let us compute the dualizing complex of $g\colon X_K^{\la}/\Gamma_K^{\la} \to \AnSpec\Qbb_{p,\square}$.
We define $\Ocal_{X_K^{\la}}\cdot\chi\in \Vect(X_{K}^{\la}/\Gamma_K^{\la})\subset \Dcal(X_{K}^{\la}/\Gamma_K^{\la})$ as the twist of the monoidal unit $\Ocal_{X_K^{\la}}$ by the $p$-adic cyclotomic character $\chi$.
\begin{remark}
	Such a Tate twist is usually denoted by $\Ocal_{X_K^{\la}}(1)$.
	However, this notation can be confused with the line bundle $\Ocal_{X_K}(1)$ of slope $1$ on the Fargues-Fontaine curve.
	Therefore, we write $\Ocal_{X_K^{\la}}\cdot\chi$ for the Tate twist.
\end{remark}

\begin{theorem}\label{thm:Local Tate duality}
	\begin{enumerate}
		\item The invertible object $\Ocal_{X_K^{\la}}\cdot\chi[2]$ is the dualizing complex of $g\colon X_K^{\la}/\Gamma_K^{\la} \to \AnSpec\Qbb_{p,\square}$.
		\item There is an equivalence
		$$H^2(X_K^{\la}/\Gamma_K^{\la},\Ocal_{X_K^{\la}}\cdot\chi)\cong \Qbb_p$$
		and the trace morphism $\Tr\colon g_!g^!\Qbb_p\simeq \Gamma(X_K^{\la}/\Gamma_K^{\la}, \Ocal_{X_K^{\la}}\cdot\chi[2])\to \Qbb_p$ is induced by the above equivalence.
		\item For $\Mcal \in \Dcal(X_K^{\la}/\Gamma_K^{\la})$, the Poincar\'{e} duality 
		\begin{align*}
			\intHom_{\Qbb_p}(\Gamma(X_K^{\la}/\Gamma_K^{\la},\Mcal),\Qbb_p)\simeq \Gamma(X_K^{\la}/\Gamma_K^{\la},\intHom_{X_K^{\la}/\Gamma_K^{\la}}(\Mcal,\Ocal_{X_K^{\la}}\cdot\chi[2]))
		\end{align*}
		is induced by
		\begin{equation}\label{eq7}
		\begin{aligned}
			&\Gamma(X_K^{\la}/\Gamma_K^{\la},\intHom_{X_K^{\la}/\Gamma_K^{\la}}(\Mcal,\Ocal_{X_K^{\la}}\cdot\chi[2])) \otimes_{\Qbb_{p,\square}} \Gamma(X_K^{\la}/\Gamma_K^{\la},\Mcal)\\
			\to &\Gamma(X_K^{\la}/\Gamma_K^{\la}, \intHom_{X_K^{\la}/\Gamma_K^{\la}}(\Mcal,\Ocal_{X_K^{\la}}\cdot\chi[2])\otimes \Mcal)\\
			\to &\Gamma(X_K^{\la}/\Gamma_K^{\la}, \Ocal_{X_K^{\la}}\cdot\chi[2])\\
			\overset{\Tr}{\to} &\Qbb_p.
		\end{aligned}
	\end{equation}
	\end{enumerate}
\end{theorem}
\begin{proof}
	By Corollary \ref{cor:Poincare duality proper}, the Poincar\'{e} duality 
	\begin{align*}
		\intHom_{\Qbb_p}(\Gamma(X_K^{\la}/\Gamma_K^{\la},\Mcal),\Qbb_p)\simeq \Gamma(X_K^{\la}/\Gamma_K^{\la},\intHom_{X_K^{\la}/\Gamma_K^{\la}}(\Mcal,\omega_g))
	\end{align*}
	is induced by
	\begin{align*}
		&\Gamma(X_K^{\la}/\Gamma_K^{\la},\intHom_{X_K^{\la}/\Gamma_K^{\la}}(\Mcal,\omega_g)) \otimes_{\Qbb_{p,\square}} \Gamma(X_K^{\la}/\Gamma_K^{\la},\Mcal)\\
		\to &\Gamma(X_K^{\la}/\Gamma_K^{\la}, \intHom_{X_K^{\la}/\Gamma_K^{\la}}(\Mcal,\omega_g)\otimes \Mcal)\\
		\to &\Gamma(X_K^{\la}/\Gamma_K^{\la}, \omega_g)\\
		\to &\Qbb_p.
	\end{align*}
	On the other hand, by the Tate local duality for $(\varphi,\Gamma_K)$-modules (\cite{Liu08}), the morphism 
	\begin{align*}
		\Gamma(X_K^{\la}/\Gamma_K^{\la},\intHom_{X_K^{\la}/\Gamma_K^{\la}}(\Mcal,\Ocal_{X_K^{\la}}\cdot\chi[2])) \to \intHom_{\Qbb_p}(\Gamma(X_K^{\la}/\Gamma_K^{\la},\Mcal),\Qbb_p)
	\end{align*}
	induced by \eqref{eq7} is an equivalence for $\Mcal \in \Vect(X_K^{\la}/\Gamma_K^{\la})$.
	By Theorem \ref{thm:smooth Tate-Sen dualizing complex}, we find $\omega_g[-2]\in\Vect(X_K^{\la}/\Gamma_K^{\la})$.
	Therefore, by the standard argument, we can show that the pair $(\omega_g, \Gamma(X_K^{\la}/\Gamma_K^{\la}, \omega_g)\simeq g_!\omega_g\to \Qbb_p)$ is equivalent to the pair $(\Ocal_{X_K^{\la}}\cdot\chi[2], \Gamma(X_K^{\la}/\Gamma_K^{\la}, \Ocal_{X_K^{\la}}\cdot\chi[2])\to \Qbb_p)$.
\end{proof}

\begin{corollary}
	Let $\Acal \in \AffRing_{\Qbb_{p,\square}}$ be a solid affinoid animated $\Qbb_{p,\square}$-algebra.
	For $\Mcal \in \Dcal(X_{K,\Acal}^{\la}/\Gamma_K^{\la})$, the Poincar\'{e} duality 
		\begin{align*}
			\intHom_{\Acal}(\Gamma(X_{K,\Acal}^{\la}/\Gamma_K^{\la},\Mcal),\underline{\Acal})\simeq \Gamma(X_{K,\Acal}^{\la}/\Gamma_K^{\la},\intHom_{X_{K,\Acal}^{\la}/\Gamma_K^{\la}}(\Mcal,\Ocal_{X_{K,\Acal}^{\la}}\cdot\chi[2]))
		\end{align*}
		is induced by
		\begin{align*}
			&\Gamma(X_{K,\Acal}^{\la}/\Gamma_K^{\la},\intHom_{X_{K,\Acal}^{\la}/\Gamma_K^{\la}}(\Mcal,\Ocal_{X_{K,\Acal}^{\la}}\cdot\chi[2])) \otimes_{\Acal} \Gamma(X_{K,\Acal}^{\la}/\Gamma_K^{\la},\Mcal)\\
			\to &\Gamma(X_{K,\Acal}^{\la}/\Gamma_K^{\la}, \intHom_{X_{K,\Acal}^{\la}/\Gamma_K^{\la}}(\Mcal,\Ocal_{X_{K,\Acal}^{\la}}\cdot\chi[2])\otimes \Mcal)\\
			\to &\Gamma(X_{K,\Acal}^{\la}/\Gamma_K^{\la}, \Ocal_{X_{K,\Acal}^{\la}}\cdot\chi[2])\\
			\overset{\Tr}{\to} &\underline{\Acal}.
	    \end{align*}
\end{corollary}
\begin{proof}
	It follows from Corollary \ref{cor:Poincare duality proper} and Theorem \ref{thm:Local Tate duality}.
\end{proof}

%%%%%%%%%%%%%%%%%%%%%%%%%%%%%%%%%%%%%%%%%%%%%% Appendix %%%%%%%%%%%%%%%%%%%%%%%%%%%%%%%%%%%%%%%%
\appendix
\section{6-functor formalisms and Poincar\'{e} duality}
In this appendix, we explore the Poincar\'{e} duality for proper morphisms.
Let $(\Ccal,E)$ be a geometric setup such that $\Ccal$ admits pullbacks, and $\Dcal \colon \Corr(\Ccal,E)\to \Cat_{\infty}$ be a 6-functor formalism.

Let $f\colon X\to S$ be a $\Dcal$-prim morphism with the codualizing complex $\delta_f$.
For $M\in \Dcal(X)$ and $N\in \Dcal(S)$, we have two morphisms from $f_!M\otimes N$ to $f_!(M\otimes f^*N)$.
One is the projection formula morphism 
\begin{align*}
	\pi \colon f_!M\otimes N\overset{\sim}{\to}f_!(M\otimes f^*N).
\end{align*} 
The other one is constructed as follows:
There is a natural morphism
\begin{align}\label{eq1}
	\delta_f\otimes f^*f_*\intHom(\delta_f,M) \otimes f^*N \to \delta_f \otimes \intHom(\delta_f,M) \otimes f^*N \to M\otimes f^*N.
\end{align}
From this, we get a morphism
\begin{align*}
	f^*(f_*\intHom(\delta_f,M) \otimes N) \simeq f^*f_*\intHom(\delta_f,M) \otimes f^*N \to \intHom(\delta_f, M\otimes f^*N)
\end{align*}
by the adjunction, and then we get a morphism
\begin{align*}
	f_*\intHom(\delta_f,M) \otimes N \to f_*\intHom(\delta_f, M\otimes f^*N).
\end{align*}
Under the equivalence $f_!\simeq f_*\intHom(\delta_f,-)$ in Corollary \ref{cor:prim map characterization},
we get a morphism 
\begin{align*}
	\rho \colon f_!M\otimes N\to f_!(M\otimes f^*N).
\end{align*} 

The author learned the proof of the following proposition from Mann.
\begin{proposition}
	The above two morphisms $\pi$ and $\rho$ coincide in the homotopy category of $\Dcal(S)$.
\end{proposition}
\begin{proof}
	We rephrase the claim in terms of the 2-category $\Kcal_{\Dcal,S}$.
	We abbreviate $\Fun_{\Kcal_{\Dcal,S}}(-,-)$ as $\Fun_S(-,-)$.
	Let $f$ denote the object of $\Fun_S(X,S)=\Dcal(X)$ corresponding to $\mathbf{1}_X$.
	First, we interpret $\pi$ in terms of $\Kcal_{\Dcal,S}$.
	The functor
	\begin{align*}
		\Dcal(X)\times \Dcal(S)\to \Dcal(X);\; (M,N)\mapsto M\otimes f^*N
	\end{align*}
	coincides with the functor
	\begin{gather*}
		\Dcal(X)\times \Dcal(S)=\Fun_S(S,X)\times \Fun_S(S,S)\to \Fun_S(S,X)=\Dcal(X);\\
		(M,N)\mapsto M\circ N.
	\end{gather*}
	Therefore, the functor
	\begin{align*}
		\Dcal(X)\times \Dcal(S)\to \Dcal(S);\; (M,N)\mapsto f_!(M\otimes f^*N)
	\end{align*}
	coincides with the functor
	\begin{gather*}
		\Dcal(X)\times \Dcal(S)=\Fun_S(S,X)\times \Fun_S(S,S)\to \Fun_S(S,S)=\Dcal(S);\\
		(M,N)\mapsto f\circ (M\circ N).
	\end{gather*}
	Similarly, the functor 
	\begin{align*}
		\Dcal(X)\times \Dcal(S)\to \Dcal(S);\; (M,N)\mapsto f_!M\otimes N
	\end{align*}
	coincides with the functor
	\begin{gather*}
		\Dcal(X)\times \Dcal(S)=\Fun_S(S,X)\times \Fun_S(S,S)\to \Fun_S(S,S)=\Dcal(S);\\
		(M,N)\mapsto (f\circ M)\circ N.
	\end{gather*}
	Under this identification, the projection formula morphism $\pi\colon f_!M\otimes N\overset{\sim}{\to} f_!(M\otimes f^*N)$ coincides with the associator $\alpha \colon f\circ (M\circ N) \overset{\sim}{\to} (f\circ M)\circ N$.

	Next, we interpret $\rho$ in terms of $\Kcal_{\Dcal,S}$.
	We regard $\delta_f$ as the object of $\Fun_S(S,X)=\Dcal(X)$.
	Then $\delta_f$ is a left adjoint morphism of $f$.
	Therefore, there is an adjunction
	\begin{gather*}
		\Dcal(X) \rightleftarrows \Dcal(S);\;
		M \mapsto f_!M,\;
		\delta_f \otimes f^*N \mapsfrom N,
	\end{gather*}
	and it coincides with the adjunction
	\begin{gather}\label{eq2}
		\Dcal(X)=\Fun_S(S,X) \rightleftarrows \Fun_S(S,S)=\Dcal(S);\;
		M \mapsto f\circ M,\;
		\delta_f \circ N \mapsfrom N.
	\end{gather}
	We note that the functor $\Dcal(S)\to \Dcal(X) ;\; N \mapsto \delta_f \otimes f^*N$ admits a right adjoint functor $\Dcal(X)\to\Dcal(S);\; M\mapsto f_*\intHom(\delta_f,M)$, and therefore, we get the equivalence $f_!\simeq f_*\intHom(\delta_f,-)$ in Corollary \ref{cor:prim map characterization}.
	Let $\varepsilon \colon \delta_f \circ f \to \id$ denote the counit of the adjunction $(\delta_f,f)$ in the 2-category $\Kcal_{\Dcal,S}$.
	Then the counit of the adjunction \eqref{eq2} is given by 
	\begin{align*}
		\delta_f\circ(f\circ M) \simeq (\delta_f\circ f)\circ M \overset{\varepsilon}{\to} \id \circ M \simeq M.
	\end{align*}
	The functor 
	\begin{align*}
		\Dcal(X)\times \Dcal(S)\to \Dcal(S);\; (M,N)\mapsto \delta_f\otimes f^*f_*\intHom(\delta_f,M) \otimes f^*N
	\end{align*}
	corresponds to the functor
	\begin{gather*}
		\Dcal(X)\times \Dcal(S)=\Fun_S(S,X)\times \Fun_S(S,S)\to \Fun_S(S,S)=\Dcal(S);\\
		(M,N)\mapsto (\delta_f\circ(f\circ M))\circ N.
	\end{gather*}
	Under the equivalence $f_!\simeq f_*\intHom(\delta_f,-)$, the morphism 
	\begin{align*}
		\delta_f\otimes f^*f_*\intHom(\delta_f,M) \otimes f^*N \to M\otimes f^*N
	\end{align*}
	in \eqref{eq1} corresponds to the morphism
	\begin{align}\label{eq3}
		(\delta_f\circ(f\circ M))\circ N \simeq ((\delta_f\circ f)\circ M)\circ N \overset{\varepsilon}{\to} (\id \circ M)\circ N \simeq M\circ N.
	\end{align}
	Moreover the equivalence 
	$$\delta_f\otimes f^*(f_*\intHom(\delta_f,M) \otimes N) \simeq \delta_f\otimes f^*f_*\intHom(\delta_f,M) \otimes f^*N$$ 
	which comes from the symmetric monoidal structure on $f^*$ corresponds to the associator
	\begin{align}\label{eq4}
		\delta_f\circ((f\circ M)\circ N)\simeq (\delta_f\circ(f\circ M))\circ N.
	\end{align}
	By composing \eqref{eq3} with \eqref{eq4}, we get a morphism  $\delta_f\circ((f\circ M)\circ N) \to M\circ N$ and it induces a morphism $\beta \colon (f\circ M)\circ N \to f\circ (M\circ N)$ by the adjunction. 
	From the above argument, this morphism coincides with the morphism $\rho \colon f_!M\otimes N\to f_!(M\otimes f^*N).$
	Therefore, it is enough to show that $\alpha$ coincides with $\beta$.
	Let $\eta \colon \id \to f\circ \delta_f$ denote the unit of the adjunction $(\delta_f,f)$ in the 2-category $\Kcal_{\Dcal,S}$.
	Then $\beta$ can be written as follows:
	\begin{align*}
	(f\circ M)\circ N \simeq &\id \circ ((f\circ M)\circ N) \\
	\overset{\eta}{\to} &(f\circ \delta_f) \circ ((f\circ M)\circ N) \\
	\simeq &f\circ (\delta_f \circ ((f\circ M)\circ N)) \\
	\simeq &f\circ ((\delta_f\circ(f\circ M))\circ N)\\
	\simeq &f\circ (((\delta_f\circ f)\circ M)\circ N)\\
	\overset{\varepsilon}{\to} &f\circ((\id \circ M)\circ N)\\
	\simeq &f\circ(M\circ N).
    \end{align*}
	By the coherence condition and the fact that the morphism 
	\begin{align*}
	f\simeq \id \circ f \overset{\eta}{\to} (f\circ \delta_f) \circ f\simeq f\circ (\delta_f \circ f)\overset{\varepsilon}{\to} f\circ \id \simeq f
    \end{align*}
	is the identity, we find that $\beta$ coincides with $\alpha$.
\end{proof}

\begin{corollary}\label{cor:projection formula proper}
	Let $f\colon X\to S$ be a weakly $\Dcal$-proper morphism, and we fix a trivialization $\delta_f\simeq \mathbf{1}_X$ of the codualizing complex $\delta_f$, which induces an equivalence $f_!\simeq f_*$.
	Under this equivalence, the projection formula 
	\begin{align*}
		f_*M \otimes N \to f_*(M\otimes f^*N)
	\end{align*}
	for $M\in \Dcal(X)$ and $N\in \Dcal(S)$ coincides with the morphism induced by 
	\begin{align*}
		f^*(f_*M \otimes N) \to f^*f_*M \otimes f^*N \to M\otimes f^*N
	\end{align*}
	via the adjunction.
\end{corollary}

We assume that $f\colon X\to S$ is weakly $\Dcal$-proper, and we fix a trivialization $\delta_f\simeq \mathbf{1}_X$ of the codualizing complex $\delta_f$, which induces an equivalence $f_!\simeq f_*$.
For $M\in \Dcal(X)$ and $N\in \Dcal(S)$, there is an equivalence
\begin{align*}
	\pi\colon f_*\intHom(M,f^!N)\simeq \intHom(f_!M,N),
\end{align*}
which is the usual one in the theory of 6-functor formalisms.
On the other hand, there is a morphism
\begin{align*}
	f_*\intHom(M,f^!N)\otimes f_*M \to &f_*(\intHom(M,f^!N)\otimes M)
	\to f_*f^!N \simeq f_!f^!N\to N,
\end{align*}
where we note that $f_*$ is lax symmetric monoidal functor since it is a right adjoint functor of the symmetric monoidal functor $f^*$.
From this, we get a morphism 
\begin{align*}
	\rho\colon f_*\intHom(M,f^!N)\to \intHom(f_*M,N)\simeq \intHom(f_!M,N)
\end{align*}
by the adjunction.

\begin{proposition}\label{prop:Poincare duality proper}
	The above two morphisms $\pi$ and $\rho$ coincide in the homotopy category of $\Dcal(S)$.
\end{proposition}
\begin{proof}
	For $L\in \Dcal(S)$, there are equivalences
	\begin{equation}
	\begin{aligned}\label{eq5}
		\Hom(L,f_*\intHom(M,f^!N))\simeq &\Hom(f^*L,\intHom(M,f^!N))\\
		\simeq &\Hom(M\otimes f^*L,f^!N)\\
		\simeq &\Hom(f_!(M\otimes f^*L),N)\\
		\simeq &\Hom(f_!M\otimes L,N)\\
		\simeq &\Hom(L,\intHom(f_!M,N)),
	\end{aligned}
    \end{equation}
	where the fourth equivalence follows from the projection formula. 
	This equivalence induces the equivalence $\pi\colon f_*\intHom(M,f^!N)\simeq \intHom(f_!M,N)$ by the Yoneda Lemma.
	On the other hand, for $L\in \Dcal(X)$, there are equivalences and a morphism
	\begin{equation}
	\begin{aligned}\label{eq6}
		\Hom(L,\intHom(M,f^!N))\simeq &\Hom(M\otimes L,f^!N)\\
		\simeq & \Hom(f_!(M\otimes L),N)\\
		\simeq & \Hom(f_*(M\otimes L),N)\\
		\to &\Hom(f_*M\otimes f_*L,N)\\
		\simeq &\Hom(f_!M\otimes f_*L,N)\\
		\simeq &\Hom(f_*L,\intHom(f_!M,N)).
	\end{aligned}
    \end{equation}
	If $L=\intHom(M,f^!N)$, then the image of $\id_L$ under the above morphism is $$\rho\colon f_*\intHom(M,f^!N)\to \intHom(f_!M,N).$$
	Therefore, it is enough to show that the morphism
	\begin{align*}
		\Hom(L,\intHom(M,f^!N)) \to &\Hom(f_*L,f_*\intHom(M,f^!N))\\ 
		\simeq &\Hom(f_*L,\intHom(f_!M,N)),
	\end{align*}
	where the last equivalence is \eqref{eq5}, coincides with \eqref{eq6}.
	By unwinding the construction of the above morphisms, it reduces to showing that a morphism
	\begin{align*}
		f_!M\otimes f_*L \simeq f_*M\otimes f_*L\to f_*(M\otimes L)\simeq f_!(M\otimes L)
	\end{align*}
	coincides with a morphism
	\begin{align*}
		f_!M\otimes f_*L \simeq f_!(M\otimes f^*f_*L) \to f_!(M\otimes L).
	\end{align*}
	It easily follows from Corollary \ref{cor:projection formula proper}.
\end{proof}

\begin{corollary}\label{cor:Poincare duality proper}
	For $M\in \Dcal(X)$, the Poincar\'{e} duality 
	\begin{align*}
		f_*\intHom(M,f^!\mathbf{1}_S)\simeq \intHom(f_!M,\mathbf{1}_S) \simeq \intHom(f_*M,\mathbf{1}_S)
	\end{align*}
	is induced by
	\begin{align*}
	f_*\intHom(M,f^!\mathbf{1}_S)\otimes f_*M \to &f_*(\intHom(M,f^!\mathbf{1}_S)\otimes M)
	\to f_*f^!\mathbf{1}_S \simeq f_!f^!\mathbf{1}_S\to \mathbf{1}_S.
	\end{align*}
\end{corollary}

\begin{corollary}\label{cor:counit proper}
	For $N\in \Dcal(S)$, the following diagram commutes:
	\begin{align*}
		\xymatrix{
		f_!f^!N\ar[d]\ar[r]^-{\simeq} &f_*\intHom(\mathbf{1}_X,f^!N)\ar[d]^-{\simeq} \\
		N\ar[d]^-{\simeq} & \intHom(f_!\mathbf{1}_X,N)\ar[d]^-{\simeq} \\
		\intHom(\mathbf{1}_S,N) &\intHom(f_*\mathbf{1}_X,N),\ar[l]
	}
	\end{align*}
	where the lower horizontal morphism is induced by the unit $\mathbf{1}_S\to f_*\mathbf{1}_X$.
\end{corollary}
\begin{proof}
	By Proposition \ref{prop:Poincare duality proper}, the morphism
	\begin{align*}
		\alpha\colon f_*\intHom(\mathbf{1}_X,f^!N)\otimes f_*\mathbf{1}_X &\to f_*(\intHom(\mathbf{1}_X,f^!N)\otimes \mathbf{1}_X) \\
		&\simeq f_*\intHom(\mathbf{1}_X,f^!N)\\
		&\simeq f_!f^!N\\
		&\to N
	\end{align*} 
	coincides with the morphism
	\begin{align*}
		\beta \colon f_*\intHom(\mathbf{1}_X,f^!N)\otimes f_*\mathbf{1}_X &\to \intHom(f_!\mathbf{1}_X,N) \otimes f_*\mathbf{1}_X \\
		&\simeq \intHom(f_*\mathbf{1}_X,N) \otimes f_*\mathbf{1}_X\\
		&\to N.
	\end{align*}
	Let $\gamma \colon f_*\intHom(\mathbf{1}_X,f^!N) \to f_*\intHom(\mathbf{1}_X,f^!N)\otimes f_*\mathbf{1}_X$ be the morphism induced by the unit $\mathbf{1}_S\to f_*\mathbf{1}_X$.
	Then the composition $\alpha\circ \gamma$ is equal to the morphism
	\begin{align*}
		f_*\intHom(\mathbf{1}_X,f^!N) \simeq f_!f^!N \to N,
	\end{align*}
	and the composition $\beta\circ \gamma$ is equal to the morphism
	\begin{align*}
		f_*\intHom(\mathbf{1}_X,f^!N) \simeq \intHom(f_!\mathbf{1}_X,N) \simeq \intHom(f_*\mathbf{1}_X,N) \to \intHom(\mathbf{1}_S,N) \simeq N.
	\end{align*}
	The equality $\alpha\circ \gamma = \beta\circ \gamma$ implies the claim.
\end{proof}
%%%%%%%%%%%%%%%%%%%%%%%%%%%%%%%%%%%%%%%%%%%%%%%%%%%%%%%%%%%%%%%%%%%%%%%%%%%%%%%%%%%%%%%%%%

%\bibliographystyle{/Users/yutaromikami/Documents/bib/my_amsalpha}
%\bibliography{/Users/yutaromikami/Documents/bib/bibliography}
\bibliography{cohomologically-smooth}
\end{document}